\documentclass[10pt,twoside, a4paper, english, reqno]{amsart}
\usepackage[dvips]{epsfig}
\usepackage{amscd}
\usepackage{amssymb}
\usepackage{amsthm}
\usepackage{amsmath}
\usepackage{latexsym}

\usepackage{upref}
\usepackage{hyperref}
\usepackage{color}
\usepackage[usenames,dvipsnames]{xcolor}
\theoremstyle{plain}
\newtheorem{thm}{Theorem}[section]
\theoremstyle{plain}
\newtheorem{lem}[thm]{Lemma}
\newtheorem{prop}[thm]{Proposition}

\theoremstyle{definition}
\newtheorem{defi}{Definition}[section]
\newtheorem{rem}{Remark}[section]

\newtheorem*{maintheorem*}{Main Theorem}

\newenvironment{Assumptions}
{
\setcounter{enumi}{0}

\begin{enumerate}}
{\end{enumerate} }

\newcommand{\R}{\ensuremath{\mathbb{R}}}

\newcommand{\rd}{\ensuremath{\mathbb{R}^d}}
\newcommand{\supp}{\ensuremath{\mathrm{supp}\,}}
\newcommand{\goto}{\ensuremath{\rightarrow}}
\newcommand{\grad}{\ensuremath{\nabla}}
\newcommand{\eps}{\ensuremath{\varepsilon}}
\newcommand{\T}{\ensuremath{\mathcal{T}}}
\newcommand{\E}{\ensuremath{\mathcal{E}}}

\numberwithin{equation}{section} \allowdisplaybreaks

\title[Flux-splitting finite volume scheme ]
{ Convergence of a flux-splitting finite volume scheme for conservation laws driven by L\'{e}vy noise.}

\date{}



\author[Ananta K. Majee]{Ananta K. Majee}
\address[Ananta K. Majee]{\newline
Mathematisches Institut,
Universit\"{a}t T\"{u}bingen,
Auf der Morgenstelle 10,
D-72076 T\"{u}bingen, Germany. }
\email[]{majee@na.uni-tuebingen.de}

\thanks{}

\begin{document}
\begin{abstract}
We explore numerical approximation of multidimensional stochastic balance laws driven by multiplicative L\'{e}vy noise via flux-
splitting finite volume method. The convergence of the approximations is proved towards the unique entropy solution of the underlying problem.
\end{abstract}

\maketitle

\section{Introduction}
 Let $\big(\Omega, \mathbb{P}, \mathcal{F}, \{\mathcal{F}_t\}_{t\ge 0} \big)$ be a filtered probability space satisfying 
 the usual hypothesis i.e  $\{\mathcal{F}_t\}_{t\ge 0}$ is a right-continuous filtration such that $\mathcal{F}_0$ 
 contains all the $\mathbb{P}$-null subsets of $(\Omega, \mathcal{F})$. In this paper, we are interested in the study of numerical scheme and 
 numerical approximation for multi-dimensional nonlinear stochastic balance laws of type
  
 \begin{align}
 du(t,x) + \mbox{div}_x\big(\vec{v}(t,x) f(u(t,x))\big)\,dt&=\int_{\mathbf{E}} \eta(u(t,x);z)\,\tilde{N}(dz,dt),\quad (t,x)\in \Pi_T,\label{eq:levy_con_laws} \\
 u(0,x) &= u_0(x),\quad x\in \R^d, \notag 
\end{align}
where $\Pi_T = [0,T) \times \R^d$ with $T>0$ fixed.  Here, $ f:\R\rightarrow \R$ is a given real valued flux function, $\vec{v}$ is a given 
 vector valued function, $u_0(x) $ is a given initial function and $\tilde{N}(dz,dt)= N(dz,dt)-m(dz)\,dt $, 
where $N$ is a Poisson random measure on $(\mathbf{E},\mathcal{E})$ with intensity measure $m(dz)$, where $(\mathbf{E},\mathcal{E},m)$ is a $\sigma$-finite measure space.
Furthermore, $(u,z)\mapsto \eta(u,z)$ is a given real valued functions signifying
the multiplicative nature of the noise.
\vspace{.1cm}

This type of equation arises in many
different fields where non-Gaussianity plays an important role. As for example, it has been used in models of neuronal activity accounting for
synaptic transmissions occurring randomly in time as well as at different locations on a spatially extended
neuron, chemicals reaction-diffusion systems, market fluctuations both for risk management and option
pricing purpose, stochastic turbulence, etc. The study of well-posedness theory for this kind of equation
is of great importance in the light of current applications in continuum physics.
\begin{rem}
We will carry out our analysis under the structural assumption $\mathbf{E} = \mathcal{O}\times \R^*$ where $\mathcal{O}$ is a 
 subset of the Euclidean space. The measure $m$ on $\mathbf{E}$ is defined as $\gamma \times \mu^*$ where $\gamma$ is a Radon measure on $\mathcal{O}$
 and $\mu^*$ is so-called L\'{e}vy measure on $\R^*$. Such a noise would be called an impulsive white noise with jump position intensity $\gamma$ and jump size intensity $\mu^*$.
 We refer to \cite{peszat} for more on L\'{e}vy sheet and related impulsive white noise. 
\end{rem}

In the case $\eta=0$, the equation \eqref{eq:levy_con_laws} becomes a standard conservation law in $\R^d$ 
and there exists a satisfactory well-posedness theory based on Kruzkov's pioneering idea to pick up the 
physically relevant solution in an unique way, called {\em entropy solution}. We refer to \cite{godu,kruzkov,malek,volpert} and  
references therein for more on entropy solution theory for deterministic conservation laws.

The study of stochastic balance laws driven by noise is comparatively new area of pursuit. Only recently balance laws with stochastic forcing have
attracted the attention of many authors \cite{BaVaWit,Vallet2014,BaVaWit_2015,BisMaj,BisKarlMaj,
BisKoleyMaj, BisMajVal,Chen:2012fk, Vovelle2010,martina,xu,nualart:2008,Kim2003,KoleyMajVal} and resulted a significant momentum in the theoretical development of such problems. 
Due to nonlinear nature of the underlying problem, explicit solution formula is hard to obtain and hence
robust numerical schemes for approximating such equation are very important. In the last decade, there has been a growing interest in numerical
approximation and numerical experiments for entropy solution to the related Cauchy problem driven by stochastic forcing.
The first documented development in this direction is \cite{risebroholden1997}, where
the authors established existence of weak solution (possibly non-unique) of one dimensional balance law driven by Brownian noise via 
splitting method. In a recent paper \cite{kroker}, Kr\"{o}ker and  Rodhe established
the convergence of monotone semi-discrete finite volume scheme by using stochastic compensated compactness method. Bauzet~\cite{bauzet} revisited the paper of Holden and Risebro \cite{risebroholden1997}, and generalized the operator-splitting method  
for the same Cauchy problem but in a bounded domain of $\R^d$. Using Young measure theory, the author established the convergence of 
approximate
solutions to an entropy solution. We also refer to see \cite{KarlStor}, where the time splitting method was analyzed for more general noise
coefficient in the spirit of Malliavin calculus
and Young measure theory.
In a recent papers \cite{bauzet-fluxsplitting, bauzet-finvolume}, Bauzet et.~ al.~ have studied fully discrete
scheme via flux-splitting and monotone finite volume schemes for stochastic conservation laws driven by multiplicative
Brownian noise and established its convergence by using Young measure technique. 

 The study of numerical schemes for stochastic balance laws driven by L\'{e}vy noise is more sparse than the previous case. 
 A semi-discrete finite difference scheme for conservation laws driven by a homogeneous multiplicative L\'{e}vy noise has been studied 
by Koley et al.\cite{KoleyMajVal-rateofconvergence}. Using BV estimates, the authors showed the convergence of approximate solutions, generated by
the finite difference scheme, to the unique entropy solution as the spatial mesh size $\Delta x \goto 0$ and established rate of convergence
which is of order $\frac{1}{2}$.  

 The above discussions clearly highlight the lack of the study of fully discrete scheme and its convergence for stochastic balance laws driven by 
L\'{e}vy noise. In this paper, drawing primary motivation from \cite{bauzet-fluxsplitting}, 
we propose a fully discrete flux-splitting finite volume scheme for \eqref{eq:levy_con_laws}, and address the convergence of the scheme. First we establish few
essential {\it a priori} estimates 
for approximate solutions and then using these estimates, we deduce entropy inequality for approximate solutions. Using Young measure theory, we conclude that the finite volume approximate
solutions tend to a generalized entropy solution of \eqref{eq:levy_con_laws}.

The rest of the paper is organized as follows. In Sections \ref{sec:technical} and \ref{sec:scheme}, we collect all the assumptions for
the subsequent analysis, then we define the numerical scheme and finally state the main result of this article. Section \ref{sec:apriori-estimate}
deals with few a priori estimates on the finite volume approximate solutions and using these {\it a priori} estimates, in Section 
 \ref{sec:entropy-inequality}, we establish discrete and continuous version of entropy inequalities on approximate solutions. The 
 Section \ref{sec:proof-maintheorem} is devoted to the proof of the main theorem along with short discussion of Young measure theory and 
 its compactness, is presented in Appendix \ref{sec:appendix}.

\section{Preliminaries and technical framework}\label{sec:technical}
It is well-known that due to nonlinear flux term in \eqref{eq:levy_con_laws}, solutions to \eqref{eq:levy_con_laws} are 
not necessarily smooth even if initial data is smooth, and hence must be interpreted via weak sense. 
Before introducing the concept of weak solutions, we first assume that 
$\big(\Omega, \mathbb{P}, \mathcal{F}, \{\mathcal{F}_t\}_{t\ge 0} \big)$ be a filtered probability space satisfying 
the usual hypothesis, i.e.,  $\{\mathcal{F}_t\}_{t\ge 0}$ is a right-continuous filtration such that $\mathcal{F}_0$ 
contains all the $\mathbb{P}$-null subsets of $(\Omega, \mathcal{F})$. Moreover,
by a predictable $\sigma$-field on $[0,T]\times \Omega$, denoted
by $\mathcal{P}_T$, we mean that the 
$\sigma$-field is generated by the sets of the form: $ \{0\}\times A$ and $(s,t]\times B$  for any $A \in \mathcal{F}_0; B
\in \mathcal{F}_s,\,\, 0<s,t \le T$.
The notion of stochastic weak solution is defined as follows:
\begin{defi} [Weak solution] 
 A square integrable $ L^2(\R^d)$-valued $\{\mathcal{F}_t: t\geq 0 \}$-predictable stochastic process $u(t)= u(t,x)$ is
called a stochastic weak solution of \eqref{eq:levy_con_laws} if for all test functions $\psi \in C_c^{\infty}( [0,T)\times \R^d)$,
\begin{align*}
\int_{\R^d} \psi(0,x) u_0(x)\,dx & + \int_{\Pi_T} \Big\{ \partial_t \psi(t,x) u(t,x)
+  \vec{v}(t,x) f(u(t,x)) \cdot \nabla_x \psi(t,x) \Big\}\,dt\,dx \notag \\ 
& + \int_{\Pi_T} \int_{\mathbf{E}} \eta( u(t,x);z) \psi(t,x)\,\tilde{N}(dz,dt)\,dx = 0, \quad \mathbb{P}\text{-a.s}.
\end{align*}
\end{defi} 
However, since there are infinitely many weak solutions, one needs to define an extra admissibility criteria to
select physically relevant solution in a unique way, and one such condition is called entropy condition. Let us begin with the notion of
entropy flux pair. 
\begin{defi}[Entropy flux pair]
$(\beta,\phi) $ is called an entropy flux pair if $ \beta \in C^2(\R) $ and $\zeta:\R \mapsto \R $ is such that
\[\zeta^{\prime}(r) = \beta'(r)f'(r). \]
 An entropy flux pair $(\beta,\zeta)$ is called convex if $ \beta^{\prime\prime}(s) \ge 0$.  
\end{defi} 
Let $\mathcal{A}= \big\{ \beta \in C^2(\R), \text{convex such that}\,\,\text{support of}\,\, \beta^{\prime \prime}\,\,\text{is compact} \big\} $.
 In the sequel, we will use specific entropy flux pairs. For any $a\in \R$
 and $\beta \in \mathcal{A}$, define $\displaystyle F^\beta(a)=\int_0^{a}\beta^{\prime}(s)f^\prime(s)\,ds$ . Note that, $F^\beta(\cdot)$ is a Lipschitz continuous function
 on $\R$ and $(\beta,F^\beta)$ is an entropy flux-pair. 
To this end, we define the notion of stochastic entropy solution of \eqref{eq:levy_con_laws}.
\begin{defi} [Stochastic entropy solution]\label{defi:stoc-entropy}
An $L^2(\R^d)$-valued $\{\mathcal{F}_t: t\geq 0 \}$-predictable stochastic process $u(t)= u(t,x)$ is
called a stochastic entropy solution of \eqref{eq:levy_con_laws} if the following hold:
\begin{itemize}
 \item[i)] For each $ T>0$  \[\sup_{0\leq t\leq T} \mathbb{E}\Big[||u(\cdot,t)||_{2}^{2}\Big] < + \infty. \] 
 \item[ii)] For each  $0\leq \psi\in C_{c}^\infty( [0,\infty)\times \R^d)$ and  $\beta \in \mathcal{A}$, there holds 
\begin{align*}
 &\int_{\R^d} \psi(x,0)\beta(u_0(x))\,dx + \int_{\Pi_T} \Big\{ \partial_t \psi(t,x) \beta(u(t,x))
+ F^\beta(u(t,x)) \vec{v}(t,x) \cdot \grad \psi(t,x) \Big\}\,dx\,dt \notag \\ 
  & + \int_{\Pi_T}\int_{\mathbf{E}} \int_0^1 \eta(u(t,x);z) \beta^\prime\big( u(t,x)+ \lambda \eta(u(t,x);z)\big) \psi(t,x)
  \,d\lambda\,\tilde{N}(dz,dt)\,dx  \notag\\
 & + \int_{\Pi_T}\int_{\mathbf{E}} \int_0^1 (1-\lambda) \eta^2(u(t,x);z) \beta^{\prime\prime}\big( u(t,x)+ \lambda \eta(u(t,x);z)\big) \psi(t,x)
  \,d\lambda\,m(dz)\,dt\,dx  \ge 0, \quad \mathbb{P}\text{-a.s.} 
\end{align*}
\end{itemize}
\end{defi} 

Due to nonlocal nature of the  It\^{o}-L\'{e}vy formula and the missing noise-noise interaction, the Definition~\ref{defi:stoc-entropy} 
 does not alone give the $L^1$-contraction principle in the sense of average when one tries to compare two entropy solutions directly,
 and hence fails to give uniqueness. For the details, we refer to see \cite{Chen:2012fk,nualart:2008}. However, in view of \cite{BaVaWit,BisKarlMaj},
 we can look for so called {\em generalized entropy solution} which are 
 $L^2\big(\R^d \times (0,1)\big)$-valued $\{\mathcal{F}_t: t\geq 0 \}$-predictable stochastic process.
 \begin{defi} [Generalized entropy solution]\label{defi:generality entropy solution}
An $ L^2\big(\R^d \times (0,1)\big)$-valued $\{\mathcal{F}_t: t\geq 0 \}$-predictable stochastic process $u(t)= u(t,x,\alpha)$ is
called a generalized stochastic entropy solution of \eqref{eq:levy_con_laws} provided \\
(1) For each $ T>0$  \[\sup_{0\leq t\leq T} \mathbb{E}\Big[||u(t,\cdot,\cdot)||_{2}^{2}\Big] < + \infty. \] 
(2) For all test functions $0\leq \psi\in C_{c}^{1,2}( [0,\infty)\times \R^d)$, and any $\beta \in \mathcal{A}$, the following inequality holds
\begin{align}
 &\int_{\R^d} \beta(u_0(x))\psi(x,0)\,dx + \int_{\Pi_T}\int_0^1 \Big\{\beta(u(t,x,\alpha))\partial_t \psi(t,x) + 
  F^\beta(u(t,x,\alpha)) \vec{v}(t,x)\cdot \grad_x \psi(t,x)\Big\}\, d\alpha\,dt\,dx \notag \\
   & + \int_{\Pi_T} \int_{\mathbf{E}} \int_0^1\int_{0}^1 \eta(u(t,x,\alpha);z) \beta^\prime \big(u(t,x,\alpha) +
   \lambda \eta(u(t,x,\alpha);z)\big) \psi(t,x)\,d\alpha\,d\lambda \,\tilde{N}(dz,dt)\,dx \notag \\
    & +  \int_{\Pi_T} \int_{\mathbf{E}}\int_0^1 \int_{0}^1 (1-\lambda) 
    \eta^2(u(t,x,\alpha);z) \beta^{\prime\prime} \big(u(t,x,\alpha) + \lambda \eta(u(t,x,\alpha);z)\big)\notag \\
    & \hspace{5cm} \times \psi(t,x)\,d\alpha\, d\lambda\,m(dz)\,dt\,dx
    \ge 0, \quad \mathbb{P}\text{-a.s.} \label{eq: generalised entropy inequality}
\end{align}
\end{defi}
The aim of this paper is to establish convergence of approximate solutions, constructed via flux-splitting finite volume scheme
 (cf.~Section \ref{sec:scheme}), to the unique entropy solution of \eqref{eq:levy_con_laws}, and we will do so under the following assumptions: 
\begin{Assumptions}
 \item \label{A1}  $f: \R \mapsto \R$ is $C^2$ and Lipschitz continuous with $f(0)=0$.
\item \label{A2} $\vec{v}:[0,T]\times \R^d \mapsto \R^d $ is a $C^1$ function with $\text{div}_x \vec{v}(t,x)=0$ for all $(t,x) \in \Pi_T$. Furthermore,
there exists $V < + \infty$ such that 
$|\vec{v}(t,x)| \le V$ for all $(t,x)\in \Pi_T$. 
 
\item \label{A3} There exist  positive constants $0<\lambda^* <1$ and $C^*>0$, and $h_1 \in L^2(\mathbf{E},m)$ with $0\le h_1(z)\le 1$ such that 
$ ~\text{for all}~  u,v \in \R;~~z\in \mathbf{E}$
 \begin{align*}
 | \eta(u;z)-\eta(v;z)|  & \leq  \lambda^* |u-v|h_1(z); \quad |\eta(u,z)|\le  C^* h_1(z).
 \end{align*}
 Moreover, $\eta(0;z)= 0$ for all $z \in \mathbf{E}$.

\item \label{A4} The initial function $u_0(x)$ is a $L^2(\R^d)$-valued $\mathcal{F}_0$ measurable random variable satisfying 
$$\mathbb{E}\Big[ ||u_0||_2^2\Big] < + \infty.$$
\end{Assumptions}
We have the following existence and uniqueness 
theorems whose proofs are postponed to the Appendix.
\begin{thm}\label{thm:existence}
 Let the assumptions \ref{A1}-\ref{A4} hold. Then there exists a generalized entropy solution 
of \eqref{eq:levy_con_laws} in the sense of Definition \ref{defi:generality entropy solution}.
\end{thm}
\begin{thm}\label{thm:uniqueness}
Under the assumptions \ref{A1}-\ref{A4}, the generalized entropy solution of \eqref{eq:levy_con_laws} is unique. Moreover, it is
the unique stochastic entropy solution.
\end{thm}

\begin{rem}
Note that we need the assumption \ref{A1} to get entropy solution for the initial data in 
$L^2(\R^d)$ to control the multi-linear integrals terms. The assumption \ref{A3} is needed to handle the nonlocal nature of the entropy inequalities. Boundedness of $\eta$ is needed to validate Proposition \ref{prop:discrete entropy ineq}.
\end{rem}

Throughout this paper, we use the letter $C$  to denote various generic constant which may change line to line. We denote by $c_f$ 
 the Lipschitz constant of $f$ and $c_{\eta}$, the finite constant (which exists thanks to \ref{A3}) as $ \displaystyle c_{\eta}= \int_{\mathbf{E}} h^2_1(z)\,m(dz)$.  The Euclidean norm on $\R^d$ is denoted by 
 $|\cdot|$.
 
\section{Flux-splitting finite volume scheme}\label{sec:scheme}
Our main point of interest is numerical approximation for the problem \eqref{eq:levy_con_laws}. Let us first introduce the space 
discretization by finite volumes (control volumes). For that we need to recall the definition of so called  admissible meshes for finite 
volume scheme (cf. \cite{gallouet00}).

\begin{defi}[Admissible mesh] \label{def:admissible mesh}
An admissible mesh $\mathcal{T}$ of $\R^d$  is
 a family of disjoint polygonal connected subset of $\R^d$ satisfying the following:
 \begin{itemize}
  \item [i)] $\R^d$ is the union of the closure of the elements (called control volume) of $\mathcal{T}$.
\item[ii)] The common interface of any
two elements of $\mathcal{T}$  is included in a hyperplane of $\R^d$.
\item[iii)] There exists a nonnegative constant $\alpha$ such that 
\begin{align}\label{cond:admissible mesh}
\begin{cases}
 \alpha h^d \le |K|, \\
 \displaystyle |\partial K | \le  \frac{1}{\alpha}h^{d-1}, \quad \forall K\in \mathcal{T},
 \end{cases}
\end{align}
where $h= \sup\Big\{ \text{diam} (K): K\in \mathcal{T}\Big\} < + \infty$, $|K|$ denotes the $d$-dimensional Lebesgue measure of $K$, and 
$|\partial K |$ represents the $(d-1)$-dimensional Lebesgue measure of $\partial K $.
 \end{itemize}
\end{defi}

In the sequel, we denote the followings:
\begin{itemize}
 \item  $\mathcal{E}_K$: the set of interfaces of the control volume $K$.
 \item $\mathcal{N}(K)$: the set of control volumes neighbors of the control volume $K$.
 \item $K|L$: the common interface between $K$ and $L$ for any $L\in \mathcal{N}(K)$.
 \item $\mathcal{E}$: the set of all the interfaces of the mesh.
 \item $n_{K,\,\sigma}$: the unit normal to the interface $\sigma$, outward to the control volume $K$, for any $\sigma \in \mathcal{E}_K$.
\end{itemize}
Consider an admissible mesh $\T$ in the sense of Definition \ref{def:admissible mesh}. In order to discretize the time variable, we split the time interval $[0,T]$  as follows: Let $N$ be a positive integer and we set $\Delta t= \frac{T}{N}$.
 Define $t_n = n \Delta t,\,\, n=0,1,\cdots,N$. Then $\{t_n:\, n=0,1,\cdots,N\}$ splits the time interval $[0,T]$
 into equal step with a length equal to $\Delta t$.
 
 It is well known that, the main idea behind flux-splitting finite volume method is to express a flux function  $f$ as the sum of a
   nondecreasing function $f_1$ and a non increasing function $f_2$. Since the flux function $f$ is Lipschitz continuous such a decomposition 
   is always possible. 
   
 We propose the following flux-splitting finite volume scheme to approximate the solution of \eqref{eq:levy_con_laws}:
   for any $K\in \T$, and $n\in \{0,1,2,\cdots, N-1\}$, we define the discrete unknowns $u_K^n$ as follows
\begin{align}
    u_K^0 & =\frac{1}{|K|} \int_{K} u_0(x)\,dx,\notag \\
     \frac{|K|}{\Delta t} \big(u_{K}^{n+1}-u_K^n \big)  + \sum_{\underset{\sigma=K|L} {\sigma\in \E_K}}&  |\sigma|  \Big\{ 
    (\vec{v}\cdot n_{K,\sigma})^+ \big( f_1(u_K^n) +f_2(u_L^n)\big)- (\vec{v}\cdot n_{K,\sigma})^- \big( f_1(u_L^n) +f_2(u_K^n)\big) \Big\} \notag  \\
   &= \frac{|K|}{\Delta t} \int_{t_n}^{t_{n+1}} \int_{\mathbf{E}} \eta(u_K^n;z)\tilde{N}(dz,dt),\label{sceme:fin volume} 
\end{align}
where, by denoting $d\nu$ the $d-1$ dimensional Lebesgue measure
\begin{align*}
  (\vec{v}\cdot n_{K,\sigma})^+ & = \frac{1}{\Delta t |\sigma|}\int_{t_n}^{t_{n+1}}\int_{\sigma} \big(\vec{v}(t,x)\cdot n_{K,\sigma}\big)^+ \,d\nu(x)\,dt, \\
(\vec{v}\cdot n_{K,\sigma})^- & = \frac{1}{\Delta t |\sigma|}\int_{t_n}^{t_{n+1}}\int_{\sigma} \big(\vec{v}(t,x)\cdot n_{K,\sigma}\big)^-\,d\nu(x)\,dt. 
\end{align*}
Since $\text{div}_x\vec{v}(t,x)=0$ for any $(t,x) \in \Pi_T$, an elementary estimate yields 
\begin{align}
\sum_{\sigma\in \E_K}|\sigma| \vec{v}\cdot n_{K,\sigma}=0.\label{sum:zero-flux}
\end{align}
We define approximate finite volume solution on $\Pi_T$ as a piecewise constant given by
\begin{align}
 u_{\T,\Delta t}^h(t,x)= u_K^n\,\, \text{for}\,\, x\in K \,\,\text{and}\,\, t\in [t_n, t_{n+1}).\label{eq:fin-approx}
\end{align} 
\begin{rem}
In view of the properties of stochastic integral with respect to compensated Poisson random measure, each $u_K^n$ is $\mathcal{F}_{n \Delta t}$
- measurable for $K\in \mathcal{T}$ and $n\in \{0,1,\cdots, N \}$. Thus, $u_{\T,\Delta t}^h(t,\cdot)$ is an $L^2(\R^d)$-valued $\mathcal{F}_t$-
predictable stochastic process as $u_0$ satisfies \ref{A4}. 
\end{rem}
Finally, we state the main theorem of this paper.
\begin{maintheorem*}\label{thm:maintheorem}
 Let the assumptions \ref{A1}-\ref{A4} be true and $\T$ be an admissible mesh on $\R^d$  with size $h$ 
 in the sense of Definition \ref{def:admissible mesh}.
 Let $\Delta t$ be the time step as discuss above and assume that 
 \begin{align*}
  \frac{\Delta t}{h} \goto 0 \,\,\text{as}\,\, h \goto 0.
 \end{align*}
Let $u_{\T,\Delta t}^h(t,x)$ be the finite volume approximation as prescribed by
 \eqref{eq:fin-approx}. Then, there exists a $L^2(\R^d\times (0,1))$-valued $\{\mathcal{F}_t:t\ge 0\}$-predictable process
 $u=u(t,x,\alpha)$ such that 
 \begin{itemize}
  \item[i)] $ u(t,x,\alpha)$ is a generalized entropy solution of \eqref{eq:levy_con_laws} and $ u_{\T,\Delta t}^h(t,x)\goto u(t,x,\alpha)$ in
  the sense of Young measure.
  \item[(ii)] $u_{\T,\Delta t}^h(t,x)\goto \bar{u}(t,x)$ in $L^p_{\text{loc}}(\R^d; L^p(\Omega \times (0,T)))$ for $1\le p < 2$, where $\bar{u}(t,x)=\int_{0}^1 u(t,x,\alpha)\,d\alpha$ is the 
  unique stochastic entropy solution of \eqref{eq:levy_con_laws}.
 \end{itemize}
\end{maintheorem*}
\begin{rem}
Under the CFL condition 
\begin{align}
\Delta t \le \frac{(1-\xi)\alpha^2 h}{c_f V}, \,\,\text{for some}\,\,\xi \in (0,1),\label{cond:cfl}
\end{align}
we have uniform moment estimate  and  weak $BV$ estimate  on $u_{\T,\Delta t}^h$  for $\xi=0$ and $\xi \in (0,1)$ respectively 
(see Lemmas \ref{lem:moment estimate} and \ref{lem:weak bv}). In the deterministic case, condition \eqref{cond:cfl} is sufficient to establish 
the convergence of approximate solutions to the unique entropy solution of the problem. But in the stochastic case, only this condition 
is not enough and hence, we assume the stronger condition, namely $\frac{\Delta t}{h} \goto 0 $ as $h \goto 0$.
\end{rem}
\begin{rem}
Since every Lipschitz continuous function can be expressed as the sum of nondecreasing function and a non increasing one, it suffices to prove the main theorem (cf.~Theorem \ref{thm:maintheorem})
for a nondecreasing Lipschitz continuous flux function $f$.
\end{rem}
For a nondecreasing Lipschitz continuous function $f$, the finite volume scheme \eqref{sceme:fin volume} reduces to an upwind finite volume scheme
\begin{align}\label{scheme:upwind fin volume}
\begin{cases}
    \displaystyle \frac{|K|}{\Delta t} \big(u_{K}^{n+1}-u_K^n \big)  + \underset{ \sigma\in \E_K}\sum |\sigma| \vec{v}\cdot n_{K,\sigma} f(u_\sigma^n)
    = \frac{|K|}{\Delta t} \int_{t_n}^{t_{n+1}} \int_{\mathbf{E}} \eta(u_K^n;z)\tilde{N}(dz,dt),\\
    u_K^0= \displaystyle\frac{1}{|K|} \int_{K} u_0(x)\,dx,
   \end{cases}
\end{align}
where $u_\sigma^n$ represents the upstream value at time $t_n$ with respect to $\sigma$. More precisely, if $\sigma$ is the interface 
between two control volumes $K$ and $L$, then 
\begin{align*}
u_\sigma^n=
 \begin{cases}
  u_K^n \,\,\text{if}\,\,\vec{v}\cdot n_{K,\sigma} \ge 0,\\
  u_L^n \,\,\text{if}\,\,\vec{v}\cdot n_{K,\sigma} < 0.
 \end{cases}
\end{align*}
The upwind finite volume approximate solution $u_{\T,\Delta t}^h(t,x)$ on $\Pi_T$ is defined as  
\begin{align}
 u_{\T,\,\Delta t}^h(t,x)= u_K^n\,\, \text{for}\,\, x\in K \,\,\text{and}\,\, t\in [t_n, t_{n+1}), \label{eq:upwind fin-approx}
\end{align}
where the discrete unknown $u_K^n,
K\in \T, n\in \{0,1,\cdots,N-1\}$ is computed from \eqref{scheme:upwind fin volume}.

\section{A priori estimates} \label{sec:apriori-estimate}
This section is devoted to {\it a priori} estimates for the upwind finite volume approximate solution $u_{\T,\Delta t}^h$ which will be 
very useful to prove its convergence. We start with the following  lemma which is essentially a uniform moment estimate.
\begin{lem} \label{lem:moment estimate}
 Let $T>0$ and the assumptions \ref{A1}-\ref{A4} hold. Consider an
  admissible mesh $\T$ on $\R^d$ with size $h$ in the sense of Definition \ref{def:admissible mesh}.
  Let $\Delta t=\frac{T}{N}$ be the time step for some $N\in \mathbb{N}^*$, satisfying the Courant-Friedrichs-Levy (CFL) condition 
  \begin{align*}
   \Delta t \le \frac{\alpha^2 h}{c_f V}.
  \end{align*}
  Then the upwind finite volume approximate solution $u_{\T,\,\Delta t}^h$ satisfies the following bound
  \begin{align}
   \| u_{\T,\,\Delta t}^h\|_{L^\infty(0,T; L^2(\Omega\times \R^d))}^2 \le e^{c_{\eta}T} \mathbb{E}\big[||u_0||_2^2\big].
   \label{estimate: moment-1}
  \end{align}
  As a consequence, we see that $u_{\T,\,\Delta t}^h$ satisfies the following bound
   \begin{align*}
   \| u_{\T,\,\Delta t}^h\|_{ L^2(\Omega\times \Pi_T)}^2 \le T\,e^{c_{\eta}T} \mathbb{E}\big[||u_0||_2^2\big].
  \end{align*}
\end{lem}

\begin{proof}
 To prove \eqref{estimate: moment-1}, it is enough to prove: for $n\in \{0,1,\cdots, N-1\}$, the following property holds
 \begin{align}
  \sum_{K\in\T} |K| \mathbb{E}\big[ (u_K^n)^2\big] \le ( 1+ \Delta t c_{\eta})^n \mathbb{E}\big[||u_0||_2^2\big].\label{estimate:Pn}
 \end{align}
 Observe that 
 \begin{align*}
  \sum_{K\in \T} |K| \mathbb{E}\big[ (u_K^0)^2\big]&= \sum_{K\in \T} |K| \mathbb{E}\Big[ \big(\frac{1}{|K|}\int_{K} u_0(x)\,dx\big)^2\Big]\\
  & \le \mathbb{E}\big[||u_0||_2^2\big]= ( 1+ \Delta t c_{\eta})^0 \mathbb{E}\big[||u_0||_2^2\big].
 \end{align*}
 Set $n\in \{ 0,1,\cdots, N-1\}$ and suppose that \eqref{estimate:Pn} holds for $n$. We will show that \eqref{estimate:Pn} holds for $n+1$. In view of \eqref{sum:zero-flux}, one has
 $\displaystyle \underset{\sigma\in \E_K}\sum |\sigma| \vec{v}\cdot n_{K,\sigma} f(u_K^n)=0$ and hence the scheme \eqref{scheme:upwind fin volume} reduces to
 \begin{align*}
     \frac{|K|}{\Delta t} \big(u_{K}^{n+1}-u_K^n \big)  + \sum_{ \sigma\in \E_K} |\sigma|\vec{v}\cdot n_{K,\sigma}\big(f(u_\sigma^n)-f(u_K^n)\big)
    = \frac{|K|}{\Delta t} \int_{t_n}^{t_{n+1}} \int_{\mathbf{E}} \eta(u_K^n;z)\tilde{N}(dz,dt).
\end{align*}
Again, in view of the definition of $u_\sigma^n$, the above finite volume  scheme is equivalent to 

 \begin{align}
     \frac{|K|}{\Delta t} \big(u_{K}^{n+1}-u_K^n \big)  + \sum_{\underset{\sigma=K|L} {\sigma\in \E_K}} |\sigma|(\vec{v}\cdot n_{K,\sigma})^-
     \big(f(u_K^n)-f(u_L^n)\big)
    = \frac{|K|}{\Delta t} \int_{t_n}^{t_{n+1}} \int_{\mathbf{E}} \eta(u_K^n;z)\tilde{N}(dz,dt).\label{scheme:upwind fin volume-0}
\end{align}
Multiplying \eqref{scheme:upwind fin volume-0} by $u_K^n$ and using the fact that $ab=\frac{1}{2}[(a+b)^2 -a^2-b^2]$ for any $a,b\in \R$, we obtain 
\begin{align*}
 &\frac{|K|}{2}\Big[ (u_K^{n+1})^2 -(u_{K}^n)^2\Big] = \frac{|K|}{2} \big(u_K^{n+1}-u_K^n\big)^2 - \Delta t \sum_{\underset{\sigma=K|L} {\sigma\in \E_K}} |\sigma|
 (\vec{v}\cdot n_{K,\sigma})^-
     \big(f(u_K^n)-f(u_L^n)\big)u_K^n \\
     & \hspace{4cm} + |K|  \int_{t_n}^{t_{n+1}} \int_{\mathbf{E}} \eta(u_K^n;z)u_K^n\tilde{N}(dz,dt).
\end{align*}
 Taking expectation, and  using the fact that
for any  two constants $T_1, T_2\ge 0$ with $T_1<T_2$,$$
\mathbb{E}\Big[X_{T_1}\int_{T_1}^{T_2} \int_{\mathbf{E}}
\zeta(t,z)\,\tilde{N}(dz,dt)\Big] = 0,$$ where $\zeta$ is a predictable integrand with 
$ \displaystyle \mathbb{E} \Big[\int_0^T \int_{\mathbf{E}}\zeta^2(t,z)\, m(dz)\, dt\Big] < + \infty$ 
and $X$ is an adapted process, we obtain, thanks to It\^{o} isometry
\begin{align*}
 \frac{|K|}{2} \mathbb{E}\Big[ (u_K^{n+1})^2 -(u_{K}^n)^2\Big]& = 
 \frac{|K|}{2} \mathbb{E}\Big[ \int_{t_n}^{t_{n+1}} \int_{\mathbf{E}} \eta^2(u_K^n;z)\,m(dz)\,dt\Big] \\
 & + \frac{(\Delta t)^2}{2|K|} \mathbb{E}\Big[ \Big( \sum_{\underset{\sigma=K|L} {\sigma\in \E_K}} |\sigma|(\vec{v}\cdot n_{K,\sigma})^-
     \big(f(u_K^n)-f(u_L^n)\big)\Big)^2\Big] \\
  &- \Delta t \mathbb{E}\Big[ \sum_{\underset{\sigma=K|L} {\sigma\in \E_K}} |\sigma| (\vec{v}\cdot n_{K,\sigma})^-
     \big(f(u_K^n)-f(u_L^n)\big)u_K^n\Big],
\end{align*}
where we have used  \eqref{scheme:upwind fin volume-0} to replace $u_K^{n+1}-u_K^n$.
Note that, thanks to \eqref{cond:admissible mesh}, the following inequality holds
 \begin{align*}
  \frac{|\partial K|}{|K|} \le \frac{1}{\alpha^2 h}.
 \end{align*}
Therefore, $ \displaystyle 
 \sum_{\sigma \in \E_K} |\sigma| |\vec{v}\cdot n_{K,\sigma}| \le V |\partial K|\le \frac{V}{\alpha^2 h}|K|$,
and hence 
\begin{align}
 \sum_{\sigma \in \E_K} |\sigma|(\vec{v}\cdot n_{K,\sigma})^- \le \frac{V}{\alpha^2 h}|K|.\label{estim:normal}
\end{align}
We use Cauchy-Schwartz inequality, the assumption \ref{A3} on $\eta$, and \eqref{estim:normal} to have
\begin{align}
 \frac{|K|}{2} \mathbb{E}\Big[ (u_K^{n+1})^2 -(u_{K}^n)^2\Big]
     & \le \Delta t E\Bigg[ \sum_{\underset{\sigma=K|L} {\sigma\in \E_K}} |\sigma|(\vec{v}\cdot n_{K,\sigma})^-
    \Big\{ \frac{\Delta t V}{2 \alpha^2 h} \big(f(u_K^n)-f(u_L^n)\big)^2 -\big(f(u_K^n)-f(u_L^n)\big) u_K^n \Big\}\Bigg] \notag \\
     &\hspace{3cm} + \Delta t\, c_{\eta} \frac{|K|}{2} \mathbb{E}\big[(u_K^n)^2\big] \notag \\
     &\equiv \mathbf{A} + \Delta t\, c_{\eta} \frac{|K|}{2} \mathbb{E}\big[(u_K^n)^2\big] . \label{inequality:Pn-1}
\end{align}
To estimate $\mathbf{A}$, we use \cite[Lemma $4.5$]{gallouet00} and have: for any $a,b \in \R$
\begin{align*}
  b\big(f(b)-f(a)\big)\ge   \phi(b)-\phi(a) + \frac{1}{2 c_f} \big( f(b)-f(a)\big)^2,
\end{align*}
where $ \displaystyle \phi(a)=\int_0^a sf^\prime(s)\,ds$. Note that $0\le \phi(a)\le c_f a^2$.
Thus using the CFL condition $ \displaystyle \frac{\Delta t V}{2 \alpha^2 h}\le \frac{1}{2c_f}$, we get
\begin{align}
 &\mathbf{A} \le  \mathbb{E}\Bigg[ \sum_{\underset{\sigma=K|L} {\sigma\in \E_K}}  |\sigma|(\vec{v}\cdot n_{K,\sigma})^- \big(\phi(u_L^n)-\phi(u_K^n)\big)\Bigg]. \label{inequality:Pn-2}
\end{align}
Combining \eqref{inequality:Pn-1} and \eqref{inequality:Pn-2}, we obtain the following inequality after summing over all $K\in \mathcal{T}$
\begin{align*}
 &\frac{1}{2} \sum_{K\in \T} |K| \mathbb{E}\Big[ (u_K^{n+1})^2 -(u_{K}^n)^2\Big] \notag \\
 & \le  \frac{\Delta t\, c_{\eta}}{2}  \sum_{K\in \T} |K| \mathbb{E}\big[(u_K^n)^2\big] 
 + \sum_{K\in \T}  \Delta t  \sum_{\underset{\sigma=K|L} {\sigma\in \E_K}} |\sigma|(\vec{v}\cdot n_{K,\sigma})^- \mathbb{E}\big[ \phi(u_L^n)-\phi(u_K^n)\big] \\
 & \equiv \frac{\Delta t\, c_{\eta}}{2}  \sum_{K\in \T} |K| \mathbb{E}\big[(u_K^n)^2\big]  + \mathbf{B}.
\end{align*}
Since $\text{div}_x \vec{v}(t,x)=0$ for any $(t,x)\in \Pi_T$, one can show that $\mathbf{B}=0$, yielding  
\begin{align*}
  \sum_{K\in \T} |K| \mathbb{E}\big[ (u_K^{n+1})^2 \big]
 & \le \big( 1+ \Delta t\, c_{\eta}\big)  \sum_{K\in \T} |K| \mathbb{E}\big[(u_K^n)^2\big] \le \big( 1+ \Delta t\, c_{\eta}\big)^{n+1}  \mathbb{E}\big[||u_0||_2^2].
 \end{align*}
 Thus \eqref{estimate:Pn} holds by mathematical induction. In other words,
 \eqref{estimate: moment-1} holds as well. As a consequence, we have
 \begin{align*}
  \| u_{\T,\Delta t}^h\|_{ L^2(\Omega\times \Pi_T)}^2 &= \sum_{n=0}^{N-1}\sum_{K\in \T} \Delta t |K| \mathbb{E}\big[ (u_K^n)^2\big]
   \le T\,e^{c_{\eta}T}\mathbb{E}\big[||u_0||_2^2\big].
 \end{align*}
 This completes the proof.
\end{proof}
\begin{lem}[Weak BV estimate]\label{lem:weak bv}
  Suppose $T>0$, and the assumptions \ref{A1}-\ref{A4} be true. Let $\T$ be an admissible mesh with size $h$ in the sense of Definition \ref{def:admissible mesh}.
  Let $\Delta t=\frac{T}{N}$ be the time step for some $N\in \mathbb{N}^*$, satisfying the CFL condition 
  \begin{align}
   \Delta t \le \frac{(1-\xi)\alpha^2 h}{c_f V}, \,\,\text{for some}\,\,\xi \in (0,1). \label{cond:CFL-1}
  \end{align}
  Let $u_K^n:\,K\in \T,\,\, n\in \{0,1,\cdots, N-1\}$ be the discrete unknowns as in
  \eqref{scheme:upwind fin volume}. Then the followings hold:
  \begin{itemize}
   \item [a)] There exists a positive constant $C $, only depending on $T, u_0, \xi, c_f,c_{\eta}$ such that
  \begin{align}
   \sum_{K\in\T}\sum_{n=0}^{N-1} \Delta t \sum_{\sigma\in \E_K} |\sigma| 
  |\vec{v}.n_{K,\sigma}| \mathbb{E}\Big[ \big( f(u_\sigma^n)-f(u_K^n)\big)^2\Big] \le C. \label{weak bv-0}
  \end{align}
  \item[b)] Let $R>0$ be such that $h< R$. Define $\T_R=\{ K\in \T: K\subset B(0,R)\}$ and
  $\E^R$ be the set of all interfaces of the mesh $\T_R$. Then there exists a positive constant  $C_1 $, only depending on $ R, d, T, u_0, \xi, c_f, c_{\eta}$ such that
  \begin{align}
  \sum_{n=0}^{N-1} \Delta t \sum_{\sigma\in \E^R} |\sigma|
 |\vec{v}\cdot n_{K,\sigma}| \mathbb{E}\Big[ \big|f(u_\sigma^n)-f(u_K^n)\big|\Big] \le C_1 h^{-\frac{1}{2}}. \label{weak bv-02}
  \end{align}
  \end{itemize}
\end{lem}

\begin{proof}
 Multiplying \eqref{scheme:upwind fin volume-0} by $ \Delta t\,u_K^n$, taking expectation and summing over $n=0,1,\cdots, N-1$ and $K\in \T$, we obtain 
\begin{align*}
     &\sum_{K\in \T} \sum_{n=0}^{N-1} |K| \mathbb{E}\Big[\big(u_{K}^{n+1}-u_K^n \big)u_K^n\Big]
     +  \sum_{K\in \T} \sum_{n=0}^{N-1}  \Delta t|\sigma|(\vec{v}\cdot n_{K,\sigma})^- \sum_{\underset{\sigma=K|L} {\sigma\in \E_K}} \mathbb{E}\Big[
     \big(f(u_K^n)-f(u_L^n)\big)u_K^n\Big]=0 \\
     & \text{i.e.,}\quad  \bar{\mathbf{A}} + \bar{\mathbf{B}}=0.
\end{align*}
Let us first consider $\bar{\mathbf{A}}$. Using the formula $ab=\frac{1}{2}\big[(a+b)^2-a^2-b^2\big]$ and \eqref{scheme:upwind fin volume-0}, we rewrite $\bar{\mathbf{A}}$ as
\begin{align*}
 \bar{\mathbf{A}} & = \frac{1}{2} \sum_{K\in \T} |K| \mathbb{E}\Big[(u_{K}^{N})^2-(u_K^0)^2\Big] - \frac{1}{2} \sum_{K\in \T} \sum_{n=0}^{N-1} |K|
 \mathbb{E}\Big[ \int_{t_n}^{t_{n+1}} \int_{\mathbf{E}} \eta^2(u_K^n;z)\,m(dz)\,dt \Big] \notag \\
 & \hspace{3cm} - \sum_{K\in \T} \sum_{n=0}^{N-1}\frac{(\Delta t)^2}{2|K|}  \mathbb{E}\Bigg[ \Big(  \sum_{\underset{\sigma=K|L} {\sigma\in \E_K}} |\sigma|
 (\vec{v}\cdot n_{K,\sigma})^-\big(f(u_K^n)-f(u_L^n)\big) \Big)^2\Bigg] \notag \\
 & \equiv \bar{\mathbf{A}}_1 + \bar{\mathbf{A}}_2 + \bar{\mathbf{A}}_3.
\end{align*}
Thanks to Cauchy-Schwartz inequality, the CFL condition \eqref{cond:CFL-1}, the inequality \eqref{estim:normal}, and the assumption \ref{A3}
\begin{align*}
 \bar{\mathbf{A}}_3& \ge - \frac{1}{2}   \sum_{K\in \T} \sum_{n=0}^{N-1} \Delta t  \frac{(1-\xi)}{c_f} \mathbb{E}\Bigg[
 \sum_{\underset{\sigma=K|L} {\sigma\in \E_K}} |\sigma|(\vec{v}\cdot n_{K,\sigma})^- \big(f(u_K^n)-f(u_L^n)\big)^2\Bigg], \\
\bar{\mathbf{A}}_2 & \ge -\frac{1}{2} \sum_{K\in \T} \sum_{n=0}^{N-1}  \Delta t |K| c_{\eta} \mathbb{E}\big[(u_K^n)^2\big].
\end{align*}
Therefore, by using Lemma \ref{lem:moment estimate}, we arrive at
\begin{align*}
 \bar{\mathbf{A}} &\ge -\frac{1}{2} T\,c_{\eta} e^{T\,c_{\eta}} \mathbb{E}\big[ ||u_0||_2^2\big] - \frac{1}{2} \mathbb{E}\big[||u_0||_2^2\big]
 - \frac{(1-\xi)}{2\,c_f} \sum_{K\in \T} \sum_{n=0}^{N-1} \Delta t \sum_{\underset{\sigma=K|L} {\sigma\in \E_K}} |\sigma|
 (\vec{v}\cdot n_{K,\sigma})^- \mathbb{E}\Big[ \big(f(u_K^n)-f(u_L^n)\big)^2\Big] \notag \\
 & \ge - \frac{(1-\xi)}{2\,c_f} \sum_{K\in \T} \sum_{n=0}^{N-1} \Delta t \sum_{\underset{\sigma=K|L} {\sigma\in \E_K}} |\sigma|
 (\vec{v}\cdot n_{K,\sigma})^- \mathbb{E}\Big[ \big(f(u_K^n)-f(u_L^n)\big)^2\Big] -\frac{1}{2}\tilde{C},
\end{align*}
for some constant $\tilde{C}>0$, depending only on $T, c_{\eta}, u_0$. A similar argumentations (cf.~estimation of $\mathbf{A}$) reveal that
\begin{align*}
  \bar{\mathbf{B}}\ge \frac{1}{2\,c_f} \sum_{K\in \T} \sum_{n=0}^{N-1} \Delta t \sum_{\underset{\sigma=K|L} {\sigma\in \E_K}} |\sigma|
 (\vec{v}\cdot n_{K,\sigma})^- \mathbb{E}\Big[ \big(f(u_K^n)-f(u_L^n)\big)^2\Big].
\end{align*}
Since $\bar{\mathbf{A}}+\bar{\mathbf{B}}=0$, there exists positive constant $C =C(T, u_0, \xi, c_f,c_{\eta})>0$ such that
\begin{align}
 \sum_{K\in \T} \sum_{n=0}^{N-1} \Delta t \sum_{\underset{\sigma=K|L} {\sigma\in \E_K}} |\sigma|
 (\vec{v}\cdot n_{K,\sigma})^- \mathbb{E}\Big[ \big(f(u_K^n)-f(u_L^n)\big)^2\Big] \le C, \label{weak bv-1}
\end{align}
or equivalently \eqref{weak bv-0} holds.

Let $\T_R=\{ K\in \T: K\subset B(0,R)\}$. Following \cite{bauzet-fluxsplitting}, there exists 
$C_1=C_1(R, d, T, u_0, \xi, c_f, c_{\eta})>0$ such that
\begin{align}
  \sum_{K\in \T_R} \sum_{n=0}^{N-1} \Delta t \sum_{\underset{\sigma=K|L} {\sigma\in \E_K}} |\sigma|
 (\vec{v}\cdot n_{K,\sigma})^- \mathbb{E}\Big[ \big|f(u_K^n)-f(u_L^n)\big|\Big] \le C_1 h^{-\frac{1}{2}},\label{weak bv-2}
\end{align}
holds as well. Let $\E^R$ denotes the set of all interfaces of $\T_R$. Then \eqref{weak bv-2} is equivalent to \eqref{weak bv-02}. This completes the proof of the lemma.
\end{proof}

\section{On entropy inequality for approximate solution}\label{sec:entropy-inequality}
In this section, we establish entropy inequality for finite volume approximate solution. Since we are in stochastic set up, one needs to encounter the It\^{o} calculus, and therefore it is
natural to consider a time-continuous approximate solution constructed from $u_{\T,\Delta t}^h$.
\subsection{Time-continuous approximate solution}
 Since $\text{div}_x \vec{v}(t,x)=0$ for any $(t,x)\in \Pi_T$, the upwind finite volume scheme \eqref{scheme:upwind fin volume} can be rewritten 
 as: for any $K\in \T$, and $n\in \{0,1,\cdots, N-1\}$
 \begin{align*}
\begin{cases}
     \displaystyle u_{K}^{n+1}= u_K^n   +  \frac{\Delta t}{|K|}\sum_{ \sigma\in \E_K} |\sigma| (\vec{v}\cdot n_{K,\sigma})^- \big(f(u_\sigma^n)-f(u_K^n)\big)
    +  \int_{t_n}^{t_{n+1}} \int_{\mathbf{E}} \eta(u_K^n;z)\tilde{N}(dz,dt),\\
    \displaystyle  u_K^0=\frac{1}{|K|} \int_{K} u_0(x)\,dx.
   \end{cases}
\end{align*}
We define a time-continuous  discrete approximation, denoted by $v_K^n(\omega,s)$ on $\Omega\times [t_n, t_{n+1}]$, $n\in \{0,1,\cdots, N-1\}$ and $K\in \T$ 
from the discrete unknowns $u_K^n$ as
\begin{align}
  v_{K}^{n}(\omega,s) & =  u_K^n   +  \int_{t_n}^{s}\sum_{ \sigma\in \E_K} |\sigma|(\vec{v}\cdot n_{K,\sigma})^-\frac{f(u_\sigma^n)-f(u_K^n)}{|K|}
    +  \int_{t_n}^{s} \int_{\mathbf{E}} \eta(u_K^n;z)\tilde{N}(dz,dt).\label{scheme:time-continuous upwind fin volume}
\end{align}
Note that,
\begin{align*}
 \begin{cases}
  v_K^n(\omega, t_n)=u_K^n \\
  v_K^n(\omega, t_{n+1})=u_K^{n+1}.
 \end{cases}
\end{align*}
We drop $\omega$ and write $v_K^n(\cdot)$ instead of $v_K^n(\omega,\cdot)$. Define a time-continuous approximate solution 
$v_{\T,\Delta t}^h(s,x)$ on $[0,T]\times \R^d$ by
\begin{align}
 v_{\T,\Delta t}^h(t,x)= v_K^n (t),\, x\in K, \,\,t\in [0,T).\label{eq:time continuous upwind fin-approx}
\end{align}
Next, we estimate the $L^2$-error between $u_{\T,\Delta t}^h$ and $v_{\T,\Delta t}^h$. We have the following proposition.
\begin{prop} \label{prop:error estimate}
  Let the assumptions of Lemma \ref{lem:weak bv} hold and $u_{\T,\Delta t}^h$ be the finite volume approximate solution defined by \eqref{scheme:upwind fin volume} and
  \eqref{eq:upwind fin-approx}, and $v_{\T,\Delta t}^h$ be the corresponding time-continuous approximate 
solution prescribed by \eqref{scheme:time-continuous upwind fin volume}-\eqref{eq:time continuous upwind fin-approx}. Then there 
exit two constants $C, C_1 \in \R_{+}^{*}$, independent of $h$ and $\Delta t$ such that 
\begin{align*}
 \| v_{\T,\Delta t}^h-u_{\T,\Delta t}^h\|_{L^2(\Omega \times \Pi_T)}^2 \le Ch + C_1 \Delta t.
\end{align*}
\end{prop}
\begin{proof} In view of Lemmas \ref{lem:moment estimate} - \ref{lem:weak bv}, and the estimate \eqref{estim:normal} along with \eqref{cond:CFL-1}, we have
 \begin{align}
  & \|v_{\T,\Delta t}^h-u_{\T,\Delta t}^h\|_{L^2(\Omega \times \Pi_T)}^2  \notag \\
   & = \sum_{K\in \T} \sum_{n=0}^{N-1} \int_{t_n}^{t_{n+1}}\int_{K}  \mathbb{E}\Big[  \int_{t_n}^{s} \int_{\mathbf{E}} \eta^2(u_K^n;z)
   m(dz)\,dt\Big]\,dx\,ds  \notag \\
   & \quad +\sum_{K\in \T} \sum_{n=0}^{N-1} \int_{t_n}^{t_{n+1}}\int_{K} \mathbb{E}\Big[ \Big( \frac{s-\Delta t}{|K|} 
   \sum_{ \sigma\in \E_K} |\sigma|(\vec{v}\cdot n_{K,\sigma})^-\big(f(u_\sigma^n)-f(u_K^n)\big)
   \Big)^2\Big]\, dx\,ds\notag \\
    & \le \sum_{K\in \T} \sum_{n=0}^{N-1} \frac{(\Delta t)^3}{|K|}\frac{V|K|}{\alpha^2 h}
    \sum_{ \sigma\in \E_K} |\sigma|(\vec{v}\cdot n_{K,\sigma})^- \mathbb{E}\Big[\big( f(u_\sigma^n)-f(u_K^n)\big)^2\Big] 
   +  c_{\eta}\Delta t  \|u_{\T,\Delta t}^h\|_{L^2(\Omega \times \Pi_T)}^2 \notag \\
     & \le  \frac{(\Delta t)^2 V}{\alpha^2 h}\sum_{K\in \T} \sum_{n=0}^{N-1} \Delta t 
    \sum_{ \sigma\in \E_K} |\sigma| (\vec{v}\cdot n_{K,\sigma})^- \mathbb{E}\Big[\big( f(u_\sigma^n)-f(u_K^n)\big)^2\Big] 
   +  c_{\eta}\Delta t  \|u_{\T,\Delta t}^h\|_{L^2(\Omega \times \Pi_T)}^2 \notag \\
   & \le h\frac{(1-\xi)^2 \alpha^2}{c_f^2 V}C +  c_{\eta}\Delta t  \|u_{\T,\Delta t}^h\|_{L^2(\Omega \times \Pi_T)}^2 
    \le  Ch + C_1 \Delta t,\notag
 \end{align}
 where  $C,\, C_1 \in \R_{+}^{*}$ are two constants, independent of $h$ and $\Delta t$. This finishes the proof. 
\end{proof}
\subsection{Entropy inequalities for the approximate solution}
This subsection is devoted to derive the entropy inequalities for the finite volume approximate solution which will be used to prove the 
convergence of the numerical scheme and hence the existence of entropy solution of the underlying problem \eqref{eq:levy_con_laws}. To do 
so, we start with the following proposition related to the entropy inequalities for the discrete unknowns $u_K^n,\, K\in \T,\,\,
n\in \{0,1,2,\dots,N-1\}$. 
\begin{prop}[Discrete entropy inequalities]\label{prop:discrete entropy ineq}
Let the assumptions \ref{A1}-\ref{A4} hold, and $T>0$ be fixed. Consider an
  admissible mesh $\T$ on $\R^d$ with size $h$ in the sense of Definition \ref{def:admissible mesh}.
  Let $\Delta t=\frac{T}{N}$ be the time step for some $N\in \mathbb{N}^*$, satisfying 
  \begin{align*}
   \frac{\Delta t}{h} \goto 0 \,\,\text{as}\,\, h\goto 0.
  \end{align*}
  Then, $\mathbb{P}$-a.s. in $\Omega$, for any $\beta \in \mathcal{A}$ and for any nonnegative test function  $\psi\in C_c^\infty([0, T)\times\R^d)$, the following inequality holds:
  \begin{align*}
   - & \sum_{n=0}^{N-1} \sum_{K\in \T_R} \int_{K} \big( \beta(u_K^{n+1})-\beta(u_K^n)\big) \psi(t_n,x)\,dx \notag \\
   & + \sum_{n=0}^{N-1} \sum_{K\in \T_R} \int_{t_n}^{t_{n+1}} \int_{K} F^\beta(u_K^n)\vec{v}(t,x)\cdot \grad_x \psi(t_n,x)\,dx\,dt \notag \\
   & + \sum_{n=0}^{N-1} \sum_{K\in \T_R} \int_{t_n}^{t_{n+1}} \int_{\mathbf{E}} \int_{K}\int_0^1 \eta(u_K^n;z) \beta^\prime \big(u_K^n + \lambda 
   \eta(u_K^n;z)\big)\psi(t_n,x)d\lambda\,dx\,\tilde{N}(dz,dt) \notag \\
    & + \sum_{n=0}^{N-1} \sum_{K\in \T_R} \int_{t_n}^{t_{n+1}} \int_{\mathbf{E}} \int_{K}\int_0^1 (1-\lambda) \eta^2(u_K^n;z) 
    \beta^{\prime\prime} \big(u_K^n + \lambda \eta(u_K^n;z)\big)\psi(t_n,x)d\lambda\,dx\,m(dz)\,dt 
     \ge R^{h,\Delta t},
  \end{align*}
where $R^{h,\Delta t}$ satisfies the following condition: for any $\mathbb{P}$-measurable set $B$, $\mathbb{E}\Big[ {\bf 1}_B R^{h,\Delta t}\Big]\goto 0$ as 
$h\goto 0$.
\end{prop}

\begin{proof}
 Let $T>0$ be fixed and $\T$ be an admissible mesh  on $\R^d$ with size $h$ in the sense of Definition \ref{def:admissible mesh}.
Let $\Delta t=\frac{T}{N}$ be the time step for some $N\in \mathbb{N}^*$ and $t_n=n\Delta t, n\in \{0,1,\cdots, N\}$.
 Let $\beta\in \mathcal{A}$. Applying It\^{o}-L\'{e}vy formula to $\beta(v_K^n)$, where $v_K^n$ is prescribed by the equation 
 \eqref{scheme:time-continuous upwind fin volume}, we have
 \begin{align}
  \beta(v_K^n(t_{n+1})) & =\beta(v_K^n(t_n)) + \int_{t_n}^{t_{n+1}} \beta^\prime(v_K^n(t)) 
  \sum_{ \sigma\in \E_K} |\sigma|(\vec{v}\cdot n_{K,\sigma})^- \frac{f(u_\sigma^n)-f(u_K^n)}{|K|} \,dt \notag \\
  & +   \int_{t_n}^{t_{n+1}} \int_{\mathbf{E}} \int_0^1 \eta(u_K^n;z) \beta^\prime \big(v_K^n(t) + \lambda 
   \eta(u_K^n;z)\big)d\lambda\,\tilde{N}(dz,dt) \notag \\
   & +   \int_{t_n}^{t_{n+1}} \int_{\mathbf{E}} \int_0^1  (1-\lambda)\eta^2(u_K^n;z) \beta^{\prime\prime} \big(v_K^n(t) + \lambda 
   \eta(u_K^n;z)\big)d\lambda\,m(dz)\,dt.\label{ineq: discrete entropy-1}
 \end{align}
 Let  $\psi\in C_c^\infty([0, T)\times\R^d)$ be a nonnegative test function. Then there exists $R>h$ such that 
   $\text{supp}\psi \subset [0,T) \times B(0,R-h)$. Also define $\T_R =\{ K\in \T : K\subset B(0,R)\}$.
  We multiply the equation \eqref{ineq: discrete entropy-1}
 by $|K| \psi_K^n$ where $\displaystyle \psi_K^n= \frac{1}{|K|} \int_{K}\psi(t_n,x)\,dx$ and then we sum over all $K\in \T_R$ and 
 $n\in \{0,1,\cdots, N-1\}$. The resulting expression reads to
 \begin{align}
  & \sum_{n=0}^{N-1} \sum_{K\in \T_R} \big[ \beta(u_K^{n+1})-\beta(u_K^n)\big] |K|\psi_K^n \notag \\
  & = \sum_{n=0}^{N-1} \sum_{K\in \T_R} \int_{t_n}^{t_{n+1}} \beta^\prime(v_K^n(t)) 
  \sum_{ \sigma\in \E_K} |\sigma|(\vec{v}\cdot n_{K,\sigma})^- \big(f(u_\sigma^n)-f(u_K^n)\big) \psi_K^n \,dt \notag \\
   & +  \sum_{n=0}^{N-1} \sum_{K\in \T_R}  \int_{t_n}^{t_{n+1}} \int_{\mathbf{E}} \int_0^1 \eta(u_K^n;z) \beta^\prime \big(v_K^n(t) + \lambda 
   \eta(u_K^n;z)\big) |K| \psi_K^n d\lambda\,\tilde{N}(dz,dt) \notag \\
   & +  \sum_{n=0}^{N-1} \sum_{K\in \T_R}  \int_{t_n}^{t_{n+1}} \int_{\mathbf{E}} \int_0^1  (1-\lambda)\eta^2(u_K^n;z)
   \beta^{\prime\prime} \big(v_K^n(t) + \lambda 
   \eta(u_K^n;z)\big)|K|\psi_K^n  d\lambda\,m(dz)\,dt\notag \\
   & \text{i.e.}\,\,\,\, A^{h,\Delta t}= B^{h,\Delta t} + M^{h,\Delta t} + D^{h,\Delta t}. \label{establish:initial-sum}
 \end{align}
 Following \cite{bauzet-fluxsplitting}, we express $B^{h,\Delta t}$ as follows.
 \begin{align*}
  B^{h,\Delta t}= B^{h,\Delta t}- B_1^{h,\Delta t} + B_1^{h,\Delta t} -B_2^{h,\Delta t} + B_2^{h,\Delta t},
 \end{align*}
 where 
 \begin{align*}
  B_1^{h,\Delta t}&= \sum_{n=0}^{N-1} \sum_{K\in \T_R} \int_{t_n}^{t_{n+1}} \beta^\prime(u_K^n) 
  \sum_{ \sigma\in \E_K} |\sigma| (\vec{v}\cdot n_{K,\sigma})^-  \big(f(u_\sigma^n)-f(u_K^n)\big) \psi_K^n \,dt, \notag \\
   B_2^{h,\Delta t}&= \sum_{n=0}^{N-1} \sum_{K\in \T_R} \int_{t_n}^{t_{n+1}} 
  \sum_{ \sigma\in \E_K} |\sigma| (\vec{v}\cdot n_{K,\sigma})^- \big(F^\beta(u_\sigma^n)-F^\beta(u_K^n)\big) \psi_K^n \,dt.
 \end{align*}
 Observe that
 \begin{align*}
   & B_1^{h,\Delta t} -B_2^{h,\Delta t} = \sum_{n=0}^{N-1} \sum_{K\in \T_R} \Delta t \sum_{ \sigma\in \E_K} |\sigma| (\vec{v}\cdot n_{K,\sigma})^-
   \Big\{  \beta^\prime(u_K^n)  \big(f(u_\sigma^n)-f(u_K^n)\big)- \big(F^\beta(u_\sigma^n)-F^\beta(u_K^n)\big)\Big\} \psi_K^n.
 \end{align*}
 Thanks to nondecreasingness of the functions $f$ and $\beta^\prime$, one has 
 \begin{align*}
   \beta^\prime(u_K^n)  \big(f(u_\sigma^n)-f(u_K^n)\big)-  \big(F^\beta(u_\sigma^n)-F^\beta(u_K^n)\big)
   & = \int_{u_K^n}^{u_\sigma^n} \big(\beta^\prime(u_K^n)- \beta^{\prime}(s)\big)f^\prime(s)ds \le 0,
 \end{align*}
 and hence $ B_1^{h,\Delta t} -B_2^{h,\Delta t} \le 0$.
 
 By the assumption \ref{A2}, we have
 $\displaystyle \sum_{\sigma\in \E_K} |\sigma|\vec{v}\cdot n_{K,\sigma} F^\beta(u_K^n)\psi_K^n=0$, and therefore
 \begin{align*}
  B_2^{h,\Delta t}= - \sum_{n=0}^{N-1} \sum_{K\in \T_R} \Delta t \sum_{ \sigma\in \E_K} |\sigma|(\vec{v}\cdot n_{K,\sigma})
  F^\beta(u_\sigma^n) \psi_K^n.
 \end{align*}
 Let $x_\sigma$ be the center of the edge $\sigma$ and $\psi_\sigma^n$ be the value of $\psi( t_n,x_\sigma)$. Then, 
 \begin{align*}
  \sum_{n=0}^{N-1} \sum_{K\in \T_R} \Delta t \sum_{ \sigma\in \E_K} |\sigma| (\vec{v}\cdot n_{K,\sigma})
  F^\beta(u_\sigma^n) \psi_\sigma^n=0.
 \end{align*}
 A similar argument (as described in Bauzet et al.\cite[Proposition\,$4$]{bauzet-fluxsplitting}) reveals that
 \begin{align*}
   B_2^{h,\Delta t} & = \sum_{n=0}^{N-1} \sum_{K\in \T_R}\int_{t_n}^{t_n+1} \int_{ K} F^\beta(u_K^n) \vec{v}(t,x)\cdot \grad_x\psi(t_n,x) \,dx\,dt
  +  R_1^{h,\Delta t} + R_2^{h,\Delta t},
 \end{align*}
 where 
 \begin{align*}
  R_1^{h,\Delta t}&= \sum_{n=0}^{N-1} \sum_{K\in \T_R} \Delta t \sum_{ \sigma\in \E_K} |\sigma| (\vec{v}\cdot n_{K,\sigma}) 
 \big[ F^\beta(u_K^n)-F^\beta(u_\sigma^n)\big] \big(\psi_K^n- \psi_\sigma^n\big ),\\
 R_2^{h,\Delta t}&=  \sum_{n=0}^{N-1} \sum_{K\in \T_R} \Delta t \sum_{ \sigma\in \E_K} \Big\{ |\sigma| (\vec{v}\cdot n_{K,\sigma})\psi_\sigma^n - 
  \int_{\sigma} (\vec{v}\cdot n_{K,\sigma})\psi(t_n,x) d\nu(x) \Big\} F^\beta(u_K^n).
 \end{align*}
 Combining all these, we obtain that 
 \begin{align}
  B^{h,\, \Delta t} &\le  \sum_{n=0}^{N-1} \sum_{K\in \T_R}\int_{t_n}^{t_n+1} \int_{ K} F^\beta(u_K^n)
  \vec{v}(t,x)\cdot \grad_x\psi(t_n,x) \,dx\,dt \notag \\
  & \hspace{3cm} + B^{h,\Delta t}- B_1^{h,\Delta t}  +  R_1^{h,\Delta t} + R_2^{h,\Delta t}. \label{establish: B(h,delta)}
 \end{align}
 Next we consider the term $M^{h,\Delta t}$. It can be decompose as follows:
 \begin{align}
  M^{h,\Delta t}= M^{h,\Delta t}- M_1^{h,\Delta t} + M_1^{h,\Delta t}, \label{establish: M(h,delta)}
 \end{align}
 where 
 \begin{align*}
  M_1^{h,\Delta t}= \sum_{n=0}^{N-1} \sum_{K\in \T_R}  \int_{t_n}^{t_{n+1}} \int_{\mathbf{E}} \int_{K}
  \int_0^1 \eta(u_K^n;z) \beta^\prime \big(u_K^n + \lambda 
   \eta(u_K^n;z)\big) \psi(t_n,x)\, d\lambda\,dx\,\tilde{N}(dz,dt).
 \end{align*}
 Similarly, we rewrite $D^{h,\Delta t}$ as 
 \begin{align}
  D^{h,\Delta t}= D^{h,\Delta t}- D_1^{h,\Delta t} + D_1^{h,\Delta t}, \label{establish: D(h,delta)}
 \end{align}
 where 
 \begin{align*}
  &D_1^{h,\Delta t} \\
  &= \sum_{n=0}^{N-1} \sum_{K\in \T_R} \int_{t_n}^{t_{n+1}} \int_{\mathbf{E}} \int_{K}\int_0^1 (1-\lambda) \eta^2(u_K^n;z) 
    \beta^{\prime\prime} \big(u_K^n + \lambda \eta(u_K^n;z)\big)\psi(t_n,x)\,d\lambda\,dx\,m(dz)\,dt.
 \end{align*}
 In view of \eqref{establish: B(h,delta)}, \eqref{establish: M(h,delta)}, \eqref{establish: D(h,delta)}, and \eqref{establish:initial-sum}
 we have
 \begin{align*}
   - & \sum_{n=0}^{N-1} \sum_{K\in \T_R} \int_{K} \big( \beta(u_K^{n+1})-\beta(u_K^n)\big) \psi(t_n,x)\,dx \notag \\
   & + \sum_{n=0}^{N-1} \sum_{K\in \T_R} \int_{t_n}^{t_{n+1}} \int_{K} F^\beta(u_K^n)\vec{v}(t,x)\cdot \grad_x \psi(t_n,x)\,dx\,dt \notag \\
   & + \sum_{n=0}^{N-1} \sum_{K\in \T_R} \int_{t_n}^{t_{n+1}} \int_{\mathbf{E}} \int_{K}\int_0^1 \eta(u_K^n;z) \beta^\prime \big(u_K^n + \lambda 
   \eta(u_K^n;z)\big)\psi(t_n,x)\,d\lambda\,dx\,\tilde{N}(dz,dt) \notag \\
    & + \sum_{n=0}^{N-1} \sum_{K\in \T_R} \int_{t_n}^{t_{n+1}} \int_{\mathbf{E}} \int_{K}\int_0^1 (1-\lambda) \eta^2(u_K^n;z) 
    \beta^{\prime\prime} \big(u_K^n + \lambda \eta(u_K^n;z)\big)\psi(t_n,x)d\lambda\,dx\,m(dz)\,dt \notag \\
    & \ge  \big(B_1^{h,\Delta t}- B^{h,\Delta t}\big) -R_1^{h,\Delta t}- R_2^{h,\Delta t} + \big( M_1^{h,\Delta t}-M^{h,\Delta t}\big)
    + \big(D_1^{h,\Delta t}-D^{h,\Delta t}\big) \equiv R^{h,\Delta t}.
 \end{align*}
 To complete the proof of the proposition, it is only required to show: for any $\mathbb{P}$-measurable set $B$, $\mathbb{E}\Big[ {\bf 1}_B R^{h,\Delta t}\Big]\goto 0$ as  $h\goto 0$. Now we assume that $ \displaystyle \frac{\Delta t}{h} \goto 0$ as 
 $h\goto 0$. In this manner, with out loss of generality, we may assume that the CFL condition 
 \begin{align*}
  \Delta t \le \frac{(1-\xi)h}{c_f V}\alpha^2 
 \end{align*}
 holds for some $\xi \in (0,1)$ and hence the estimates given in Lemmas \ref{lem:moment estimate} and \ref{lem:weak bv} hold as well.
To proceed further, we will separately show the convergence of $\mathbb{E}\Big[ {\bf 1}_B\big( B_1^{h,\Delta t}- B^{h,\Delta t}\big)\Big]$, 
$\mathbb{E}\Big[ {\bf 1}_B\big( M_1^{h,\Delta t}- M^{h,\Delta t}\big)\Big]$, $\mathbb{E}\Big[ {\bf 1}_B\big( D_1^{h,\Delta t}- D^{h,\Delta t}\big)\Big]$,
$\mathbb{E}\big[ {\bf 1}_B\, R_1^{h,\Delta t}\big]$, and $\mathbb{E}\big[ {\bf 1}_B\, R_2^{h,\Delta t}\big]$.
\vspace{.2cm}

\noindent{1. \underline{\bf Study of $\mathbb{E}\big[ {\bf 1}_B\big( B_1^{h,\Delta t}- B^{h,\Delta t}\big)\big]$:}}
Let $B$ be any $\mathbb{P}$-measurable set. Then, by using \eqref{scheme:time-continuous upwind fin volume} we get
\begin{align*}
 &\Big| \mathbb{E}\big[ {\bf 1}_B\big( B_1^{h,\Delta t}- B^{h,\Delta t}\big)\big]\Big| \\
 & = \Bigg| \mathbb{E}\Big[\sum_{n=0}^{N-1} \sum_{ K \in \T_R} 
 \int_{t_n}^{t_{n+1}} {\bf 1}_B  \beta^{\prime\prime}(\xi_K^n)\big(v_K^n(s))-u_K^n)\big) \sum_{\sigma\in \E_K} 
 |\sigma| (\vec{v}\cdot n_{K,\sigma})^-  \big(f(u_\sigma^n)-f(u_K^n)\big) \psi_K^n \,ds\Big]\Bigg|\\
 & \le \Bigg| \mathbb{E}\Big[\sum_{n=0}^{N-1} \sum_{ K \in \T_R} {\bf 1}_B 
 \int_{t_n}^{t_{n+1}} \beta^{\prime\prime}(\xi_K^n) \frac{s-t_n}{|K|} \Big( \sum_{\sigma\in \E_K} 
 |\sigma| (\vec{v}\cdot n_{K,\sigma})^-  \big(f(u_\sigma^n)-f(u_K^n)\big)\Big)^2 \psi_K^n \,ds\Big]\Bigg|\\
 & + \Bigg| \mathbb{E}\Big[\sum_{n=0}^{N-1} \sum_{ K \in \T_R} {\bf 1}_B 
 \int_{t_n}^{t_{n+1}} \beta^{\prime\prime}(\xi_K^n) \int_{t_n}^{s}\int_{\mathbf{E}} \eta(u_K^n;z)\tilde{N}(dz,dr) \sum_{\sigma\in \E_K} 
 |\sigma| (\vec{v}\cdot n_{K,\sigma})^-  \big(f(u_\sigma^n)-f(u_K^n)\big) \psi_K^n \,ds\Big]\Bigg| \\
 & \equiv T_1^{h,\Delta t} + T_2^{h,\Delta t}.
\end{align*}
Following computations as in \cite[estimation for $\tilde{T}_1^{h,k}$]{bauzet-fluxsplitting}, it can be shown that 
\begin{align*}
& T_1^{h,\Delta t} \le \frac{\Delta t}{h} ||\beta^{\prime\prime}||_{L^\infty} ||\psi||_{L^\infty} \frac{V}{\alpha^2} C.
\end{align*}
  Next, we move on to estimate $T_2^{h,\Delta t}$. Note that 
\begin{align*}
 \big|T_2^{h,\Delta t}\big|^2
 & \le \mathbb{E}\Big[ \sum_{n=0}^{N-1} \sum_{ K \in \T_R} \int_{t_n}^{t_{n+1}} \Big( {\bf 1}_B \beta^{\prime\prime}(\xi_K^n)
 \psi_K^n \sum_{\sigma\in \E_K}  |\sigma| (\vec{v}\cdot n_{K,\sigma})^- \big(f(u_\sigma^n)-f(u_K^n)\big) \Big)^2\,ds \Big] \\
 & \hspace{2cm} \times  \mathbb{E}\Big[ \sum_{n=0}^{N-1} \sum_{ K \in \T_R} \int_{t_n}^{t_{n+1}} \Big( 
 \int_{t_n}^{s}\int_{\mathbf{E}} \eta(u_K^n;z)\tilde{N}(dz,dr) \Big)^2\,ds\Big]\\
 &\le ||\beta^{\prime\prime}||_{L^\infty}^2 ||\psi||_{L^\infty}^2 \sum_{n=0}^{N-1} \sum_{ K \in \T_R} \Delta t 
  \Big(\sum_{\sigma\in \E_K} |\sigma| (\vec{v}\cdot n_{K,\sigma})^-\Big) 
  \sum_{\sigma\in \E_K} |\sigma| (v\cdot n_{K,\sigma})^- \mathbb{E}\Big[\big(f(u_\sigma^n)-f(u_K^n)\big)^2\Big]\\
  & \hspace{4cm} \times  \mathbb{E}\Big[ \sum_{n=0}^{N-1} \sum_{ K \in \T_R} \int_{t_n}^{t_{n+1}}
 \int_{t_n}^{s}\int_{\mathbf{E}} \eta^2(u_K^n;z)\,m(dz)\,dr\,ds\Big]\\
 &\le \frac{V}{\alpha^2 h} ||\beta^{\prime\prime}||_{L^\infty}^2 ||\psi||_{L^\infty}^2 \sum_{n=0}^{N-1} \sum_{ K \in \T_R} \Delta t 
  \sum_{\sigma\in \E_K} |\sigma| (\vec{v}\cdot n_{K,\sigma})^- \mathbb{E}\Big[\big(f(u_\sigma^n)-f(u_K^n)\big)^2\Big]\\
  & \hspace{4cm} \times   c_{\eta}\Delta t \sum_{n=0}^{N-1} \sum_{ K \in \T_R} |K| \Delta t
 \mathbb{E}\big[(u_K^n)^2\big]\\
 & \le \frac{\Delta t}{h} \frac{c_{\eta}V }{\alpha^2} ||\beta^{\prime\prime}||_{L^\infty}^2 ||\psi||_{L^\infty}^2
 \|u_{\T,\Delta}\|_{L^2(\Omega\times \Pi_T)}^2 C.
\end{align*}
In the above, the first inequality follows from Cauchy-Schwartz inequality, second inequality follows from Cauchy-Schwartz inequality and It\^{o}-L\'{e}vy isometry. In view of \eqref{estim:normal} and the assumption \ref{A3} on $\eta$, the third inequality holds 
true. In the last inequality, we have used the constant $C$ given by Lemma \ref{lem:weak bv}. Here we note that the assumption 
$ \displaystyle\frac{\Delta t}{h} \goto 0$ as $h\goto 0$ is crucial. Passing to the limit as $h\goto 0$, we conclude that 
$\mathbb{E}\Big[ {\bf 1}_B\big( B_1^{h,\Delta t}- B^{h,\Delta t}\big)\Big]\goto 0$. 
\vspace{.1cm}

\noindent{2. \underline{\bf Study of $\mathbb{E}\big[ {\bf 1}_B\big( M_1^{h,\Delta t}- M^{h,\Delta t}\big)\big]$:}} In view of triangle inequality, one has
\begin{align*}
 \Big|\mathbb{E}\Big[ {\bf 1}_B\big( M_1^{h,\Delta t}- M^{h,\Delta t}\big)\Big]\Big| \le \mathcal{M}_1^{h,\Delta t} + \mathcal{M}_2^{h,\Delta t},
 \end{align*}
 where 
 \begin{align*}
 &\mathcal{M}_1^{h,\Delta t}:=  \Bigg| \mathbb{E}\Big[{\bf 1}_B \sum_{n=0}^{N-1} \sum_{K\in \T_R}  \int_{t_n}^{t_{n+1}} \int_{\mathbf{E}} \int_{K}
  \int_0^1 \eta(u_K^n;z)\Big\{ \beta^\prime \big(u_K^n + \lambda 
   \eta(u_K^n;z)\big)-  \beta^\prime \big(v_K^n(t) + \lambda \eta(u_K^n;z)\big)\Big\} \\
   & \hspace{6cm} \times \big(\psi(t_n,x)-\psi(t,x)\big) d\lambda\,dx\,\tilde{N}(dz,dt)\Big]\Bigg|\\
   & \mathcal{M}_2^{h,\Delta t}:= \Bigg| \mathbb{E}\Big[{\bf 1}_B \sum_{n=0}^{N-1} \sum_{K\in \T_R}  \int_{t_n}^{t_{n+1}} \int_{\mathbf{E}} \int_{K}
  \int_0^1 \eta(u_K^n;z)\Big\{ \beta^\prime \big(u_K^n + \lambda 
   \eta(u_K^n;z)\big)- \beta^\prime \big(v_K^n(t) + \lambda \eta(u_K^n;z)\big)\Big\} \\
   & \hspace{6cm} \times \psi(t,x) d\lambda\,dx\,\tilde{N}(dz,dt)\Big]\Bigg|.
 \end{align*}
 Let us turn our focus on the term $\mathcal{M}_1^{h,\Delta t}$. Note that  $ \text{supp}\, \psi \subset B(0,R-h)\times [0,T)$ for some 
 $R>h$. Using Cauchy-Schwartz inequality along with the assumptions \ref{A1}-\ref{A4}, and It\^{o}-L\'{e}vy isometry, we get 
 \begin{align*}
   \Big|\mathcal{M}_1^{h,\Delta t}\Big|^2 
  & \le \Bigg(\sum_{n=0}^{N-1} \Bigg\{\sum_{K\in \T_R} \int_{K} \mathbb{E}\Big[\Big( \int_{t_n}^{t_{n+1}} \int_{\mathbf{E}} 
  \int_0^1 \big( \beta^\prime \big(u_K^n + \lambda \eta(u_K^n;z)\big)-\beta^\prime \big(v_K^n(t) + \lambda \eta(u_K^n;z)
  \big)\big) \\
   & \hspace{3cm} \times \eta(u_K^n;z)  \big(\psi(t_n,x)-\psi(t,x)\big) d\lambda \tilde{N}(dz,dt)
   \Big)^2\Big] \,dx \Bigg\}^\frac{1}{2}\Bigg)^2 |B(0,R)|  \\
   & \le 2 ||\beta^\prime||_{L^\infty}^2|B(0,R)| \Bigg(\sum_{n=0}^{N-1} \Big \{\sum_{K\in \T_R} \int_{K} \mathbb{E}\Big[\int_{t_n}^{t_{n+1}} 
   \int_{\mathbf{E}}  \eta^2(u_K^n;z)  \big(\psi(t_n,x)-\psi(t,x)\big)^2\,m(dz)\,dt\Big]\,dx\Big\}^\frac{1}{2}\Bigg)^2\\
   & \le 2 \Delta t ||\beta^\prime||_{L^\infty}^2 |B(0,R)|\,||\partial_t \psi||_{L^\infty}^2 c_{\eta}
  \Big( \sum_{n=0}^{N-1} \Delta t  \Big\{ \sum_{K\in \T_R}   |K| \mathbb{E}\big[(u_K^n)^2\big]\Big\}^\frac{1}{2}\Big)^2 \\
   & \le (\Delta t) C\big(\beta^\prime, \partial_t \psi, c_{\eta}, R \big) T \|u_{\T,\Delta t}^h\|_{L^\infty(0,T; L^2(\Omega\times \R^d))}^2\\
   & \le h C\big(\xi,\alpha, c_f, V,\beta^\prime, \partial_t \psi, c_{\eta}, R\big) T \|u_{\T,\Delta t}^h\|_{L^\infty(0,T; L^2(\Omega\times \R^d))}^2
   \, (\text{by}\,\eqref{cond:CFL-1}).
 \end{align*}
Thus, we see that 
\begin{align*}
 \mathcal{M}_1^{h,\Delta t} \longrightarrow 0 \,\,\,\text{as}\,\,h\goto 0.
\end{align*}
Now, we estimate $\mathcal{M}_2^{h,\Delta t}$. Here we note that the boundedness of $\eta$ i.e. $|\eta(u,z)|\le C h_1(z)$ for 
any $u\in \R$ and $z\in \mathbf{E}$ is  crucial. In view of the Cauchy-Schwartz inequality and It\^{o}-L\'{e}vy isometry, we obtain
\begin{align*}
\Big|\mathcal{M}_2^{h,\Delta t}\Big|^2
   & \le  \sum_{n=0}^{N-1} \sum_{K\in \T_R} \int_{K} \mathbb{E}\Big[\int_{t_n}^{t_{n+1}} \int_{\mathbf{E}} 
  \int_0^1 \Big\{ \beta^\prime \big(u_K^n + \lambda \eta(u_K^n;z)\big)-\beta^\prime \big(v_K^n(t) + \lambda \eta(u_K^n;z)
  \big)\Big\}^2 \\
   & \hspace{3cm} \times \eta^2(u_K^n;z)\psi^2(t,x)\, d\lambda \,m(dz)\,dt\Big]\,dx |B(0,R)| \\
    & \le C(R)||\beta^{\prime\prime}||_{\infty}^2 ||\psi||_\infty^2 \sum_{n=0}^{N-1} \sum_{K\in \T_R} \int_{K} \mathbb{E}\Big[\int_{t_n}^{t_{n+1}} \int_{\mathbf{E}} 
   |u_K^n -v_K^n(t)|^2 \eta^2(u_K^n;z)\,m(dz)\,dt\Big]\,dx  \\   
   & (\text{by the boundedness of}\,\eta)\\
   & \le C\big(R,\beta^{\prime\prime}, \psi,\big) \sum_{n=0}^{N-1} \sum_{K\in \T_R} \int_{K} \int_{t_n}^{t_{n+1}}
   \mathbb{E}\big[ (u_K^n-v_K^n(t))^2\big]\,dt\,dx \Big( \int_{\mathbf{E}} h_1^2(z) m(dz)\Big) \\
   & = C\big(R,\beta^{\prime\prime}, \psi,\big) c_{\eta} \|u_{\T,\Delta t}^h- v_{\T,\Delta t}^h\|_{L^2(\Omega \times \Pi_T)}^2 \longrightarrow 0\,\text{as}\,\,
   h\goto 0,
\end{align*}
where in the last line, we have invoked Proposition \ref{prop:error estimate} and the CFL condition \eqref{cond:CFL-1}. Hence 
 \begin{align*}
 \mathbb{E}\Big[ {\bf 1}_B\big( M_1^{h,\Delta t}- M^{h,\Delta t}\big)\Big] \goto 0 \quad \text{as} \quad h \goto 0.
 \end{align*}
\vspace{.1cm}

\noindent{3. \underline{\bf Study of $\mathbb{E}\big[ {\bf 1}_B\big( D_1^{h,\Delta t}- D^{h,\Delta t}\big)\big]$:}} Observe that
\begin{align*}
 &\Big|\mathbb{E}\big[ {\bf 1}_B\big( D_1^{h,\Delta t}- D^{h,\Delta t}\big)\big]\Big|\\
  & \le ||\beta^{\prime \prime\prime}||_{\infty} \,||\psi||_{\infty} 
  \mathbb{E}\Big[ \sum_{n=0}^{N-1} \sum_{K\in \T_R} \int_{t_n}^{t_{n+1}} \int_{\mathbf{E}}\int_{B(0,R)} \eta^2(u_K^n;z) \big|u_K^n -v_K^n(t)\big| dx\,m(dz)\,dt \Big]\\
  & \le C(\beta^{\prime\prime\prime}, \psi) \mathbb{E}\Big[ \sum_{n=0}^{N-1} \sum_{K\in \T_R} \int_{t_n}^{t_{n+1}} \int_{B(0,R)} 
  \big|u_K^n -v_K^n(t)\big| dx\,dt\Big] \Big( \int_{\mathbf{E}} h_1^2(z)\,m(dz)\Big) \\
  & =  C(\beta^{\prime\prime}, \psi, c_{\eta}) \| u_{\T,\Delta t}^h-v_{\T,\Delta t}^h\|_{L^1\big(\Omega \times B(0,R)\times [0,T)\big)}
   \longrightarrow 0 \,\,\text{as}\,\, h \goto 0.
\end{align*}
In the above, the second inequality follows from the boundedness condition on $\eta$, and the last line holds because of Proposition 
\ref{prop:error estimate} and the CFL condition \eqref{cond:CFL-1}.
\vspace{.1cm}

\noindent{4. \underline{\bf Study of $\mathbb{E}\big[ {\bf 1}_B\,R_1^{h,\Delta t}\big]$ and $\mathbb{E}\big[ {\bf 1}_B\,R_2^{h,\Delta t}\big]$: }}
Following computations as in Bauzet et al. \cite[Proposition\,$4$]{bauzet-fluxsplitting}
we infer that 
\begin{align*}
\Big| \mathbb{E}\big[ {\bf 1}_B\,R_1^{h,\Delta t}\big]\Big|
  & \le C ||\psi_x||_{\infty} ||\beta^\prime||_{\infty} \sqrt{h}; \,\,
 \Big|\mathbb{E}\big[ {\bf 1}_B\,R_2^{h,\Delta t}\big]\Big|  \le ||\beta^\prime||_{\infty} \frac{V c_f}{\alpha^2} \bar{\eps}(h)
   \|u_{\T,\Delta t}^h\|_{L^1\big( \Omega \times B(0,R)\times [0,T)\big)},
\end{align*}
where $\bar{\eps}(r) \goto 0$ as $r\goto 0$. 

We now combine all the above estimates to conclude: for any $\mathbb{P}$-measurable set $B$, 
$$\mathbb{E}\Big[ {\bf 1}_B R^{h,\Delta t}\Big]\goto 0\,\,\text{as}\,\, h\goto 0.$$ This completes the proof of the proposition. 
\end{proof}
To prove convergence of the proposed scheme and hence existence of entropy solution for \eqref{eq:levy_con_laws}, 
one also needs a continuous entropy inequality on the discrete solutions. Regarding this, we have the following proposition which essentially 
gives the entropy inequality for the finite volume approximate solution $u_{\T,\Delta t}^h$.

\begin{prop}[Entropy inequality for approximate solution]\label{prop:entropy inequality on appro solution}
Let the assumptions \ref{A1}-\ref{A4} hold, and $T>0$ be fixed. Let  $\T$ be an
  admissible mesh on $\R^d$ with size $h$ in the sense of Definition \ref{def:admissible mesh}.
  Let $\Delta t=\frac{T}{N}$ be the time step for some $N\in \mathbb{N}^*$ satisfying 
  \begin{align*}
   \frac{\Delta t}{h} \goto 0 \,\,\text{as}\,\, h\goto 0.
  \end{align*}
  Then, $\mathbb{P}$-a.s.~in $\Omega$, for any $\beta \in \mathcal{A}$ and for any nonnegative test function 
  $\psi\in C_c^\infty([0, T)\times\R^d)$, the following inequality holds:
  \begin{align}
   & \int_{\R^d} \beta(u_0(x))\psi(0,x)\,dx  + \int_{\Pi_T} \Big\{\beta(u_{\T,\,\Delta t}^h)\partial_t \psi(t,x) + 
  F^\beta(u_{\T,\,\Delta t}^h) \vec{v}(t,x)\cdot \grad_x \psi(t,x)\Big\}\,dt\,dx \notag \\
   & + \int_{\R^d}\int_{0}^T \int_{\mathbf{E}} \int_0^1 \eta(u_{\T,\,\Delta t}^h;z) \beta^\prime \big(u_{\T,\,\Delta t}^h + \lambda 
   \eta(u_{\T,\,\Delta t}^h;z)\big)\psi(t,x)\,d\lambda \,\tilde{N}(dz,dt)\,dx \notag \\
    & + \int_{\R^d}\int_{0}^T \int_{\mathbf{E}}\int_0^1 (1-\lambda) 
    \eta^2(u_{\T,\,\Delta t}^h;z) \beta^{\prime\prime} \big(u_{\T,\,\Delta t}^h + \lambda 
   \eta(u_{\T,\,\Delta t}^h;z)\big)\psi(t,x) d\lambda\,m(dz)\,dt\,dx \notag \\
    & \ge \mathcal{R}^{h,\,\Delta t}, \label{eq: entropy inequality for approximate solution}
  \end{align}
where for any $\mathbb{P}$-measurable set $B$, $\mathbb{E}\Big[ {\bf 1}_B \mathcal{R}^{h,\,\Delta t}\Big]\goto 0$ as $h\goto 0$.
\end{prop}
\begin{proof}
Let the assumptions of the proposition hold true. Since  $ \displaystyle \frac{\Delta t}{h} \goto 0$ as 
 $h\goto 0$, we may assume that the CFL condition 
 \begin{align*}
  \Delta t \le \frac{(1-\xi)h}{c_f V}\alpha^2 
 \end{align*}
 holds for some $\xi \in (0,1)$ and hence the estimates given in Lemmas \ref{lem:moment estimate} and \ref{lem:weak bv} and  
  Proposition \ref{prop:discrete entropy ineq} hold as well. Let  $\psi\in C_c^\infty([0, T)\times\R^d)$ be a nonnegative test function. Then there exists $R>h$ such that 
  supp $\psi \subset  [0,T)\times B(0,R-h)$. Also define $\T_R =\{ K\in \T : K\subset B(0,R)\}$.
\vspace{.1cm}

Note that $\psi(t_N,x)=0$ for any $x\in \R^d$. Using the summation by parts formula, $$ \sum_{n=1}^{N} a_n \big(b_n -b_{n-1}\big)= a_N b_N -a_0 b_0 -\sum_{n=0}^{N-1} b_n \big(a_{n+1}-
a_n\big)$$
one has
\begin{align}
  - & \sum_{n=0}^{N-1}\sum_{ K\in \T_R} \int_{K}  \big( \beta(u_K^{n+1})-\beta(u_K^n)\big) \psi(t_n,x)\,dx \notag \\
  & = \sum_{ K\in \T_R} \int_{K} \beta(u_K^0)\psi(0,x)\,dx + \int_{\Delta t}^{T}\int_{\R^d} \beta(u_{\T,\,\Delta t}^h)
  \partial_t \psi(t-\Delta t,x)\,dx\,dt. \label{eq:a(h,delta t)}
\end{align}
 Let $R^{h,\Delta t}$ be the quantity as in Proposition \ref{prop:discrete entropy ineq}. Define 
\begin{align*}
 \mathcal{R}^{h,\Delta t}
=& R^{h,\Delta t} + \Bigg\{ \int_{\R^d} \beta(u_0(x))\psi(0,x)\,dx - \sum_{K\in \T_R}\int_{K} 
 \beta(u_K^0)\psi(0,x)\,dx\Bigg\} \notag \\
  + & \Bigg\{ \int_{\Pi_T} \beta(u_{\T,\, \Delta t}) \partial_t \psi(t,x)\,dt\,dx - 
 \int_{\Delta t}^{T}\int_{\R^d} \beta(u_{\T,\,\Delta t}^h) \partial_t \psi(t-\Delta t,x)\,dx\,dt\Bigg\} \notag \\
  + & \Bigg\{ \int_{\Pi_T} F^\beta(u_{\T,\,\Delta t}^h) \vec{v}(t,x)\cdot\grad_x \psi(t,x)\,dt\,dx -
 \sum_{n=0}^{N-1} \sum_{K\in \T_R} \int_{t_n}^{t_{n+1}} \int_{K} F^\beta(u_K^n)\vec{v}(t,x)\cdot \grad_x \psi(t_n,x)\,dx\,dt \Bigg\} \notag \\
  + & \Bigg\{ \int_{\R^d}\int_{0}^T \int_{\mathbf{E}} \int_0^1 \eta(u_{\T,\,\Delta t}^h;z) \beta^\prime \big(u_{\T,\,\Delta t}^h + \lambda 
   \eta(u_{\T,\,\Delta t}^h;z)\big)\psi(t,x)\,d\lambda \,\tilde{N}(dz,dt)\,dx \notag \\ 
   & \quad - \sum_{n=0}^{N-1} \sum_{K\in \T_R} \int_{t_n}^{t_{n+1}} \int_{\mathbf{E}} \int_{K}\int_0^1 \eta(u_K^n;z) \beta^\prime \big(u_K^n + \lambda 
   \eta(u_K^n;z)\big)\psi(t_n,x)\,d\lambda\,dx\,\tilde{N}(dz,dt) \Bigg\} \notag \\
  +& \Bigg\{ \int_{\R^d}\int_{0}^T \int_{\mathbf{E}}\int_0^1 (1-\lambda) \eta^2(u_{\T,\,\Delta t}^h;z) \beta^{\prime\prime} \big(u_{\T,\,\Delta t}^h + \lambda 
   \eta(u_{\T,\,\Delta t}^h;z)\big)\psi(t,x)\, d\lambda\,m(dz)\,dt\,dx \notag \\ 
    & - \sum_{n=0}^{N-1} \sum_{K\in \T_R} \int_{t_n}^{t_{n+1}} \int_{\mathbf{E}} \int_{K}\int_0^1 (1-\lambda) \eta^2(u_K^n;z) 
    \beta^{\prime\prime} \big(u_K^n + \lambda \eta(u_K^n;z)\big)\psi(t_n,x)\,d\lambda\,dx\,m(dz)\,dt \Bigg\} \notag \\
    & \equiv R^{h,\Delta t} + \mathcal{I}^{h,\Delta t} + \mathcal{T}^{h,\Delta t} + \mathcal{D}^{h,\Delta t} + 
    \mathcal{M}^{h,\Delta t} + \mathcal{A}^{h,\Delta t}.
\end{align*}
In view of Proposition \ref{prop:discrete entropy ineq} and the definition of $\mathcal{R}^{h,\Delta t}$ along with \eqref{eq:a(h,delta t)}, we note
that \eqref{eq: entropy inequality for approximate solution} holds. 

In order to prove the proposition, it remains to prove the convergence of the following quantities:
$\mathbb{E}\Big[{\bf 1}_B R^{h,\Delta t}\Big]$, $\mathbb{E}\Big[{\bf 1}_B \mathcal{I}^{h,\Delta t}\Big]$, $\mathbb{E}\Big[{\bf 1}_B \mathcal{T}^{h,\Delta t}\Big]$,
$\mathbb{E}\Big[{\bf 1}_B \mathcal{D}^{h,\Delta t}\Big]$, $\mathbb{E}\Big[{\bf 1}_B \mathcal{M}^{h,\Delta t}\Big]$ and 
$\mathbb{E}\Big[{\bf 1}_B \mathcal{A}^{h,\Delta t}\Big]$, where $B$ is any $\mathbb{P}$- measurable subset of $\Omega$.
\vspace{.1cm}

\noindent{1. \underline{\bf Convergence of $\mathbb{E}\big[ {\bf 1}_B\,\mathcal{I}^{h,\Delta t}\big]$:}} 
Note that, due to Lebesgue differentiation theorem, for almost all $x\in K$, 
$\big| u_0(x)- u_K^0\big| \longrightarrow 0$ as diameter of $K$ tends to zero  $(\text{i.e.,}\, h\goto 0)$. Now 
\begin{align*}
 \Big|\mathbb{E}\big[ {\bf 1}_B\,\mathcal{I}^{h,\Delta t}\big] \Big|& = \Big| \mathbb{E}\Big[ {\bf 1}_B \sum_{K\in \T_R}\int_{K} \big(\beta(u_0(x))-\beta(u_K^0)\big)
 \psi(x,0)\,dx \Big]\Big| \\
 & \le ||\beta^\prime||_{\infty} \mathbb{E} \Big[ \sum_{K\in \T_R}\int_{K} \big| u_0(x)- u_K^0\big|\psi(x,0)\,dx\Big],
\end{align*}
and hence $\mathbb{E}\big[ {\bf 1}_B\,\mathcal{I}^{h,\Delta t}\big]\longrightarrow 0$ as $h\goto 0$.
\vspace{.1cm}

\noindent{2. \underline{\bf Convergence of $\mathbb{E}\big[ {\bf 1}_B\,\mathcal{T}^{h,\Delta t}\big]$:}}  In view of Lemma \ref{lem:weak bv} and 
 the CFL condition \eqref{cond:CFL-1}, we obtain
\begin{align*}
 & \Big|\mathbb{E}\big[ {\bf 1}_B\,\mathcal{T}^{h,\Delta t}\big]\Big|\notag \\
   = & \Bigg| \mathbb{E}\Big[  {\bf 1}_B  \int_{0}^{\Delta t}\int_{\R^d} \beta(u_{\T,\Delta t}^h)\partial_t 
 \psi(x,t)\,dx\,dt \Big] + \mathbb{E}\Big[ {\bf 1}_B \int_{\Delta t}^{T} \int_{\R^d} \beta(u_{\T,\Delta t}^h) \big( \partial_t \psi(t,x)- \partial_t \psi(t-\Delta t,x)
 \big)\,dx\,dt \Big] \Bigg| \\
  \le & ||\beta^\prime||_{\infty}\, \Delta t \Big(  ||\partial_t \psi||_{\infty} \|u_{\T,\Delta t}^h\|_{L^\infty\big(0,T; L^1(\Omega \times
 B(0,R))\big)} +  ||\partial_{tt} \psi||_{\infty}\, \|u_{\T,\Delta t}^h\|_{L^1\big(\Omega \times B(0,R)\times [0,T)\big)}\Big),
\end{align*}
and hence $\mathbb{E}\big[ {\bf 1}_B\,\mathcal{T}^{h,\Delta t}\big]\longrightarrow 0$ as $h\goto 0$.
\vspace{.1cm}

\noindent{3. \underline{\bf Convergence of $\mathbb{E}\big[ {\bf 1}_B\,\mathcal{D}^{h,\Delta t}\big]$:}} In view of the assumption \ref{A3}, one has
\begin{align*}
 \Big| \mathbb{E}\big[ {\bf 1}_B\,\mathcal{D}^{h,\Delta t}\big] \Big|
 & \le V ||\grad_x \partial_t \psi||_{\infty} \Delta t \sum_{n=0}^{N-1}\sum_{K\in \T_R}\int_{K}\int_{t_n}^{t_{n+1}} \mathbb{E}\Big[\big| 
 F^\beta(u_K^n)\big|\Big] \,dx\,dt \\
 & \le V ||\grad_x \partial_t \psi||_{\infty} \Delta t ||\beta^\prime||_{\infty} c_f \sum_{n=0}^{N-1}\sum_{K\in \T_R}\int_{K}\int_{t_n}^{t_{n+1}}
 \mathbb{E}\big[|u_K^n|\big] \,dx\,dt \\
 & \le V ||\grad_x \partial_t \psi||_{\infty} \Delta t ||\beta^\prime||_{\infty} c_f 
 \|u_{\T,\Delta t}^h\|_{L^1 \big(\Omega \times B(0,R)\times [0,T)\big)}  \longrightarrow 0\,\,\text{as}\,\, h \goto 0.
\end{align*}
Thanks to Lemma \ref{lem:moment estimate} and the CFL condition 
\eqref{cond:CFL-1}, one can pass to the limit in the last line as well.
\vspace{.1cm}

\noindent{4. \underline{\bf Convergence of $\mathbb{E}\big[ {\bf 1}_B\,\mathcal{M}^{h,\Delta t}\big]$:}}
 By using Cauchy-Schwartz inequality, It\^{o}-L\'{e}vy isometry, the CFL condition \eqref{cond:CFL-1} and Lemma \ref{lem:moment estimate}, we obtain

 \begin{align*}
  \Big| E\big[ {\bf 1}_B\,\mathcal{M}^{h,\Delta t}\big]\Big|^2 
  & \le |B(0,R)| \Bigg(\sum_{n=0}^{N-1}\Bigg\{ \mathbb{E}\Big[ \sum_{K\in \T_R} \int_{t_n}^{t_{n+1}} \int_{\mathbf{E}} \int_{K}\int_0^1 \eta^2(u_K^n;z)
  {\beta^{\prime}}^2 \big(u_K^n + \lambda \eta(u_K^n;z)\big)  \\
  & \hspace{5cm} \times \big( \psi(t,x)-\psi(t_n,x)\big)^2 \,d\lambda\,dx\,m(dz)\,dt \Big]\Bigg\}^\frac{1}{2}\Bigg)^2 \\
  & \le C(R,\psi,c_\eta) \Delta t \Bigg(\sum_{n=0}^{N-1} \Delta t \Big( \sum_{K\in \T_R} |K| \mathbb{E}\big[(u_K^n)^2\big]\Big)^\frac{1}{2}\Bigg)^2 \\
  & \le C(R,\psi,c_\eta, T) \Delta t \|u_{\T,\Delta t}^h\|_{L^\infty(0,T;L^2(\Omega \times \R^d))}^2 
    \longrightarrow 0\,\,\text{as}\,\, h \goto 0.
  \end{align*}
 
 \vspace{.1cm}

\noindent{5. \underline{\bf Convergence of $\mathbb{E}\big[ {\bf 1}_B\,\mathcal{A}^{h,\Delta t}\big]$:}} Note that 
\begin{align*}
 \mathcal{A}^{h,\Delta t}& = \sum_{n=0}^{N-1} \sum_{K\in \T_R} \int_{t_n}^{t_{n+1}} \int_{\mathbf{E}} \int_{K}\int_0^1 (1-\lambda) \eta^2(u_K^n;z) 
    \beta^{\prime\prime} \big(u_K^n + \lambda \eta(u_K^n;z)\big) \\
    & \hspace{5cm} \times \Big\{ \psi(t,x)-\psi(t_n,x)\Big\}\,d\lambda\,dx\,m(dz)\,dt.
\end{align*}
Therefore, by \eqref{cond:CFL-1} and Lemma \ref{lem:moment estimate}, we obtain 
\begin{align*}
 \Big|\mathbb{E}\big[ {\bf 1}_B\,\mathcal{A}^{h,\Delta t}\big] \Big| & \le ||\beta^{\prime \prime}||_{\infty} \Delta t ||\partial_t \psi||_{\infty}
 \mathbb{E}\Big[ \sum_{n=0}^{N-1} \sum_{K\in \T_R} \int_{t_n}^{t_{n+1}} \int_{\mathbf{E}} \int_{K} \eta^2(u_K^n;z) \,dx\,m(dz)\,dt\Big] \\
 & \le C(\beta,\psi,c_\eta) \Delta t \sum_{n=0}^{N-1} \sum_{K\in \T_R}
 \int_{t_n}^{t_{n+1}} \int_{K} \mathbb{E}\big[(u_K^n)^2\big] \,dx\,dt
    \longrightarrow 0\,\,\text{as}\,\, h \goto 0.
\end{align*}
\vspace{.1cm}

\noindent{6. \underline{\bf Convergence of $\mathbb{E}\big[ {\bf 1}_B\,R^{h,\Delta t}\big]$:}}  Thanks to Proposition \ref{prop:discrete entropy ineq}, 
we have seen that $$\mathbb{E}\big[ {\bf 1}_B\,R^{h,\Delta t}\big] \longrightarrow 0\,\, \text{as} \,\,h \goto 0,$$
for any $\mathbb{P}$-measurable set $B$.
\end{proof}
\section{Proof of the main theorem} \label{sec:proof-maintheorem}
In this section, we establish the convergence of the scheme and hence existence of entropy solution to the underlying problem 
\eqref{eq:levy_con_laws}. Note that {\it a-priori} estimates on $u_{\T,\Delta t}^h (x,t)$ given by Lemma \ref{lem:moment estimate}
only guarantee weak compactness of the family $\{u_{\T,\Delta t}^h\}_{h>0}$, which is inadequate in view of the nonlinearities in the 
equation. The concept of Young measure theory is appropriate in this case. We now recapitulate the results we shall use from Young measure 
theory due to Dafermos \cite{dafermos} and Panov \cite{panov} for the deterministic setting, and Balder \cite{Balder} for the 
stochastic version of the theory.
\subsection{Young measure and convergence of approximate solutions} 
Roughly speaking a Young measure is a
parametrized family of probability measures where the parameters are drawn from a measure space. 
Let $(\Theta, \Sigma, \mu)$ be a $\sigma$-finite measure space and $\mathcal{P}(\R)$ be the space of probability measures on $\R$.

\begin{defi}[Young Measure]
A Young measure from $\Theta$ into $\R$ is a map $\tau \mapsto \mathcal{P}(\R)$ such that for any $\phi \in C_b(\R)$, 
$ \theta \mapsto \langle \tau(\theta), \phi\rangle: = \int_{\R} \phi(\xi) \tau(\theta)(d\xi)$ is measurable from $\Theta$ to $\R$.
The set of all Young measures from $\Theta$ into $\R$ is denoted by $\mathcal{R}(\Theta, \Sigma, \mu).$
\end{defi}

In this context, we mention that with an appropriate choice of $(\Theta, \Sigma, \mu)$, the family
$\{u_{\T,\Delta t}^h\}_{h>0}$
can be thought of as a family of Young measures. 
We are interested in finding a subsequences out of this family that ``converges" to a Young measure in a suitable sense. To this end, we consider the
  predictable $\sigma$-field of $\Omega\times(0,T)$ with respect to $\{\mathcal{F}_t\}$, denoted by $\mathcal{P}_T$, and set 
\begin{align*}
  \Theta = \Omega\times (0,T)\times \R^d,\quad \Sigma = \mathcal{P}_T \times \mathcal{L}(\R^d)\quad \text{and} \quad \mu= P\otimes \lambda_t\otimes \lambda_x,
\end{align*} where $\lambda_t$ and $\lambda_x$ are respectively the Lebesgue measures on $(0,T)$ and $\R^d$.
Moreover, for $M\in \mathbb{N}$, set
  $ \Theta_M = \Omega\times (0,T)\times B_M,$ where $B_M$ be the ball of radius $M$ around zero in $\R^d$. We sum up the necessary results in the following lemma to carry over 
the subsequent analysis. For a proof of this lemma, consult \cite{BaVaWit,BisKarlMaj}.
\begin{prop} \label{prop:young-measure}
  Let  $\{u_{\T,\,\Delta t}^h(t,x)\}_{{h}> 0}$ be a sequence of  $L^2(\R^d)$-valued predictable processes such that
  \eqref{estimate: moment-1} holds.Then there exists a subsequence $\{{h}_n\}$ with ${h}_n\goto 0$ and a Young measure
  $\tau\in \mathcal{R}(\Theta, \Sigma, \mu) $ such that the following hold:
  \begin{itemize}
   \item [(A)] If  $g(\theta,\xi)$ is a Carath\'{e}odory function on $\Theta\times \R$ such that $\mbox{supp}(g)\subset \Theta_M\times \R$
 for some $M \in \mathbb{N}$ and $\{g(\theta, u_{\T,\Delta t}^{h_n}(\theta))\}_n$ (where $\theta\equiv (\omega; t, x)$) is uniformly 
 integrable,  then 
 \begin{align*}
    \lim_{h_n\rightarrow 0} \int_{\Theta}g(\theta, u_{\T,\Delta t}^{h_n}(\theta))\,\mu(d\theta)
    = \int_\Theta\Big[\int_{\R} g(\theta, \xi)\tau(\theta)(\,d\xi)\Big]\,\mu(d \theta).
    \end{align*}
\item [(B)]  Denoting a triplet $(\omega,x,t)\in \Theta$ by $\theta$, we define 
\begin{align*}
 u(\theta,\alpha)=\inf\Big\{ c\in \R: \tau(\theta)\big((-\infty,c)\big)>\alpha \Big\}\quad \text{for}\quad \alpha \in (0,1)
 ~\text{and}~\theta\in \Theta.
\end{align*} 
 Then, $u(\theta,\alpha)$ is non-decreasing, right continuous on $(0,1)$ and
 $ \mathcal{P}_T\times \mathcal{L}(\R^d\times (0,1))$- measurable.
  Moreover, if $g(\theta,\xi)$ is a nonnegative Carath\'{e}odory function on $\Theta\times \R$, then 
  \begin{align*}
 \int_{\Theta}\Big[\int_\R g(\theta,\xi)\tau(\theta)(\,d\xi)\Big]\, \mu \,(d\theta)
 = \int_\Theta\int_{\alpha=0}^1 g(\theta, u(\theta,\alpha))\,d\alpha\, \mu(d\theta). 
 \end{align*}
\end{itemize}
\end{prop}
\subsection{Proof of the main theorem}\label{subsec:proof-main-thm}
Having all the necessary {\it a priori} bounds and entropy inequality on $u_{\T, \Delta t}^h$, we are now ready to prove the main theorem (cf. Main Theorem \ref{thm:maintheorem}). Here we mentioned that 
$u(\theta,\alpha)$ given by Proposition \ref{prop:young-measure} will serve as a possible generalized entropy solution to \eqref{eq:levy_con_laws} for the above
choice of the measure space $(\Theta, \Sigma, \mu)$. In view of \eqref{eq: entropy inequality for approximate solution}, we have for any $B\in \mathcal{F}_T$
  \begin{align}
    &  \mathbb{E}\Big[ {\bf 1}_{B} \int_{\Pi_T} \Big\{\beta(u_{\T,\,\Delta t}^h)\partial_t \psi(t,x) + 
  F^\beta(u_{\T,\,\Delta t}^h)\vec{v}(t,x)\cdot \grad_x \psi(t,x)\Big\}\,dt\,dx\Big] \notag \\
   & + \mathbb{E}\Big[ {\bf 1}_{B} \int_{\R^d}\int_{0}^T \int_{\mathbf{E}} \int_0^1 \eta(u_{\T,\,\Delta t}^h;z) \beta^\prime \big(u_{\T,\,\Delta t}^h + \lambda 
   \eta(u_{\T,\,\Delta t}^h;z)\big)\psi(t,x)\,d\lambda \,\tilde{N}(dz,dt)\,dx \Big]\notag \\
    & + \mathbb{E}\Big[ {\bf 1}_{B} \int_{\R^d}\int_{0}^T \int_{\mathbf{E}}\int_0^1 (1-\lambda) 
    \eta^2(u_{\T,\,\Delta t}^h;z) \beta^{\prime\prime} \big(u_{\T,\,\Delta t}^h + \lambda 
   \eta(u_{\T,\,\Delta t}^h;z)\big)\psi(t,x)\, d\lambda\,m(dz)\,dt\,dx \Big]\notag \\
   & +  \mathbb{E}\Big[ {\bf 1}_{B}\int_{\R^d} \beta(u_0(x))\psi(0,x)\,dx\Big] \ge \mathbb{E} \Big[ {\bf 1}_{B}\mathcal{R}^{h,\,\Delta t}\Big] \notag \\
   & \text{i.e.,} \qquad  \mathcal{T}_1 + \mathcal{T}_2 + \mathcal{T}_3 + \mathbb{E}\Big[ {\bf 1}_{B}\int_{\R^d} \beta(u_0(x))\psi(0,x)\,dx\Big] \ge \mathbb{E} 
   \Big[ {\bf 1}_{B}\mathcal{R}^{h,\,\Delta t}\Big]. \label{eq:final entropy inequality for approximate solution}
  \end{align}
  We would like to pass the limit in \eqref{eq:final entropy inequality for approximate solution} as $h$ approaches to 
  zero. To do this, here we use the technique of Young measure theory in stochastic setting. Let $(\Theta, \Sigma, \mu)$ be a $\sigma$-
  finite measure space as mentioned previously. 
Note that $ L^2(\Theta, \Sigma, \mu)$ is a closed subspace of the larger space $L^2\Big(0,T; L^2({(\Omega, \mathcal{F}_T)}, L^2(\R^d))\Big)$ and hence the weak convergence
in $ L^2\big(\Theta, \Sigma, \mu\big)$ would imply weak convergence
in $L^2\Big(0,T; L^2({(\Omega, \mathcal{F}_T)}, L^2(\R^d))\Big)$. Now, for any $B \in \mathcal{F}_T$, the functions 
${\bf 1}_B \partial_t \psi(t,x),\,{\bf 1}_B \partial_{x_i} \psi(t,x)$ and ${\bf 1}_B \psi(t,x)$ are all members of \\
$L^2\Big(0,T; L^2({(\Omega, \mathcal{F}_T)}, L^2(\R^d))\Big)$. Therefore, in view of Proposition \ref{prop:young-measure} and the above 
discussion, one has 
\begin{align}
  \lim_{{h}\rightarrow 0} \mathcal{T}_1=  &\lim_{{h}\rightarrow 0}  \mathbb{E}\Big[ {\bf 1}_{B} \int_{\Pi_T} \Big\{\beta(u_{\T,\,\Delta t}^h)\partial_t \psi(t,x) + 
  F^\beta(u_{\T,\,\Delta t}^h) \vec{v}(t,x)\cdot \grad_x \psi(t,x)\Big\}\,dt\,dx\Big] \notag \\
    =& \mathbb{E}\Big[ {\bf 1}_{B} \int_{\Pi_T} \int_0^1 \Big\{\beta(u(t,x,\alpha))\partial_t \psi(t,x) + 
  F^\beta(u(t,x,\alpha)) \vec{v}(t,x)\cdot \grad_x \psi(t,x)\Big\}\,d\alpha\,dt\,dx\Big]. \label{limit-partial-approximation}
 \end{align} 
 Next we want to pass to the limit in $\mathcal{T}_3$. For this, we fix $(\lambda,z)$, and define a Carath\'{e}odory function
 $$ G_{\lambda,z}(r,x,\omega,\xi)= {\bf 1}_{B}(\omega)(1-\lambda)\eta^2(\xi,z)\beta^{\prime\prime}\big(\xi + \lambda \eta(\xi,z)\big)\psi(r,x).$$ 
 Note that $\{G_{\lambda,z}(r,x,\omega,u_{\T,\,\Delta t}^{h_n}(r,x,\omega))\}_{n}$ is uniformly integrable in $L^1((\Theta,\Sigma,\mu);\R)$. Thus, in view of Proposition \ref{prop:young-measure} 
 we have, for fixed $(\lambda, z)\in (0,1)\times \mathbf{E}$
 \begin{align*}
  \lim_{h \rightarrow 0}  &\mathbb{E}\Big[ \int_{\Pi_T} {\bf 1}_{B}  (1-\lambda) 
    \eta^2(u_{\T,\,\Delta t}^h;z) \beta^{\prime\prime} \big(u_{\T,\,\Delta t}^h + \lambda 
   \eta(u_{\T,\,\Delta t}^h;z)\big)\psi(t,x)\,dt\,dx \Big]\notag \\
  =&  \mathbb{E}\Big[ \int_{\Pi_T}\int_{0}^1 {\bf 1}_{B}  (1-\lambda) \eta^2(u(t,x,\alpha);z) \beta^{\prime\prime} \big(u(t,x,\alpha) + \lambda 
   \eta( u(t,x,\alpha);z)\big) \psi(t,x)\,d\alpha\,dt\,dx \Big].
 \end{align*}
 Thanks to the assumption \ref{A3}, and Lemma \ref{lem:moment estimate}, we invoke dominated convergence theorem and have
 \begin{align}
  \lim_{{h}\rightarrow 0} \mathcal{T}_3=  & \lim_{h \rightarrow 0}  \mathbb{E}\Big[ {\bf 1}_{B} \int_{\Pi_T} \int_{\mathbf{E}}\int_0^1 (1-\lambda) 
    \eta^2(u_{\T,\,\Delta t}^h;z) \beta^{\prime\prime} \big(u_{\T,\,\Delta t}^h + \lambda 
   \eta(u_{\T,\,\Delta t}^h;z)\big)\psi(t,x)\, d\lambda\,m(dz)\,dt\,dx \Big]\notag \\
  = & \mathbb{E}\Big[{\bf 1}_{B} \int_{\Pi_T} \int_{\mathbf{E}}\int_0^1 \int_{0}^1 (1-\lambda) 
    \eta^2(u(t,x,\alpha);z) \beta^{\prime\prime} \big(u(t,x,\alpha) + \lambda 
   \eta(u(t,x,\alpha);z)\big) \notag \\
   & \hspace{4cm}\times \psi(t,x)\,d\alpha\, d\lambda\,m(dz)\,dt\,dx \Big].\label{limit-correction term-approximation}
 \end{align}
 Now passage to the limit in the martingale term requires some additional reasoning. Let $\Gamma = \Omega\times [0,T]\times \mathbf{E} $,
 $\mathcal{G}= \mathcal{P}_T \times \mathcal{L}(\mathbf{E})$ and
 $\varsigma = \mathbb{P}\otimes \lambda_t \otimes m(dz)$, where $\mathcal{L}(\mathbf{E})$ represents a Lebesgue $\sigma$- algebra on $\mathbf{E}$. 
The space $L^2\big((\Gamma, \mathcal{G}, \varsigma); \R\big)$ represents the space of square integrable predictable integrands
for It\^{o}-L\'{e}vy integrals with respect to the compensated 
Poisson random measure $\tilde{N}(dz,dt)$. Moreover, by It\^{o}-L\'{e}vy isometry and martingale representation theorem,
it follows that It\^{o}-L\'{e}vy integral defines isometry between two Hilbert spaces $L^2\big((\Gamma, \mathcal{G}, \varsigma); \R\big)$
and $L^2\big((\Omega, \mathcal{F}_T); \R\big)$. In other words, if $\mathcal{I}$ denotes the It\^{o}-L\'{e}vy integral operator, i.e., the application 
\begin{align*}
 \mathcal{I}: L^2\big((\Gamma, \mathcal{G}, \varsigma); \R\big) &\goto L^2\big((\Omega, \mathcal{F}_T); \R\big)\\
 v & \mapsto \int_0^T \int_{\mathbf{E}} v(\omega,z,r)\tilde{N}(dz,dr)
\end{align*}
 and
$\{X_n\}_n$ be sequence in $L^2\big((\Gamma, \mathcal{G}, \varsigma); \R\big)$ weakly converging to $X$; then $\mathcal{I}(X_n)$ will
converge weakly to $\mathcal{I}(X)$ in $L^2\big((\Omega, \mathcal{F}_T); \R\big)$. 
Note that, for fixed $z\in \mathbf{E}$, $G(t,x,\omega,\xi)=\Big(\beta\big(\xi+ \eta(\xi;z)\big)-\beta(\xi)\Big)\psi(t,x)$ is a Carath\'{e}odory 
function and $\{G(t,x,\omega,u_{\T,\,\Delta t}^{h_n}(t,x,\omega))\}_{n}$ is uniformly integrable in $L^1((\Theta,\Sigma,\mu);\R)$.
Therefore, one can apply Proposition \ref{prop:young-measure} and 
 conclude that for $m(dz)$-almost every $z\in \mathbf{E}$ and $g(t,z)\in L^2\big((\Gamma, \mathcal{G}, \varsigma); \R\big) $,
\begin{align*}
 \lim_{h \goto 0} & \mathbb{E}\Big[ \int_{0}^T \int_{\R^d} \Big(\beta\big(u_{\T,\,\Delta t}^h+ \eta(u_{\T,\,\Delta t}^h;z)\big) - \beta(u_{\T,\,\Delta t}^h)\Big)
 \psi(r,x)g(r,z)\,dx\,dr\Big] \notag \\
 &=  \mathbb{E}\Big[ \int_{0}^T \int_{\R^d}\int_{0}^1 \Big(\beta\big( u(r,x,\alpha)+ \eta(u(r,x,\alpha);z)\big)
 -\beta(u(r,x,\alpha))\Big)\psi(r,x)g(r,z)\,d\alpha\,dx\,dr\Big].
\end{align*}
 We apply dominated convergence theorem along with Lemma \ref{lem:moment estimate} and the assumption \ref{A3} to have
\begin{align*}
 &\lim_{h \goto 0}  \mathbb{E}\Big[ \int_{0}^T \int_{\mathbf{E}}\int_{\R^d} \Big(\beta\big(u_{\T,\,\Delta t}^h + \eta(u_{\T,\,\Delta t}^h;z)\big)
 - \beta(u_{\T,\,\Delta t}^h)\Big)\psi(r,x)h(r,z)\,dx\,m(dz)\,dr\Big] \notag \\
 &=  \mathbb{E}\Big[ \int_{0}^T \int_{\mathbf{E}} \Big\{ \int_{\R^d} \int_{0}^1 \Big(\beta\big(u(r,x,\alpha)+ \eta(u(r,x,\alpha);z)
 \big) - \beta(u(r,x,\alpha))\Big) \notag \\
 & \hspace{4cm} \times \psi(r,x)h(r,z)\,d\alpha\,dx \Big\} \,m(dz)\,dr\Big].
\end{align*}
Hence, if we denote 
\begin{align*}
 X_n(t,z)=  \int_{\R^d} \Big(\beta\big( u_{\T,\,\Delta t}^h+ \eta(u_{\T,\,\Delta t}^h;z)\big) - \beta(u_{\T,\,\Delta t}^h)\Big)\psi(t,x)\,dx
\end{align*}
and 
\begin{align*}
 X(t,z)=  \int_{\R^d}\int_{0}^1 \Big(\beta\big( u(t,x,\alpha)+ \eta(u(t,x,\alpha);z)\big)-\beta(u(t,x,\alpha))\Big)\psi(t,x)\,d\alpha\, dx
\end{align*}
then, $X_n$ converges to $X$ in $L^2\big((\Gamma, \mathcal{G}, \varsigma); \R\big)$ which implies, in view of the above discussion
\begin{align*}
\int_{0}^T \int_{\mathbf{E}} X_n(t,z) \tilde{N}(dz,dt) \rightharpoonup \int_0^T \int_{\mathbf{E}} X(t,z) \tilde{N}(dz,dt)\quad  \text{ in} \quad 
L^2\big((\Omega, \mathcal{F}_T); \R\big).
\end{align*}
In other words, since $B\in \mathcal{F}_T$, we obtain 
\begin{align}
   \lim_{{h}\rightarrow 0} \mathcal{T}_2&=  \lim_{{h}\rightarrow 0} \mathbb{E}\Big[ {\bf 1}_{B} \int_{\Pi_T} \int_{\mathbf{E}} \int_0^1 \eta(u_{\T,\,\Delta t}^h;z) \beta^\prime \big(u_{\T,\,\Delta t}^h + \lambda 
   \eta(u_{\T,\,\Delta t}^h;z)\big)\psi(t,x)\,d\lambda \,\tilde{N}(dz,dt)\,dx \Big]\notag \\
   & =  \mathbb{E}\Big[{\bf 1}_{B} \int_{\Pi_T} \int_{\mathbf{E}} \int_0^1 \int_0^1 \eta(u(t,x,\alpha);z) \beta^\prime \big(u(t,x,\alpha) + \lambda 
   \eta(u(t,x,\alpha);z)\big) \notag \\
   & \hspace{4cm}\times \psi(t,x)\,d\alpha\,d\lambda \,\tilde{N}(dz,dt)\,dx \Big]. \label{limit-stochastic term-approximation}
\end{align}
 By \eqref{limit-partial-approximation}, \eqref{limit-correction term-approximation} and  
\eqref{limit-stochastic term-approximation} and the fact that $ \mathbb{E}\Big[{\bf 1}_{B}\mathcal{R}^{h,\Delta t}\Big] \longrightarrow 0$ as $h \goto 0$
(cf. Proposition \ref{prop:entropy inequality on appro solution}),
one can pass to the limit in \eqref{eq:final entropy inequality for approximate solution} yielding \eqref{eq: generalised entropy inequality}. 
  Also, in view of Proposition \ref{prop:young-measure} and the uniform moment estimate \eqref{estimate: moment-1} along
 with Fatou's lemma, we have
 \[ \sup_{0\le t\le T} \mathbb{E}\Big[||u(t,\cdot,\cdot)||_2^2 \Big] < + \infty.\]
 This implies that $u(t,x,\alpha)$ is a generalized entropy solution of \eqref{eq:levy_con_laws}. Again, thanks to Theorem \ref{thm:uniqueness}, we conclude that
 $u(t,x,\alpha)$ is an independent function of variable $\alpha$ and $\bar{u}(t,x)=\int_0^1 u(t,x,\tau) d\tau = u(t,x,\alpha)$ (for almost all $\alpha$) is the unique stochastic entropy solution.
 Moreover, since $u_{\T,\,\Delta t}^h$ is bounded in $L^2(\Omega \times \Pi_T)$, we conclude that $u_{\T,\,\Delta t}^h$ converges to $\bar{u}$ in 
 $L_{\text{loc}}^p(\R^d; L^p (\Omega \times (0,T))$, for $1\le p < 2$. 
 This completes the proof.
\section{Appendix}\label{sec:appendix}
In this section, we study existence and uniqueness of entropy solution for the underlying problem \eqref{eq:levy_con_laws}. 
\subsection{Existence of weak solution for viscous problem}
Just as the deterministic problem, here also we study the corresponding regularized problem by adding a small diffusion operator
and derive some \textit{a priori} bounds. Due to the nonlinearity in equation, one cannot expect classical solution
and instead seeks a weak solution.
 \vspace{.1cm}  

For a small parameter $\eps >0$, we consider the following viscous approximation of \eqref{eq:levy_con_laws}
 \begin{align}
  du(t,x) + \mbox{div}_x (\vec{v}(t,x)f(u(t,x))) \,dt &= \int_{\mathbf{E}} \eta(u(t,x);z)\tilde{N}(dz,dt) + \eps \Delta u(t,x)\,dt,~~(t,x)\in \Pi_T \label{eq:levy_con_laws-viscous} \\
  u(0,x)&=u_0^\eps(x),~~ x\in \R^d, \notag 
\end{align}
where $u_0^{\eps}\in L^2(\R^d)$. To establish existence of a weak solution for \eqref{eq:levy_con_laws-viscous}, we follow \cite{BisMajVal,vallet2008} and use an implicit time discretization
scheme. Let $\Delta t= \frac{T}{N}$ for some fixed positive integer $N \ge 1$.
  Set $t_n= n\,\Delta t$ for $n=0,1,2\,\cdots, N$. Define
 \begin{align}
  \mathcal{N}= L^2(\Omega;H^1(\R^d)),\quad
  \mathcal{N}_n= \{ \text{the }\mathcal{F}_{n\Delta t}\text{ measurable elements of }\mathcal{N}\},\notag \\
  \mathcal{H}= L^2(\Omega;L^2(\R^d)),\quad
  \mathcal{H}_n= \{ \text{the }\mathcal{F}_{n\Delta t}\text{ measurable elements of }\mathcal{H}\}.\notag 
 \end{align}
 The following proposition holds.
 \begin{prop}
  Assume that $\Delta t$ is small with $ \displaystyle \Delta t < \frac{2 \eps}{V^2 c_f^2}$. Then, for any given $u_n \in \mathcal{H}_n$, there exists a
   unique $u_{n+1} \in \mathcal{N}_{n+1}$ such that $\mathbb{P}\text{-a.s.}$ for any $v \in H^1(\R^d)$, the following variational formula 
   holds:
   \begin{align}
    &\int_{\R^d} \Big((u_{n+1}-u_n)v + \Delta t \big\{ \eps \grad u_{n+1}\cdot \grad v -\vec{v}(t_n,x) f(u_{n+1})\cdot \grad v\big\} \Big)\,dx \notag \\
    &= \int_{\R^d} \int_{t_n}^{t_{n+1}} \int_{\mathbf{E}} \eta(u_n;z)\,v\, \tilde{N}(dz,ds)\,dx. \label{variational_formula_discrete}
   \end{align}
 \end{prop}
\begin{proof}
 Let us define a map 
 \begin{align*}
  T: &\mathcal{H}_{n+1} \mapsto  \mathcal{H}_{n+1} \\
  & S \mapsto u=T(S)
 \end{align*}
 via the variational problem in $\mathcal{N}_{n+1}$ : for all $v\in \mathcal{N}_{n+1} $
  \begin{align*}
    & \mathbb{E}\Big[\int_{\R^d} \Big((u-u_n)v + \Delta t \big\{ \eps \grad u \cdot \grad v -\vec{v}(t_n,x) f(S)\cdot \grad v\big\} \Big)\,dx \Big]\notag \\
    &= \mathbb{E}\Big[\int_{\R^d} \int_{t_n}^{t_{n+1}} \int_{\mathbf{E}} \eta(u_n;z)\,v\, \tilde{N}(dz,ds)\,dx\Big].
   \end{align*}
   Thanks to Lax-Milgram theorem, $T$ is a well-defined function. Moreover, for any $S_1, S_2 \in \mathcal{H}_{n+1}$, we see that 
   \begin{align*}
    & \mathbb{E}\Big[\int_{\R^d} |T(S_1)-T(S_2)|^2\,dx + \Delta t \eps \int_{\R^d} \big|\grad (T(S_1)-T(S_2))\big|^2\,dx \Big] \notag \\
    &=  \Delta t \mathbb{E}\Big[\int_{\R^d} \vec{v}(t_n,x) (f(S_1)-f(S_2))\cdot \grad (T(S_1)-T(S_2)) \,dx\Big]
   \end{align*}
   and hence, by Young's inequality and the assumptions \ref{A1} and \ref{A2}
    \begin{align*}
    & \mathbb{E}\Big[\int_{\R^d} |T(S_1)-T(S_2)|^2\,dx + \frac{\Delta t}{2} \eps \int_{\R^d} \big|\grad (T(S_1)-T(S_2))\big|^2\,dx \Big]\notag \\
    &\le  \frac{\Delta t}{2\,\eps} \mathbb{E}\Big[\int_{\R^d}   |\vec{v}(t_n,x)|^2 |f(S_1)-f(S_2)|^2 \,dx\Big] \le \frac{\Delta t\,V^2 c_f^2}{2\,\eps}
    \mathbb{E}\Big[\int_{\R^d} |S_1-S_2|^2\,dx\Big].
   \end{align*}
   Thus, if $\displaystyle \Delta t < \frac{2 \eps}{V^2 c_f^2}$, then $T$ is a contractive mapping in $\mathcal{H}_{n+1}$ which completes the proof. 
\end{proof}
\subsubsection{\textbf{\textit{A priori} estimate}}
 Note that, since $\mbox{div}_x \vec{v}(t,x)=0$ for all $(t,x)\in \Pi_T$, for any $\theta\in \mathcal{D}(\R^d)$, $\int_{\R^d} \vec{v}(t,x)f(\theta)\cdot \grad \theta\,dx=0$ and hence true for any $\theta\in H^1(\R^d)$ by density argument.
 We choose a test function $v= u_{n+1}$ in \eqref{variational_formula_discrete} and have
\begin{align}
 &\int_{\R^d} (u_{n+1}-u_n)u_{n+1}\,dx 
 + \eps\,\Delta t \int_{\R^d}|\grad u_{n+1}|^2\,dx
 = \int_{\R^d} \int_{t_n}^{t_{n+1}} \int_{\mathbf{E}} \eta(u_n;z) \, \tilde{N}(dz,ds)u_{n+1}\,dx \notag \\
 & \le \int_{\R^d} \int_{t_n}^{t_{n+1}} \int_{\mathbf{E}} \eta(u_n;z)u_n\, \tilde{N}(dz,ds)\,dx + \frac{\alpha}{2} ||u_{n+1}-u_n||_{L^2(\R^d)}^2 \notag \\
  & \qquad + \frac{1}{2\alpha} \int_{\R^d} \Big(\int_{t_n}^{t_{n+1}} \int_{\mathbf{E}} \eta(u_n;z)\, \tilde{N}(dz,ds)\Big)^2\,dx,\quad
  \text{for some}~~\alpha >0. \label{esti: discrete-0}
\end{align}
Therefore, thanks to the assumption \ref{A3}, and It\^{o}-L\'{e}vy isometry
\begin{align*}
 &\frac{1}{2} \Big[ ||u_{n+1}||_{\mathcal{H}}^2 + ||u_{n+1}-u_n||_{\mathcal{H}}^2 - ||u_n||_{\mathcal{H}}^2 \Big] +  \eps \,\Delta t ||\grad u_{n+1}||_{\mathcal{H}}^2 
 \le  \frac{\alpha}{2} ||u_{n+1}-u_n||_{\mathcal{H}}^2 + \frac{C\,\Delta t}{2\alpha} \big( 1+ ||u_n||_{\mathcal{H}}^2\big). 
\end{align*}
 Since $\alpha >0$ is arbitrary, one can choose $\alpha >0$ so that
\begin{align*}
  ||u_{n}||_{\mathcal{H}}^2 +  \sum_{k=0}^{n-1}||u_{k+1}-u_k||_{\mathcal{H}}^2  +  \eps \Delta t \sum_{k=0}^{n-1}
  ||\grad u_{k+1}||_{\mathcal{H}}^2  \le  C_1 + C_2\Delta t \sum_{k=0}^{n-1} ||u_{k}||_{\mathcal{H}}^2,
\end{align*}
for some constants $C_1, C_2 >0$. Hence an application of discrete Gronwall's lemma implies
\begin{align}
  ||u_{n}||_{\mathcal{H}}^2 +  \sum_{k=0}^{n-1}||u_{k+1}-u_k||_{\mathcal{H}}^2  +  \eps \Delta t \sum_{k=0}^{n-1}
  ||\grad u_{k+1}||_{\mathcal{H}}^2 \le  C. \label{a-prioriestimate:1}
\end{align}
For fixed $\Delta t = \frac{T}{N}$, we define 
\begin{equation*}
 u^{\Delta t} (t)= \sum_{k=1}^N u_k {\bf 1}_{[t_{k-1}, t_k )}(t); \quad 
  \tilde{u}^{\Delta t}(t)= \sum_{k=1}^N \Big[ \frac{u_k - u_{k-1}}{\Delta t}( t- t_{k-1}) + u_{k-1}\Big] 
  {\bf 1}_{[t_{k-1}, t_k )}(t)
 \end{equation*} 
 with $u^{\Delta t} (t)= u_0$ for $t<0$. Similarly, we define 
 \begin{align*}
 & \tilde{B}^{\Delta t}(t)= \sum_{k=1}^N \Big[ \frac{B_k - B_{k-1}}{\Delta t}( t- t_{k-1}) + B_{k-1}\Big] {\bf 1}_{[t_{k-1}, t_k )}(t), 
\end{align*}
where 
\begin{align*}
 B_n &= \sum_{k=0}^{n-1} \int_{t_k} ^{t_{k+1}} \int_{\mathbf{E}} \eta(u_k;z) \tilde{N}(dz,ds)
 = \int_0^{t_n} \int_{\mathbf{E}} \eta(u^{\Delta t}(s-\Delta t);z) \tilde{N}(dz,ds).
\end{align*}
A straightforward calculation shows that 
\begin{equation*}
 \begin{cases}
 \big\|u^{\Delta t}\big \|_{L^\infty(0,T;\mathcal{H})}= \underset{k=1,2,\cdots,N}\max\, \big\|u_k\big\|_{\mathcal{H}}; \quad
   \big\|\tilde{u}^{\Delta t}\big \|_{L^\infty(0,T;\mathcal{H})}= \underset{k=0,1,\cdots,N}\max\, \big\|u_k\big \|_{\mathcal{H}}, \\
   \big \| u^{\Delta t}-\tilde{u}^{\Delta t}\big\|_{L^2(0,T;\mathcal{H})}^2 \le \displaystyle \Delta t \sum_{k=0}^{N-1} \big \|u_{k+1}-u_k\big\|_{\mathcal{H}}^2.
  \end{cases}
\end{equation*}
In view of the above definitions and \textit{a priori} estimate \eqref{a-prioriestimate:1}, we have the following lemma.
\begin{lem} \label{lem: a-priori_bound_1}
Assume  that $\Delta t$ is small. Then  $u^{\Delta t},\, \tilde{u}^{\Delta t}$ are bounded sequences in $L^\infty(0,T;\mathcal{H})$;
 $\sqrt{\epsilon}u^{\Delta t}$ is a bounded sequence in $L^2(0,T;\mathcal{N})$ 
and $ || u^{\Delta t}-\tilde{u}^{\Delta t}||_{L^2(0,T;\mathcal{H})}^2
 \le C \Delta t$. Moreover, $u^{\Delta t}-u^{\Delta t}(\cdot -\Delta t)\goto 0$ in $L^2 (\Omega \times \Pi_T)$.
\end{lem}
Next, we want to find some upper bound for  $\tilde{B}^{\Delta t}(t)$. Regarding this, we have the following lemma.
\begin{lem}\label{lem: a-priori_bound_2}
$\tilde{B}^{\Delta t}$ is a bounded sequence in $L^2(\Omega \times \Pi_T)$ and 
\begin{align}
 \Big\| \tilde{B}^{\Delta t}(\cdot) - \int_{0}^\cdot \int_{\mathbf{E}}  \eta(u^{\Delta t}(s-\Delta t);z) 
 \tilde{N}(dz,ds)\Big\|_{L^2(\Omega \times \R^d)}^2 \le C \Delta t. \notag
\end{align}
\end{lem}
\begin{proof} First we prove the boundedness of $\tilde{B}^{\Delta t}(t)$. By using the definition of $\tilde{B}^{\Delta t}(t)$,
the assumption \ref{A3}, and the boundedness of $u^{\Delta t}$ in  $L^\infty(0,T;\mathcal{H})$ along with It\^{o}-L\'{e}vy isometry, we obtain
\begin{align*}
 \big\| \tilde{B}^{\Delta t} \big\|_{L^2\big(0,T;L^2(\Omega,L^2(\R^d))\big)}^2 &\le \Delta t \sum_{k=0}^N ||B_k||_{L^2(\Omega \times \R^d)}^2 \notag \\
 & \le \Delta t \sum_{k=0}^N  \mathbb{E} \Big[ \int_{\R^d} \Big| \int_{0}^{t_k}\int_{\mathbf{E}}  \eta(u^{\Delta t}(s-\Delta t);z) \tilde{N}(dz,ds) \Big|^2\,dx\Big] \notag \\
 & \le C \Delta t \sum_{k=0}^N  E \Big[ \int_{\R^d} \int_{0}^{t_k} |u^{\Delta t}(s-\Delta t)|^2\,dx\,ds\Big]
  \le C \|u^{\Delta t}\|_{L^\infty(0,T;L^2(\Omega \times \R^d))} \le C.
\end{align*}
Thus,  $\tilde{B}^{\Delta t}$ is a bounded sequence in $L^2(\Omega \times \Pi_T)$.
\vspace{.1cm}

To prove second part of the lemma, we see that for any $t\in \big[t_n,t_{n+1}\big)$,
\begin{align*}
  &\tilde{B}^{\Delta t}(t)-  \int_{0}^t \int_{\mathbf{E}}  \eta(u^{\Delta t}(s-\Delta t);z) \tilde{N}(dz,ds) \notag \\
   =& \frac{t-t_n}{\Delta t} \int_{t_n}^{t_{n+1}} \int_{\mathbf{E}} \eta(u_n;z) \tilde{N}(dz,ds)
  - \int_{t_n}^{ t} \int_{\mathbf{E}} \eta(u_n;z) \tilde{N}(dz,ds).
\end{align*}
Therefore, in view of \eqref{a-prioriestimate:1} and the assumption \ref{A3}, we have
\begin{align*}
 &  \Big\| \tilde{B}^{\Delta t}(t) - \int_{0}^t \int_{\mathbf{E}}  \eta(u^{\Delta t}(s-\Delta t);z) \tilde{N}(dz,ds)\Big\|_{L^2(\Omega \times \R^d)}^2 \notag \\
   \le & 2  \int_{\R^d} \mathbb{E}\Bigg[ \Big( \frac{t-t_n}{\Delta t}\Big)^2  \int_{t_n}^{t_{n+1}} \int_{\mathbf{E}} \eta^2(u_n;z)\,m(dz)\,ds 
    + \int_{t_n}^{ t} \int_{\mathbf{E}} \eta^2(u_n;z)\,m(dz)\,ds\Bigg]dx   \\
    \le & C ||u_n||_{\mathcal{H}}^2 \Big[ \frac{(t-t_n)^2}{\Delta t} + (t-t_n)\Big] \le  C\,\Delta t.
\end{align*} 
This completes the proof.
\end{proof}
\subsubsection{\bf Convergence of $u^{\Delta t}(t,x)$}
Thanks to Lemma \ref{lem: a-priori_bound_1} and Lipschitz property of $f$ and $\eta$, there exist $u, f_u$ and $\eta_u$
such that (up to a subsequence) 
\begin{equation}\label{convergence:weak-1}
\left\{\begin{array}{lcl}
 u^{\Delta t}\rightharpoonup^* u & \text{in} &  L^\infty\big(0,T;L^2(\Omega \times \R^d)\big)  \\
u^{\Delta t}\rightharpoonup u & \text{in}&  L^2\big((0,T)\times\Omega; H^1(\R^d)\big)  \quad\text{ (for fixed $\eps >0$)} \\
  f(u^{\Delta t}) \rightharpoonup f_u & \text{in} &  L^2\big((0,T)\times\Omega; H^1(\R^d)\big)\\
  \eta(u^{\Delta t}(\cdot-\Delta t);\cdot) \rightharpoonup \eta_u & \text{in} &  L^2\big(\Omega \times \Pi_T \times \mathbf{E} \big). 
  \end{array}
  \right.
\end{equation} 
Let $ \displaystyle v^{\Delta t}(t)=\sum_{k=1}^N \vec{v}(t_k,\cdot){\bf 1}_{[t_{k-1}, t_k )}(t)$. Then, for any $\theta \in H^1(\R^d)$, we can rewrite \eqref{variational_formula_discrete},
in terms of $u^{\Delta t}, \tilde{u}^{\Delta t}, \tilde{B}^{\Delta t} $
and $v^{\Delta t}$ as 

\begin{align}
  \Big\langle \frac{\partial}{\partial t}\big(\tilde{u}^{\Delta t}-\tilde{B}^{\Delta t}\big)(t), \theta \Big\rangle + \int_{\R^d}
  \big\{ \eps \grad u^{\Delta t}(t) - v^{\Delta t}(t) f(u^{\Delta t}(t))\big\}\cdot \grad \theta\,dx =0.
  \label{variational_formula_discrete_1}
\end{align}

 In view of \eqref{variational_formula_discrete_1}, one needs to show the boundedness of
 $\frac{\partial}{\partial t}(\tilde{u}^{\Delta t}-\tilde{B}^{\Delta t}) $ in 
$ L^2(\Omega\times (0,T);H^{-1}(\R^d))$ and then identify the weak limit. Regarding this, we have the following lemma.
\begin{lem}\label{lem:weak-stochastic-term}
The sequence $\Big\{\frac{\partial}{\partial t}(\tilde{u}^{\Delta t}-\tilde{B}^{\Delta t})(t)\Big\}$ is  bounded in $L^2\big(\Omega\times (0,T);H^{-1}
 (\R^d)\big)$, and 
 \begin{align*}
 \frac{\partial}{\partial t}(\tilde{u}^{\Delta t}-\tilde{B}^{\Delta t}) \rightharpoonup
 \frac{\partial}{\partial t}\Big( u-\int_{0}^\cdot \int_{\mathbf{E}} \eta_u\tilde{N}(dz,ds)\Big)
 \quad \text{in}~~L^2\big(\Omega\times (0,T);H^{-1}(\R^d)\big)
 \end{align*}
 where $u$ is given by \eqref{convergence:weak-1}. 
\end{lem}
 \begin{proof}
To prove the lemma, we use similar argumentation (cf.~ passage to the limit in $\mathcal{T}_2$) as in Section \ref{sec:proof-maintheorem}. Note that 
 It\^{o}-L\'{e}vy integral defines a linear operator from  $L^2\big((\Gamma, \mathcal{G}, \varsigma); \R\big)$ to 
$L^2\big((\Omega, \mathcal{F}_T); \R\big)$ and it preserves the norm (cf. for example \cite{peszat}). 
Therefore, in view of \eqref{convergence:weak-1} and  Lemma \ref{lem: a-priori_bound_2}, we have
 \begin{align*} 
  \tilde{B}^{\Delta t} \rightharpoonup \int_{0}^\cdot \int_{\mathbf{E}} \eta_u\tilde{N}(dz,ds) \quad \text{in}\,~~L^2(\Omega \times \Pi_T).
 \end{align*}
Again, note that
 \begin{align*}
 \frac{\partial}{\partial t}\big(\tilde{u}^{\Delta t}-\tilde{B}^{\Delta t}\big)(t)
 = \sum_{k=1}^N \frac{ (u_k -u_{k-1}) - (B_k -B_{k-1})}{\Delta t}{\bf 1}_{\big[t_{k-1}, t_k\big)}.
 \end{align*}
 From \eqref{variational_formula_discrete}, we see that for any $\theta\in H^1(\R^d)$,
 \begin{align*}
 & \int_{\R^d} \Big(\frac{u_{n+1}-u_n}{\Delta t} -\frac{1}{\Delta t} \int_{t_n}^{t_{n+1}} \int_{\mathbf{E}} \eta(u_n;z)
  \tilde{N}(dz,ds)\Big) \theta \,dx \notag \\
  &= -\eps \int_{\R^d} \grad u_{n+1}\cdot \grad \theta\,dx -\int_{\R^d} \vec{v}(t_n,\cdot) f(u_{n+1})\cdot \grad \theta\,dx \notag \\
  &\le  \Big\{ \eps \big\|\grad u_{n+1}\big\|_{L^2(\R^d)} +
  c_f V\,\big\|u_{n+1}\big\|_{L^2(\R^d)}\Big\} ||\theta||_{H^1(\R^d)},
 \end{align*} 
 and hence 
 \begin{align*}
  & \sup_{\theta \in H^1(\R^d)\setminus \{0\}} \frac{ \displaystyle \int_{\R^d} \Big( \frac{u_{n+1}-u_n}{\Delta t} -\frac{1}{\Delta t} \int_{t_n}^{t_{n+1}} \int_{\mathbf{E}} \eta(u_n;z)
  \tilde{N}(dz,ds)\Big)\theta \,dx }{ \displaystyle||\theta||_{H^1(\R^d)}} \notag \\
   & \hspace{2.5cm} \le  \eps \big\|\grad u_{n+1}\big\|_{L^2(\R^d)}
   + c_f V \big\|u_{n+1}\big\|_{L^2(\R^d)}.
 \end{align*}
 This implies that $ \displaystyle \frac{\partial}{\partial t}(\tilde{u}^{\Delta t}-\tilde{B}^{\Delta t})(t)$ is a bounded
 sequence in $L^2(\Omega\times (0,T);H^{-1}(\R^d))$.
 \vspace{.1cm}
 
 To prove the second part of the lemma, we recall that
  $ \displaystyle \tilde{B}^{\Delta t} \rightharpoonup \int_{0}^\cdot \int_{\mathbf{E}} \eta_u \tilde{N}(dz,ds)
  $  and $\tilde{u}^{\Delta t}  \rightharpoonup u $ in $L^2(\Omega \times \Pi_T)$. In view of the first part of this lemma, 
   one can conclude that, up to a subsequence 
 \begin{align*}
 \frac{\partial}{\partial t}(\tilde{u}^{\Delta t}-\tilde{B}^{\Delta t}) \rightharpoonup
 \frac{\partial}{\partial t}\Big( u-\int_{0}^\cdot \int_{\mathbf{E}} \eta_u \tilde{N}(dz,ds)\Big) \quad \text{in}~~L^2(\Omega\times (0,T);H^{-1}(\R^d)).
 \end{align*}
 This completes the proof.
 \end{proof}
 In view of \eqref{convergence:weak-1} and Lemma \ref{lem:weak-stochastic-term}, one can pass to the limit in \eqref{variational_formula_discrete_1} and has, for $\theta \in H^1(\R^d)$

 \begin{align*}
 \Big\langle \frac{\partial}{\partial t}\Big(u-\int_{0}^\cdot \int_{\mathbf{E}} \eta_u \tilde{N}(dz,ds)\Big), \theta \Big\rangle + \int_{\R^d}
  \big\{ \eps \grad u(t) - \vec{v}(t,\cdot) f_u\big\}\cdot \grad \theta \,dx =0.
\end{align*}
We denote by $\|\cdot\|_2$ the norm in $L^2(\R^d)$. An application of It\^{o}-L\'{e}vy formula \cite[similar to Theorem $3.4$]{tudor} to the functional $e^{-ct}\|u(t)\|_{2}^2$ yields
 \begin{align}
  & e^{-ct} \mathbb{E}\Big[ \|u(t)\|_2^2\Big] + 2 \eps \int_0^t  e^{-cs} \mathbb{E}\big[ \|\grad u(s)\|_2^2\big]\,ds
  - 2  \int_0^t  \mathbb{E}\Big[ \int_{\R^d} e^{-cs} \vec{v}(s,x)f_u \cdot \grad u\,dx \Big]\,ds \notag \\
  & = \mathbb{E}\big[ \|u_0\|_2^2\big] - c \int_0^t e^{-cs} \mathbb{E}\big[ \| u(s)\|_2^2\big]\,ds
  + \mathbb{E}\Big[\int_{\mathbf{E}} \int_0^t e^{-cs}\|\eta_u\|_2^2 \,ds\,m(dz)\Big].\label{esti:energy-1}
 \end{align}
By choosing $\alpha >0$ suitably in \eqref{esti: discrete-0} and multiplying by $e^{-ct_n}$ for positive $c>0$, we have 
\begin{align}
 & \mathbb{E}\Big[ \int_{\R^d} \Big( e^{-ct_n}|u_{n+1}|^2 - e^{-ct_{n-1}}|u_{n}|^2 \Big)\,dx \Big] + 2\eps\,\Delta t e^{-ct_n} \mathbb{E}\Big[\int_{\R^d} |\grad u_{n+1}|^2\,dx\Big] \notag \\
 & \le \Delta t e^{-ct_n} \mathbb{E}\Big[ \int_{\mathbf{E}} \int_{\R^d} \eta^2(u_n;z)\,dx\,m(dz)\Big] + \Big( e^{-ct_n}- e^{-ct_{n-1}} \Big)
 \mathbb{E}\Big[\int_{\R^d} | u_{n}|^2\,dx\Big]. \label{esti:energy-discrete-1}
\end{align}
Therefore, by summing over $n$ from $0$ to $k$ in \eqref{esti:energy-discrete-1} 
we get 
\begin{align*}
 & e^{-ct_k} \mathbb{E}\big[\|u_{k+1}\|_2^2\big] + 2\eps \sum_{n=0}^k \Delta t e^{-ct_n} \mathbb{E}\Big[ \|\grad u_{n+1}\|_2^2\Big] \notag \\
 & \le e^{c\Delta t} \mathbb{E}\big[ \|u_0\|_2^2\big] + \Delta t \sum_{n=0}^k  e^{-ct_n} \mathbb{E}\Big[ \int_{\mathbf{E}} \int_{\R^d} \eta^2(u_n;z)\,dx\,m(dz)\Big] + 
 \sum_{n=0}^k \Big( e^{-ct_n}- e^{-ct_{n-1}} \Big)\mathbb{E}\Big[\| u_{n}\|_2^2\Big].
\end{align*}
Note that 
\begin{align*}
 &\sum_{n=0}^k \Big( e^{-ct_n}- e^{-ct_{n-1}} \Big)\mathbb{E}\Big[\| u_{n}\|_2^2\Big] \\
 &= \big( 1- e^{c\Delta t}\big)\mathbb{E}\big[ \|u_0\|_2^2\big] + \sum_{n=1}^k \Big( e^{-ct_n}- e^{-ct_{n-1}} \Big)\mathbb{E}\Big[\| u_{n}\|_2^2\Big] \\
 &= \big( 1- e^{c\Delta t}\big)\mathbb{E}\big[ \|u_0\|_2^2\big] -c \sum_{n=1}^k \int_{t_{n-1}}^{t_n} e^{-cs}\,ds \mathbb{E}\Big[\| u_{n}\|_2^2\Big] \\
 & \le \big( 1- e^{c\Delta t}\big)\mathbb{E}\big[ \|u_0\|_2^2\big] - c e^{-c\Delta t} \int_0^{t_k} e^{-cs} \mathbb{E}\big[\| u^{\Delta t}(s)\|_2^2\big]\,ds,
\end{align*}
and 
\begin{align*}
& \Delta t \sum_{n=0}^k e^{-ct_n} \mathbb{E}\Big[ \int_{\mathbf{E}} \int_{\R^d} \eta^2(u_n;z)\,dx\,m(dz)\Big]
 \le \int_0^{t_{k}} e^{-cs}\mathbb{E} \Big[ \int_{\mathbf{E}} \int_{\R^d} \eta^2(u^{\Delta t};z)\,dx\,m(dz)\Big]\,ds.
\end{align*}
Thus, we obtain, for $t\in [t_k,t_{k+1})$
\begin{align}
& e^{-ct} \mathbb{E}\big[ \|u^{\Delta t}(t)\|_2^2\big] + 2 \eps \int_0^{t} e^{-cs} \mathbb{E}\big[ \|\grad u^{\Delta t}\|_2^2\big]\,ds  \notag \\
 & \le  \mathbb{E}\big[ \|u_0\|_2^2\big] + \int_0^{t} e^{-cs}\mathbb{E}\Big[ \int_{\mathbf{E}} \int_{\R^d} \eta^2(u^{\Delta t};z)\,dx\,m(dz)\Big]\,ds \notag \\
 & \hspace{2cm} -c e^{-c \Delta t} \int_0^{t} e^{-cs} \mathbb{E}\big[ \| u^{\Delta t}\|_2^2\big]\,ds.\label{esti:energy-discrete-3}
\end{align}
Note that, for any $\theta\in H^1(\R^d)$ and any $s\in [0,T]$, there holds $\displaystyle \int_{\R^d} \vec{v}(s,x) f(\theta)\grad \theta\,dx=0$. Thus, using \eqref{esti:energy-discrete-3} we obtain 
\begin{align}
 & e^{-ct} \mathbb{E}\big[ \|u^{\Delta t}(t)\|_2^2\big] + 2 \eps \int_0^{t} e^{-cs} \mathbb{E}\big[ \|\grad (u^{\Delta t}-u)\|_2^2\big]\,ds \notag \\
 & \hspace{1cm}- 2 \int_0^{t} e^{-cs} \mathbb{E}\big[\int_{\R^d}  \vec{v}(s,x)[f(u^{\Delta t})-f(u)]\grad (u^{\Delta t}-u)\,dx\big]\,ds \notag \\
 & \le \mathbb{E}\big[ \|u_0\|_2^2\big] + \int_0^{t} e^{-cs} \mathbb{E}\Big[ \int_{\mathbf{E}} \|\eta(u^{\Delta t};z)-\eta(u;z)\|^2\,m(dz)\Big]\,ds \notag \\
 & - \int_0^{t} e^{-cs}\mathbb{E}\Big[ \int_{\mathbf{E}} \|\eta(u;z)\|^2\,m(dz)\Big]\,ds + 2 \int_0^{t} e^{-cs}\mathbb{E}\Big[ \int_{\mathbf{E}} \int_{\R^d} \eta(u^{\Delta t};z)\eta(u;z)\,dx\,m(dz)\Big]\,ds \notag \\
 & -c e^{-c \Delta t} \int_0^{t} e^{-cs} \mathbb{E}\big[ \| u^{\Delta t}-u\|_2^2\big]\,ds + c e^{-c \Delta t} \int_0^{t} e^{-cs} \mathbb{E}\big[ \|u\|_2^2\big]\,ds 
 + 2 \eps \int_0^{t} e^{-cs} \mathbb{E}\big[ \|\grad u\|_2^2\big]\,ds \notag \\
 & -2c e^{-c \Delta t} \int_0^{t} e^{-cs} \mathbb{E}\big[  \int_{\R^d} u^{\Delta t}u\,dx\big]\,ds 
 +  2 \int_0^{t} e^{-cs} \mathbb{E}\big[\int_{\R^d} \vec{v}(s,x)f(u^{\Delta t})\grad u\,dx\big]\,ds \notag \\
 & + 2 \int_0^{t} e^{-cs} \mathbb{E}\big[\int_{\R^d} \vec{v}(s,x)f(u)\grad u^{\Delta t}\,dx\big]\,ds -4\eps \int_0^{t} e^{-cs} \mathbb{E}\big[  \int_{\R^d} \grad u^{\Delta t} \grad u\,dx\big]\,ds. 
 \label{esti:energy-discrete-4}
\end{align}
In view of Young's inequality, one has 
\begin{align}
 & - 2 \eps \int_0^{t} e^{-cs} \mathbb{E}\big[ \|\grad (u^{\Delta t}-u)\|_2^2\big]\,ds +  2 \int_0^{t} e^{-cs} \mathbb{E}\big[\int_{\R^d} \vec{v}(s,x)[f(u^{\Delta t})-f(u)]\grad (u^{\Delta t}-u)\,dx\big]\,ds \notag \\
 & \le -\eps \int_0^{t} e^{-cs} \mathbb{E}\big[ \|\grad (u^{\Delta t}-u)\|_2^2\big]\,ds
 + \frac{1}{\eps}\int_0^{t} e^{-cs} \mathbb{E}\big[\|\vec{v}(s,\cdot)[f(u^{\Delta t})-f(u)]\|_2^2\big]\,ds,\label{esti:***1}
\end{align}
and by choosing $c>0$ with $\frac{1}{\eps} V^2 C_f^2 + c_{\eta}\le c e^{-c \Delta t}$, one arrive at 
\begin{align}
 & \frac{1}{\eps}\int_0^{t} e^{-cs} \mathbb{E}\big[\|\vec{v}(s,\cdot)[f(u^{\Delta t})-f(u)]\|_2^2\big]\,ds + 
  \int_0^{t} e^{-cs}\mathbb{E} \Big[ \int_{\mathbf{E}} \|\eta(u^{\Delta t};z)-\eta(u;z)\|_2^2\,m(dz)\Big]\,ds \notag \\
  & \hspace{1cm}-c e^{-c \Delta t} \int_0^{t} e^{-cs} \mathbb{E}\big[ \| u^{\Delta t}-u\|_2^2\big]\,ds \le 0. \label{esti:***2}
\end{align}
We use \eqref{esti:***1}-\eqref{esti:***2} in \eqref{esti:energy-discrete-4} for the above choice of $c>0$ along with 
 \eqref{convergence:weak-1} and \eqref{esti:energy-1}  to have 
\begin{align*}
 & \limsup_{\Delta t} \int_0^T e^{-ct} \mathbb{E}\big[ \|u^{\Delta t}(t)\|_2^2\big]\,dt
 + \int_0^T\int_0^{t} e^{-cs}\mathbb{E}\Big[ \int_{\mathbf{E}} \| \eta_u- \eta(u;z)\|_2^2\,m(dz)\Big]\,ds\,dt\notag \\
 & \le \int_0^T e^{-ct} \mathbb{E}\big[ \|u(t)\|_2^2\big]\,dt.
\end{align*}
Thus, we obtain $\eta_u=\eta(u;z)$ and $u^{\Delta t}\rightarrow u$ in $L^2(\Omega \times \Pi_T)$. Moreover, one can show that $f_u=f(u)$. 
Thus $u$ is a weak solution to the viscous 
problem \eqref{eq:levy_con_laws-viscous}. Since it depends on $\eps>0$, we denote it by $u_\eps$. 
\subsubsection{ \bf \textit{A priori} bounds for viscous solutions}
Note that for fixed $\eps>0$, there exists a weak solution $u_\eps \in H^1(\R^d)$ satisfying: $\mathbb{P}$-a.s., and for a.e. $t\in (0,T)$
\begin{align}
& \Big\langle \frac{\partial}{\partial t} [u_\eps-\int_0^t \int_{\mathbf{E}} \eta(u_\eps(s,\cdot);z) \tilde{N}(dz,ds)],v \Big\rangle 
 + \int_{\R^d}\Big\{\vec{v}(t,x) f(u_\eps(t,x)) + \eps \grad u_\eps (t,x)\Big\} \cdot\grad v(x)\,dx=0,\label{eq:weak-viscous}
\end{align}
for any $v\in H^1(\R^d)$. We apply It\^{o}-L\'{e}vy formula to $\beta(u)=\|u\|_2^2$, and then take expectation. The result is 
\begin{align*}
  \mathbb{E}\Big[ \big\|u_\eps(t)\big\|_2^2\Big] + 2 \eps \int_0^t \mathbb{E}\Big[\big\|\grad u_\eps\big\|_2^2\Big]\,ds 
  \le  \mathbb{E}\Big[ \big\|u_\eps(0)\big\|_2^2  + C \int_0^t \mathbb{E}\Big[\big\|u_\eps(s)\big\|_2^2\Big]\,ds.
\end{align*}
An application of Gronwall's inequality yields
\begin{align*}
 \sup_{0\le t\le T}   \mathbb{E}\Big[\big\|u_\eps(t)\big\|_2^2\Big]  + \eps \int_0^T  \mathbb{E}\Big[\big\|\grad u_\eps(s)\big\|_2^2\Big]\,ds \le C.
\end{align*}

The following lemma states that $\displaystyle \frac{\partial}{\partial t} \Big[u_\eps-\int_0^t \int_{\mathbf{E}} \eta(u_\eps;z)
\tilde{N}(dz,ds)\Big] \in  L^2(\Omega\times \Pi_T) $ if the initial data $u_0^\eps \in H^1(\R^d)$.
\begin{lem}
Suppose that $u_0^\eps \in H^1(\R^d)$. Then, a weak solution $u_\eps$ of \eqref{eq:levy_con_laws-viscous} satisfies the following regularity properties:
 $ \frac{\partial}{\partial t} \big[u_\eps-\int_0^t \int_{\mathbf{E}} \eta(u_\eps;z) \tilde{N}(dz,ds)\big],\, \Delta u_\eps \in L^2(\Omega\times \Pi_T)$.
\end{lem}
\begin{proof}
Let $u_0^\eps \in H^1(\R^d)$. By choosing $v= u_{n+1}-u_n - \int_{t_n}^{t_{n+1}}\int_{\mathbf{E}}\eta(u_n;z)\tilde{N}(dz,ds)$ in \eqref{variational_formula_discrete}, we obtain 
\begin{align*}
& \big\| u_{n+1}-u_n - \int_{t_n}^{t_{n+1}}\int_{\mathbf{E}}\eta(u_n;z)\tilde{N}(dz,ds)\big\|_{L^2(\R^d)}^2 \\
& \hspace{.2cm}+ \Delta t \eps \int_{\R^d} \grad u_{n+1}\cdot \grad\big[u_{n+1}-u_n - \int_{t_n}^{t_{n+1}}\int_{\mathbf{E}}\eta(u_n;z)\tilde{N}(dz,ds)\big]\,dx \\
&= -\Delta t \int_{\R^d} \vec{v}(t_n,x)f^{\prime}(u_{n+1})\cdot \grad u_{n+1}\Big( u_{n+1}-u_n - \int_{t_n}^{t_{n+1}}\int_{\mathbf{E}}\eta(u_n;z)\tilde{N}(dz,ds)\Big)\,dx \\
& \le \frac{1}{2} \big\| u_{n+1}-u_n - \int_{t_n}^{t_{n+1}}\int_{\mathbf{E}}\eta(u_n;z)\tilde{N}(dz,ds)\big\|_{L^2(\R^d)}^2 
+ \frac{1}{2} C(V,f^\prime) (\Delta t)^2 \|\grad u_{n+1}\|_{L^2(\R^d)^d}^2.
\end{align*}
Note that, $\mathbb{E}\Big[\grad u_n \int_{t_n}^{t_{n+1}}\int_{\mathbf{E}} \grad\eta(u_n;z)\tilde{N}(dz,ds)\Big]=0$. Since $\tilde{N}$ is a compensated Poisson random measure, an application of differentiation under integral sign, the assumption \ref{A3} along with Young's inequality and It\^{o}-L\'{e}vy isometry reveals that 
\begin{align*}
& \mathbb{E}\Big[  \int_{\R^d} \grad u_{n+1}\cdot \grad\big[u_{n+1}-u_n - \int_{t_n}^{t_{n+1}}\int_{\mathbf{E}}\eta(u_n;z)\tilde{N}(dz,ds)\big]\,dx\Big] \\
&= \frac{1}{2}\mathbb{E}\Big[ \|\grad u_{n+1}\|_{L^2(\R^d)^d}^2- \|\grad u_{n}\|_{L^2(\R^d)^d}^2 + \|\grad[u_{n+1}-u_n]\|_{L^2(\R^d)^d}^2\Big] \\
& \hspace{1cm}- \mathbb{E} \Big[ \int_{\R^d} \grad[u_{n+1}-u_n]\int_{t_n}^{t_{n+1}}\int_{\mathbf{E}} \grad \eta(u_n;z)\tilde{N}(dz,ds)\,dx\Big] \\
& \ge \frac{1}{2}\mathbb{E}\Big[ \|\grad u_{n+1}\|_{L^2(\R^d)^d}^2- \|\grad u_{n}\|_{L^2(\R^d)^d}^2 + \frac{1}{2} \|\grad[u_{n+1}-u_n]\|_{L^2(\R^d)^d}^2 
-2 \lambda^* \Delta t \|\grad u_{n}\|_{L^2(\R^d)^d}^2 \int_{\mathbf{E}}h_1^2(z)\,m(dz) \Big],
\end{align*}
and hence
\begin{align*}
& \mathbb{E}\Big[ \big\|u_{n+1}-u_n-\int_{t_n}^{t_{n+1}}\int_{\mathbf{E}}\eta(u_n;z)\tilde{N}(dz,ds)\big\|_{L^2(\R^d)}^2\Big] \\
& + \Delta t \eps \,\mathbb{E}\Big[ \|\grad u_{n+1}\|_{L^2(\R^d)^d}^2- \|\grad u_{n}\|_{L^2(\R^d)^d}^2 + \frac{1}{2} \|\grad[u_{n+1}-u_n]\|_{L^2(\R^d)^d}^2 \Big] \\
& \le 2 \lambda^* (\Delta t)^2 c_{\eta}\mathbb{E}\Big[ \|\grad u_{n}\|_{L^2(\R^d)^d}^2\Big]
+ C(V,f^\prime) (\Delta t)^2 \mathbb{E}\Big[\|\grad u_{n+1}\|_{L^2(\R^d)^d}^2\Big].
\end{align*}
Thus, for any $k\in \{0,1,\cdots, N-1\}$, 
\begin{align*}
& \sum_{n=0}^k \Delta t \mathbb{E}\Big[ \big\| \frac{u_{n+1}-u_n-\int_{t_n}^{t_{n+1}}\int_{\mathbf{E}}\eta(u_n;z)\tilde{N}(dz,ds)}{\Delta t}\big\|_{L^2(\R^d)}^2\Big] \\
& + \eps \mathbb{E}\Big[ \|\grad u_{k+1}\|_{L^2(\R^d)^d}^2\Big] + 
\frac{\eps}{2} \sum_{n=0}^{k}\mathbb{E}\Big[\|\grad[u_{n+1}-u_n]\|_{L^2(\R^d)^d}^2 \Big] 
\\
& \le  \eps \mathbb{E}\Big[ \|\grad u_{0}^\eps\|_{L^2(\R^d)^d}^2\Big] + 
C(V,f^\prime, c_{\eta},\lambda^*) \Delta t \sum_{n=0}^{k+1} \mathbb{E}\Big[\|\grad u_n\|_{L^2(\R^d)^d}^2 \Big] \le C.
\end{align*}
Therefore, in view of the definitions of $u^{\Delta t},~\tilde{u}^{\Delta t},~ \tilde{B}^{\Delta t}$, we see that $u^{\Delta t},~\tilde{u}^{\Delta t}$ are bounded in \\
$L^\infty(0,T; L^2(\Omega;H^1(\R^d)))$, and the sequence $\Big\{\frac{\partial}{\partial t}(\tilde{u}^{\Delta t}-\tilde{B}^{\Delta t})(t)\Big\}$ is bounded in $L^2\big(\Omega\times (0,T); L^2(\R^d)\big)$. Moreover, second part of Lemma \ref{lem:weak-stochastic-term} reveals that $\displaystyle \frac{\partial}{\partial t} \Big[u_\eps-\int_0^t \int_{\mathbf{E}} \eta(u_\eps;z)
\tilde{N}(dz,ds)\Big]\in L^2(\Omega\times \Pi_T)$ and hence by using the equation \eqref{eq:levy_con_laws-viscous} we arrive at the conclusion that
$\Delta u_\eps \in L^2(\Omega\times \Pi_T)$. Furthermore, \eqref{eq:weak-viscous} holds with an integral over $\R^d$ instead of the duality bracket if the initial data $u_0^\eps \in H^1(\R^d)$. 
\end{proof}
 In addition, if $u_0^\eps\in L^{2p}(\R^d),\, p\ge 1$, then a straightforward argumentation
as in the proof of \cite[Proposition $A.5$]{BaVaWit} gives 
$u_\eps \in L^\infty \big(0,T; L^{2p}(\Omega \times \R^d)\big)$.
\vspace{.1cm}

The achieved results can be summarized into the following theorem.
\begin{thm}
 Let $\eps>0$ is fixed and $u_0^\eps \in H^1(\R^d)$. Then there exists a weak solution $u_\eps$ of \eqref{eq:levy_con_laws-viscous} such that 
 $\displaystyle \frac{\partial}{\partial t} \Big[u_\eps-\int_0^t \int_{\mathbf{E}} \eta(u_\eps;z) \tilde{N}(dz,ds)\Big],\, \Delta u_\eps \in L^2(\Omega\times \Pi_T)$.
  Moreover the following estimate holds:
  \begin{align*}
   \sup_{0\le t\le T} \mathbb{E}\Big[\big\|u_\eps(t)\big\|_2^2\Big]  + \eps \int_0^T \mathbb{E}\Big[\big\|\grad u_\eps(s)\big\|_2^2\Big]\,ds \le C.
  \end{align*} 
Furthermore, if $u_0^\eps\in L^{2p}(\R^d),\, p\ge 1$, then $u_\eps \in L^\infty \big(0,T; L^{2p}(\Omega \times \R^d)\big)$.
 \end{thm}
\subsection{Proof of Theorem \ref{thm:existence}}
In this subsection, we prove existence of generalized entropy solution in the sense of Definition \ref{defi:generality entropy solution}. For this, fix a nonnegative test
function $ \psi\in C_c^\infty([0, \infty)\times \R^d)$, $B\in \mathcal{F}_T$ and convex entropy flux pair $(\beta,\zeta)$. For any $\eps>0$, we consider the 
viscous problem \eqref{eq:levy_con_laws-viscous} with initial data $u_0^\eps \in \mathcal{D}(\R^d)$. We apply It\^{o}-L\'{e}vy formula to the functional
$F(t,u_\eps)= \int_{\R^d} \beta(u_\eps)\psi(t,x)\,dx$ and conclude
 \begin{align}
   0 \le  & \mathbb{E} \Big[ {\bf 1}_{B} \int_{\R^d} \beta(u_0^\eps(x))\psi(0,x)\,dx\Big] 
   -\eps \mathbb{E}  \Big[{\bf 1}_{B}\int_{\Pi_T}\beta^\prime(u_\eps(t,x))\grad u_\eps(t,x) \cdot \grad \psi(t,x)\,dx\,dt\Big]\notag \\
 +&  \mathbb{E}  \Big[ {\bf 1}_{B}  \int_{\Pi_T} \Big(\beta(u_\eps(t,x)) \partial_t\psi(t,x)
  + \grad \psi(t,x)\cdot \vec{v}(t,x) \zeta(u_\eps(t,x)) \Big)\,dx\,dt\Big]\notag \\
 + &  \mathbb{E}  \Big[ {\bf 1}_{B} \int_{\Pi_T} \int_{\mathbf{E}} \int_0^1 \eta(u_\eps(t,x);z)\beta^\prime (u_\eps(t,x) + \theta\,\eta(u_\eps(t,x);z))
 \psi(t,x)\,d\theta\,\tilde{N}(dz,dt)\,dx \Big] \notag \\
+&  \mathbb{E}  \Big[ {\bf 1}_{B} \int_{\Pi_T} \int_{\mathbf{E}} \int_0^1  (1-\theta)\eta^2(u_\eps(t,x);z)\beta^{\prime\prime}(u_\eps(t,x)
+ \theta\,\eta(u_\eps(t,x);z))
 \psi(t,x)\,d\theta\,m(dz)\,dt\,dx \Big] \label{viscous-measure-inequality}
\end{align}
We use Young measure technique (cf.~Subsection \ref{subsec:proof-main-thm}) to pass to the limit in \eqref{viscous-measure-inequality}
as $\eps \goto 0$. Moreover, there exists a $L^2(\R^d\times(0,1))$-valued 
 predictable limit process $u\in L^\infty\big(0,T; L^2(\Omega\times\R^d\times(0,1))\big)$ such that
\begin{align}
    &  \mathbb{E}\Big[ {\bf 1}_{B}  \int_{\Pi_T}\int_0^1 \Big(\beta(u(t,x,\alpha)) \partial_t\psi(t,x)
  + \grad \psi(t,x)\cdot \vec{v}(t,x) F^\beta(u(t,x,\alpha)) \Big)\,d\alpha\,dx\,dt\Big]\notag \\
 + &  \mathbb{E} \Big[ {\bf 1}_{B} \int_{\Pi_T}\int_{\mathbf{E}} \int_0^1\int_0^1 \eta(u(t,x,\alpha);z)\beta^\prime (u(t,x,\alpha)
 + \theta\,\eta(u(t,x,\alpha);z))\psi(t,x)\,d\alpha\,d\theta\,\tilde{N}(dz,dt)\,dx \Big] \notag \\
+&  \mathbb{E} \Big[ {\bf 1}_{B} \int_{\Pi_T} \int_{\mathbf{E}} \int_0^1 \int_0^1 (1-\theta)\eta^2(u(t,x,\alpha);z)\beta^{\prime\prime}(u(t,x,\alpha)
+ \theta\,\eta(u(t,x,\alpha);z))\psi(t,x)\,d\alpha\,d\theta\,m(dz)\,dt\,dx \Big] \notag \\
& \hspace{4cm}+ \mathbb{E}\Big[ {\bf 1}_{B} \int_{\R^d} \beta(u_0(x))\psi(0,x)\,dx\Big] \ge 0.\label{viscous-measure-inequality-1}
\end{align}
Since \eqref{viscous-measure-inequality-1} holds for every $B\in \mathcal{F}_T$, we conclude that $\mathbb{P}$-a.s., inequality \eqref{eq: generalised entropy inequality} holds true as well. In other words,
$u(t,x,\alpha)$ is a generalized entropy solution to the problem \eqref{eq:levy_con_laws}. 
\subsection{Proof of Theorem \ref{thm:uniqueness}}
To prove uniqueness of generalized entropy solutions, we follow the same argumentations as in \cite{BisKarlMaj}. 
Let $\rho$ and $\varrho$ be the standard nonnegative 
mollifiers on $\R$ and $\R^d$ respectively such that 
$\supp(\rho) \subset [-1,0]$ and $\supp(\varrho) = B_1(0)$, where $B_1(0)$ denotes the bounded ball of radius $1$ around $0$ in $\R^d$.  
We define  $\rho_{\delta_0}(r) = \frac{1}{\delta_0}\rho(\frac{r}{\delta_0})$ 
and $\varrho_{\delta}(x) = \frac{1}{\delta^d}\varrho(\frac{x}{\delta})$, 
where $\delta$ and $\delta_0$ are two positive constants. Given a  nonnegative test 
function $\psi\in C_c^{1,2}([0,\infty)\times \rd)$ and two 
positive constants $\delta$ and $ \delta_0 $, we define 
\begin{align}
\phi_{\delta,\delta_0}(t,x, s,y) = \rho_{\delta_0}(t-s) 
\varrho_{\delta}(x-y) \psi(s,y).\notag
\end{align}
Let $\beta:\R \rightarrow \R$ be a $C^\infty$ function satisfying 
\begin{align*}
      \beta(0) = 0,\quad \beta(-r)= \beta(r),\quad 
      \beta^\prime(-r) = -\beta^\prime(r),\quad \beta^{\prime\prime} \ge 0,
\end{align*} 
and 
\begin{align*}
	\beta^\prime(r)=\begin{cases} -1\quad \text{when} ~ r\le -1,\\
                               \in [-1,1] \quad\text{when}~ |r|<1,\\
                               +1 \quad \text{when} ~ r\ge 1.
                 \end{cases}
\end{align*} 
For any $\vartheta > 0$, define $\beta_\vartheta:\R \rightarrow \R$ by 
$\beta_\vartheta(r) = \vartheta \beta(\frac{r}{\vartheta})$. 
Then
\begin{align*}
 |r|-M_1\vartheta \le \beta_\vartheta(r) \le |r|\quad 
 \text{and} \quad |\beta_\vartheta^{\prime\prime}(r)| 
 \le \frac{M_2}{\vartheta} {\bf 1}_{|r|\le \vartheta},
\end{align*} 
where $M_1 = \sup_{|r|\le 1}\big | |r|-\beta(r)\big |$ and 
$M_2 = \sup_{|r|\le 1}|\beta^{\prime\prime} (r)|$. For $\beta= \beta_\vartheta$~ we define 
\begin{align*}
\mathcal{F}^{\beta_{\vartheta}}(a,b)=\int_{b}^a \beta_{\vartheta}^\prime(\sigma-b)f^\prime(\sigma)\,d(\sigma). 
\end{align*} 
Let $v(t,x,\alpha)$ be a generalized entropy 
solution of \eqref{eq:levy_con_laws}. Moreover, let $\varsigma$ be the standard symmetric 
nonnegative mollifier on $\R$ with support in $[-1,1]$ 
and $\varsigma_l(r)= \frac{1}{l} \varsigma(\frac{r}{l})$ 
for $l > 0$. Given $k\in \R$, the function 
$\beta_{\vartheta}(\cdot-k)$ is a smooth convex function 
and $(\beta_{\vartheta}(\cdot-k), \mathcal{F}^{\beta_{\vartheta}}(\cdot, k))$ is a 
convex entropy pair. Consider the entropy inequality 
for $v(t,x,\alpha)$, based on the 
entropy pair $(\beta_{\vartheta}(\cdot-k), \mathcal{F}^{\beta_{\vartheta}}(\cdot, k))$, and 
then multiply by $\varsigma_l(u_\eps(s,y)-k)$, integrate with 
respect to $ s, y, k$ and take the expectation. The result is
\begin{align}
0\le  & \mathbb{E} \Big[\int_{\Pi_T}\int_{\R^d}\int_{\R} \beta_{\vartheta}(v_0(x)-k)
\phi_{\delta,\delta_0}(0,x,s,y) \varsigma_l(u_\eps(s,y)-k)\,dk \,dx\,dy\,ds\Big] \notag \\
&  + \mathbb{E}  \Big[\int_{\Pi_T^2} \int_{0}^1 \int_{\R} \beta_{\vartheta}(v(t,x,\alpha)-k)\partial_t \phi_{\delta,\delta_0}(t,x,s,y)
\varsigma_l(u_\eps(s,y)-k)\,dk\, d\alpha \,dx\,dt\,dy\,ds \Big]\notag \\ 
& + \mathbb{E}  \Big[ \int_{\Pi_T^2}\int_{\R}\int_{\mathbf{E}} 
\int_{0}^1\Big(\beta_{\vartheta} \big(v(t,x,\alpha) 
+\eta(v(t,x,\alpha);z)-k\big) -\beta_{\vartheta}(v(t,x,\alpha)-k)\Big) \notag\\
&\hspace{4cm} \times \phi_{\delta,\delta_0}\,dx\,d\alpha \,\tilde{N}(dz,dt) \varsigma_l(u_\eps(s,y)-k)\,dk \,dy\,ds \Big] \notag\\
& + \mathbb{E}  \Big[\int_{\Pi_T^2}\int_{\mathbf{E}}
\int_{\R} \int_{0}^1 \Big(\beta_{\vartheta} \big(v(t,x,\alpha) +\eta(v(t,x,\alpha);z)-k\big) -\beta_{\vartheta}(v(t,x,\alpha)-k) \notag \\
 & \hspace{3cm}-\eta(v(t,x,\alpha);z) \beta_{\vartheta}^{\prime}(v(t,x,\alpha)-k)\Big)\phi_{\delta,\delta_0}\varsigma_l(u_\eps(s,y)-k)\,d\alpha\,dk\,dx\,m(dz)\,dt\,dy\,ds\Big]\notag \\
& + \mathbb{E} \Big[\int_{\Pi_T^2}\int_{0}^1 \int_{\R} \mathcal{F}^{\beta_{\vartheta}}(v(t,x,\alpha),k)\vec{v}(t,x) \cdot \grad_x\phi_{\delta,\delta_0}(t,x,s,y)
\varsigma_l(u_\eps(s,y)-k)\,dk\,d\alpha\,dx\,dt\,dy\,ds\Big] \notag \\
& =:  I_1 + I_2 + I_3 +I_4 + I_5.\label{stochas_entropy_1}
\end{align}
Since $u_\eps(s,y)$ is a viscous solution to the problem \eqref{eq:levy_con_laws-viscous}, one has
\begin{align}
0  & \le \mathbb{E} \Big[\int_{\Pi_T}\int_{\R^d} \int_{0}^1\int_{\R} \beta_{\vartheta}(u_\eps(0,y)-k)\phi_{\delta,\delta_0}(t,x,0,y) 
\varsigma_l(v(t,x,\alpha)-k)dk\,d\alpha dx\,dydt\Big] \notag \\
 &  + \mathbb{E} \Big[\int_{\Pi_T^2} \int_{0}^1 \int_{\R} 
 \beta_{\vartheta}(u_\eps(s,y)-k)\partial_s \phi_{\delta,\delta_0}(t,x,s,y)
 \varsigma_l(v(t,x,\alpha)-k)\,dk\, d\alpha \,dy\,ds\,dx\,dt\Big]\notag \\ 
 & + \mathbb{E} \Big[\int_{\Pi_T^2}\int_{\mathbf{E}} \int_{\R} 
 \int_{0}^1\Big(\beta_{\vartheta} \big(u_\eps(s,y) +\eta_\eps(y,u_\eps(s,y);z)-k\big)-\beta_{\vartheta}(u_\eps(s,y)-k)\Big) \notag \\
 & \hspace{4cm} \times \phi_{\delta,\delta_0}(t,x,s,y)\varsigma_l(v(t,x,\alpha)-k)
 \,dy\,d\alpha \,dk\,\tilde{N}(dz,ds)\,dx\,dt \Big]\notag\\
 & + \mathbb{E} \Big[\int_{\Pi_T^2}\int_{\mathbf{E}} \int_{\R}\int_{0}^1 
 \Big(\beta_{\vartheta} \big(u_\eps(s,y) +\eta_{\eps}(y,(u_\eps(s,y);z)-k\big)  -\eta_{\eps}(y,u_\eps(s,y);z) \beta_{\vartheta}^{\prime}(u_\eps(s,y)-k)\notag \\
 & \hspace{4cm} -\beta_{\vartheta}(u_\eps(s,y)-k)\Big) \phi_{\delta,\delta_0}(t,x;s,y)\varsigma_l(v(t,x,\alpha)-k)\,d\alpha\,dk\,dy\,m(dz)\,ds\,dx\,dt \Big]\notag \\
 & + \mathbb{E} \Big[\int_{\Pi_T^2}\int_{0}^1 \int_{\R}  \mathcal{F}^{\beta_{\vartheta}}(u_\eps(s,y),k)\vec{v}(s,y)  \cdot \grad_y\varrho_\delta(x-y) \psi(s,y)
 \rho_{\delta_0}(t-s ) \varsigma_l(v(t,x,\alpha)-k)\,dk\,d\alpha\,dx\,dt\,dy\,ds\Big] \notag \\
 & + \mathbb{E}\Big[\int_{\Pi_T^2}\int_{0}^1\int_{\R}\mathcal{F}^{\beta_{\vartheta}}(u_\eps(s,y),k) \vec{v}(s,y) \cdot \grad_y \psi(s,y) \varrho_\delta(x-y) 
  \rho_{\delta_0}(t-s )\varsigma_l(v(t,x,\alpha)-k)\,dk \,d\alpha\,dx\,dt\,dy\,ds\Big] \notag \\
  & - \eps \mathbb{E}\Big[ \int_{\Pi_T^2} \int_{0}^1\int_{\R} \beta_{\vartheta}^\prime(u_\eps(s,y)-k)\grad_y u_\eps(s,y)\cdot \grad_y  \phi_{\delta,\delta_0}(t,x,s,y)
  \varsigma_l(v(t,x,\alpha)-k)\,dk\, d\alpha \,dy\,ds\,dx\,dt\Big]  \notag \\ 
 &=:  J_1 + J_2 + J_3 +J_4 + J_5 + J_6 + J_7. \label{stochas_entropy_2}
\end{align}
We now add \eqref{stochas_entropy_1} and \eqref{stochas_entropy_2}, and compute limits with respect to the various parameters involved. In \cite{BisKarlMaj}, 
convergence of the terms $I_i (i=1,2,3,4)$ and $J_j (j=1,2,3,4,7)$ has been studied in details. Therefore, we only study the terms involving flux function namely 
the terms $I_5$, $J_5$ and $J_6$ in details. 

We first consider the term $I_5 + J_5$ and prove the following lemma.
\begin{lem}
There holds
\begin{align*} 
 \limsup_{\delta \downarrow 0, \, \, \vartheta \downarrow 0,\,\eps \downarrow 0\, l \downarrow 0\,\delta_0\downarrow 0} \big|I_5 + J_5\big| = 0. 
 \end{align*}
 \end{lem} 
\begin{proof}
Note that  
\begin{align}
& \Big| I_5 -\mathbb{E}\Big[\int_{\Pi_T}\int_{\R^d}\int_{0}^1  \int_{\R} \mathcal{F}^{\beta_\vartheta}(v(s,x,\alpha),k) \vec{v}(s,x)
\cdot\grad_x \varrho_\delta(x-y)\psi(s,y)\varsigma_l(u_\eps(s,y)-k)\,dk\,d\alpha\,dx\,dy\,ds\Big] \Big| \notag \\
& \le  \mathbb{E}\Big[\int_{\Pi_T^2}
\int_{0}^1  \int_{\R}  \big|\mathcal{F}^{\beta_\vartheta}(v(t,x,\alpha),k)
-\mathcal{F}^{\beta_\vartheta}(v(s,x,\alpha),k) \big||\vec{v}(t,x)| |\grad_x \varrho_\delta(x-y)| \psi(s,y)\notag \\
& \hspace{3cm} \times  \rho_{\delta_0}(t-s)\varsigma_l(u_\eps(s,y)-k)
\,dk\,d\alpha\,dx\,dy\,dt\,ds\Big] \notag \\
& \quad + \Big| \mathbb{E}\Big[\int_{\Pi_T}\int_{\R^d} 
\int_{0}^1  \int_{\R} \mathcal{F}^{\beta_\vartheta}(v(s,x,\alpha),k)\vec{v}(s,x)\cdot \grad_x \varrho_\delta(x-y)\,\psi(s,y)
 \Big( 1- \int_{t=0}^T \rho_{\delta_0}(t-s)\,dt \Big)\notag \\
 & \hspace{5cm} \times \varsigma_l(u_\eps(s,y)-k)\,dk\,d\alpha\,dx\,dy\,ds\Big] \Big| \notag \\
 &  \quad +   \mathbb{E}\Big[\int_{\Pi_T^2}
\int_{0}^1  \int_{\R} | \mathcal{F}^{\beta_\vartheta}(v(s,x,\alpha),k)||\vec{v}(s,x)-\vec{v}(t,x)|| \grad_x \varrho_\delta(x-y)|\,\psi(s,y)
 \notag \\
 & \hspace{5cm} \times  \rho_{\delta_0}(t-s)
 \varsigma_l(u_\eps(s,y)-k)\,dk\,d\alpha\,dx\,dy\,ds\,dt\Big] \Big| \notag \\
 & \le \mathbb{E}\Big[\int_{s=\delta_0}^T \int_{\Pi_T}\int_{\R^d}
 \int_{0}^1  \int_{\R}\big|\mathcal{F}^{\beta_\vartheta}(v(t,x,\alpha),k)-\mathcal{F}^{\beta_\vartheta}(v(s,x,\alpha),k) \big||\vec{v}(t,x)| |\grad_x \varrho_\delta(x-y)|\psi(s,y) 
 \notag \\
&  \hspace{5cm} \times \rho_{\delta_0}(t-s) \varsigma_l(u_\eps(s,y)-k)\,dk\,d\alpha\,dx\,dy\,dt\,ds\Big] + \mathcal{O}(\delta_0)\notag \\
& + C\mathbb{E}\Big[\int_{s=0}^{\delta_0} \int_{\R^d}\int_{\R^d}
\int_{0}^1  \int_{\R} \big| \mathcal{F}^{\beta_\vartheta}(v(s,x,\alpha),k) \vec{v}(s,x)\cdot 
\grad_x \varrho_\delta(x-y)\big|\psi(s,y)\varsigma_l(u_\eps(s,y)-k)\,dk\,d\alpha\,dx\,dy\,ds\Big] \notag\\
&\qquad  (\text{we have used the fact that}~  \int_{0}^T \rho_{\delta_0}(t-s)\,dt \le 1, 
\,\text{equality holds if}\, s\ge \delta_0).
\notag\\
 & \le  C \mathbb{E}\Big[\int_{s=\delta_0}^T \int_{\Pi_T}
 \int_{\R^d} \int_{0}^1 |\grad_x \varrho_\delta(x-y)|  \big|v(t,x,\alpha)-
 v(s,x,\alpha)\big|\psi(s,y) \rho_{\delta_0}(t-s)\,d\alpha\,dx\,dy\,dt\,ds\Big] + \mathcal{O}(\delta_0)\notag  \\
 &\qquad  (\text{we have used the Lipschitz 
 continuity of}\, \mathcal{F}^{\beta_\vartheta}(\cdot, k)\, \text{in above} )\notag\\
 & \le  C \Big( \mathbb{E}\Big[ \int_{s=\delta_0}^T \int_{\Pi_T} 
\int_{0}^1    |v(t,x,\alpha)
-v(s,x,\alpha)|^2 \rho_{\delta_0}(t-s)
d\alpha dx\,dt\,ds\Big] \Big)^\frac{1}{2} + \mathcal{O}(\delta_0) \notag \\
&  \le  C \Big( \mathbb{E} \Big[\int_{r=0}^1 \int_{\Pi_T} \int_{0}^1 
|v(t+ \delta_0\,r ,x,\alpha)-v(t,x,\alpha)|^2 
\rho(-r)\,d\alpha \,dt\,dx\,dr\Big]\Big)^\frac{1}{2} + \mathcal{O}(\delta_0) \notag.
\end{align}
In the above, we have used the notation $\mathcal{O}(\delta_0)$ to denote quantities that depend on $\delta_0$ and are bounded above by $C\delta_0$.
Note that, $ \displaystyle \underset{\delta_0\downarrow 0 }\lim\,\int_{\Pi_T} \int_{0}^1 |v(t+\delta_0 r,x,\alpha)-v(t,x,\alpha)|^2\,
\,d\alpha\,dx\,dt \rightarrow 0$ almost surely for all $r\in [0,1]$. 
Therefore, by the bounded convergence theorem,
\begin{align*}
\lim_{\delta_0\downarrow 0}\mathbb{E}\Big[\int_{r=0}^1\int_{\Pi_T} \int_{0}^1  
|v(t+\delta_0 r,x,\alpha)-v(t,x,\alpha)|^2\rho(-r)\,d\alpha\,dx\,dt\,dr\Big]=0.
\end{align*} 
 Since $\grad_y \varrho_\delta(x-y)= - \grad_x \varrho_\delta(x-y)$, we see that
\begin{align*}
& \Big|I_5 + \mathbb{E}\Big[\int_{\Pi_T}\int_{\R^d} 
\int_{0}^1  \int_{\R} \mathcal{F}^{\beta_\vartheta}(v(s,x,\alpha),u_{\eps}(s,y)-k)\vec{v}(s,x)\cdot \grad_y \varrho_\delta(x-y)
\,\psi(s,y)\varsigma_l(k)\,dk\,d\alpha\,dx\,dy\,ds\Big]\Big| \\
& \le A(\delta_0) + \mathcal{O}(\delta_0),
\end{align*} 
for some $A(\delta_0)$, with the property that $A(\delta_0)\goto 0$ as $\delta_0 \goto 0$.
In a similar manner, one has
\begin{align*}
&\Big|J_5 -\mathbb{E} \Big[\int_{\Pi_T} \int_{\R^d} \int_{\R}\int_{0}^1 
\mathcal{F}^{\beta_\vartheta}(u_\eps(s,y),v(s,x,\alpha)-k)\vec{v}(s,y)\cdot
\grad_y \varrho_\delta(x-y)
\, \psi(s,y)\varsigma_l(k) \,d\alpha\,dk\,dx\,dy\,ds\Big]\Big| \\
& \le B(\delta_0) + \mathcal{O}(\delta_0),
\end{align*}
where $B(\delta_0)$ is a quantity satisfying $B(\delta_0)\goto 0$ as $\delta_0 \goto 0$.
Note that, since $\text{div}_x \vec{v}(t,x)=0$ for all $(t,x)\in \Pi_T$, and $\grad_y \varrho_\delta(x-y)= - \grad_x \varrho_\delta(x-y)$, integration by parts formula yields 
\begin{align*}
 & \mathbb{E}\Big[\int_{\Pi_T} \int_{\R^d} \int_{\R}\int_{0}^1 \mathcal{F}^{\beta_{\vartheta}}(u_\eps(s,y),v(s,y,\alpha)-k)[\vec{v}(s,x)-\vec{v}(s,y)]\cdot \grad_y \varrho_\delta(x-y) \notag \\
 & \hspace{3cm} \times \psi(s,y)\varsigma_l(k)\,dk\,d\alpha\,dx\,dy\,ds\Big]=0.
\end{align*}
Hence, we have
\begin{align}
\big|I_5 + J_5\big| 
& \le \Bigg|\mathbb{E} \Big[\int_{\Pi_T}\int_{\R^d} \int_{\R}\int_{0}^1 
\Big( \mathcal{F}^{\beta_\vartheta}(v(s,x,\alpha),u_\eps(s,y)-k)-  \mathcal{F}^{\beta_\vartheta}(u_\eps(s,y),v(s,x,\alpha)-k)\Big)  \notag\\
&\hspace{4cm} \times \vec{v}(s,x) \cdot \grad_x \varrho_\delta(x-y) \psi(s,y) 
\varsigma_l(k)\,d\alpha\,dk\,dx\,ds\,dy\Big] \notag \\
& \quad + \mathbb{E} \Big[\int_{\Pi_T}\int_{\R^d} \int_{\R}\int_{0}^1 
\Big( \mathcal{F}^{\beta_\vartheta}(u_\eps(s,y), v(s,x,\alpha)-k)-  \mathcal{F}^{\beta_{\vartheta}}(u_\eps(s,y),v(s,y,\alpha)-k)\Big)  \notag\\
&\hspace{2cm} \times \big(\vec{v}(s,x)-\vec{v}(s,y)\big) \cdot \grad_x \varrho_\delta(x-y) \psi(s,y) 
\varsigma_l(k)\,d\alpha\,dk\,dx\,ds\,dy\Big]\Bigg| \notag \\
& \hspace{3cm} + A(\delta_0)  + B(\delta_0)  + \mathcal{O}(\delta_0). \notag 
\end{align}
Define $\displaystyle \mathcal{F}(a,b)=\text{sign}(a-b)\big(f(a)-f(b)\big)$.
Then, $\mathcal{F}$ is symmetric (\textit{i.e.}, $\mathcal{F}(a,b)= \mathcal{F}(b,a)$) and Lipschitz continuous in both of its variables. Moreover, 
\begin{align}
 \big|\mathcal{F}^{\beta_\vartheta}(a,b)-\mathcal{F}(a,b)\big|\le \vartheta c_f.\label{error:flux-}
\end{align}
Therefore, one has
\begin{align*}
 \Big|I_5 + J_5\Big| \le C(c_f, v, \psi) \delta + C(c_f, V, \psi) \frac{\vartheta}{\delta} + C(c_f, V, \psi) \frac{l}{\delta} + A(\delta_0)  + B(\delta_0)  + \mathcal{O}(\delta_0) 
\end{align*}
and hence 
\begin{align*}
 \limsup_{\delta \downarrow 0, \, \, \vartheta \downarrow 0,\,\eps \downarrow 0\, l \downarrow 0\,\delta_0\downarrow 0} \big|I_5 + J_5\big| = 0. 
\end{align*}
\end{proof}

\begin{lem}
It holds that
\begin{align*}
J_6  &\underset{\delta_0 \goto 0}{\rightarrow} \mathbb{E}\Big[\int_{\Pi_T}\int_{\R^d}\int_{0}^1 
\int_{\R} \mathcal{F}^{\beta_\vartheta}(u_\eps(s,y),k)\vec{v}(s,y)\cdot\grad_y \psi(s,y) \varrho_\delta(x-y) 
 \varsigma_l(v(s,x,\alpha)-k)\,dk \,d\alpha\,dx\,dy\,ds\Big] \\
&\underset{l \goto 0}{\rightarrow}  
\mathbb{E} \Big[\int_{\Pi_T}\int_{\R^d} \int_{0}^1\mathcal{F}^{\beta_\vartheta}(u_\eps(s,y),v(s,x,\alpha))\vec{v}(s,y)\cdot \grad_y \psi(s,y)
\, \varrho_\delta(x-y)\,d\alpha\,dx\,dy\,ds\Big] \\
 &\underset{\eps \goto 0}{\rightarrow}  \mathbb{E} \Big[\int_{\Pi_T}\int_{\R^d} 
 \int_{0}^1 \int_{0}^1 \mathcal{F}^{\beta_\vartheta}(u(s,y,\gamma),v(s,x,\alpha)) \vec{v}(s,y)
 \cdot \grad_y \psi(s,y)\varrho_\delta(x-y) \,d\gamma\,d\alpha\,dx\,dy\,ds\Big]\\
& \underset{\vartheta \goto 0}{\rightarrow} 
\mathbb{E} \Big[\int_{\Pi_T}\int_{\R^d} 
 \int_{0}^1 \int_{0}^1  \mathcal{F}(u(s,y,\gamma),v(s,x,\alpha)) \vec{v}(s,y)
 \cdot \grad_y \psi(s,y)\varrho_\delta(x-y) \,d\gamma\,d\alpha\,dx\,dy\,ds\Big]\\
& \underset{\delta \goto 0}{\rightarrow} \mathbb{E} \Big[\int_{\Pi_T} \int_{0}^1 \int_{0}^1 
\mathcal{F}(u(s,y,\gamma),v(s,y,\alpha)) \cdot \grad_y \psi(s,y)
 \,d\gamma\,d\alpha\,dy\,ds\Big].
 \end{align*}
\end{lem}

\begin{proof} The proof is divided into five steps.
\vspace{.1cm}

\textbf{Step 1:} We will justify the $\delta_0\to 0$ limit.  Define 
\begin{align}
\mathcal{B}_1& :=\Big| J_6
- \mathbb{E}\Big[\int_{\Pi_T}\int_{\R^d} \int_{\R} \int_{0}^1  
\mathcal{F}^{\beta_\vartheta}(u_\eps(s,y),k)\vec{v}(s,y)\cdot \grad_y \psi(s,y) \varrho_{\delta}(x-y)
 \varsigma_l(v(s,x,\alpha)-k) d\alpha\,dk \,dy\,dx\,ds \Big] \Big| \notag \\
 & = \Big| \mathbb{E} \Big[\int_{\Pi_T^2}\int_{\R} \int_{0}^1  \Big(\mathcal{F}^{\beta_\vartheta}(u_{\eps}(s,y), v(t,x,\alpha)-k)
 - \mathcal{F}^{\beta_\vartheta}(u_{\eps}(s,y), v(s,x,\alpha)-k)\Big) \vec{v}(s,y)\cdot  \grad_y\psi(s,y)\notag \\
 & \hspace{5cm} \times  \rho_{\delta_0}(t-s) \varrho_{\delta}(x-y) \varsigma_l(k)\,d\alpha\,dk\,dy\,ds\,dx\,dt\Big]\notag \\
& \quad- \mathbb{E} \Big[\int_{\Pi_T}\int_{\R^d} \int_{\R} \int_{0}^1  
\mathcal{F}^{\beta_\vartheta}(u_\eps(s,y),v(s,x,\alpha)-k)\vec{v}(s,y)\cdot \grad_y \psi(s,y)\,\varrho_{\delta}(x-y)  \notag \\
& \hspace{4cm} \times\Big( 1- \int_{0}^T \rho_{\delta_0}(t-s)\,dt \Big) \varsigma_l(k) d\alpha\,dk \,dy\,dx\,ds \Big]  \Big| \notag \\
 &\le C \mathbb{E}\Big[\int_{\Pi_T^2}\int_{0}^1 |\grad_y\psi(s,y)| \,\rho_{\delta_0}(t-s)
 \varrho_{\delta}(x-y) |v(s,x,\alpha)-v(t,x,\alpha)| \,d\alpha\,dy\,ds\,dx\,dt\Big]\notag \\
&\quad + C\mathbb{E} \Big[\int_{0}^{\delta_0} \int_{\R^d}\int_{\R^d} \int_{\R} 
\int_{0}^1 \big| \mathcal{F}^{\beta_\vartheta}(u_\eps(s,y),k)||\grad_y \psi(s,y)\big|
\,\varrho_{\delta}(x-y) \varsigma_l(v(s,x,\alpha)-k) 
d\alpha\,dk \,dy\,dx\,ds\Big]\notag\\
&\le C \mathbb{E} \Big[\int_{\delta_0}^T \int_{\R^d}\int_{0}^T\int_{0}^1 |v(s,x,\alpha)-v(t,x,\alpha)| 
\rho_{\delta_0}(t-s)\,d\alpha\,dt\,dx\,ds\Big]
+ \mathcal{O}(\delta_0)\notag\\
&\le C \Big(\mathbb{E} \Big[\int_{\delta_0}^T
\int_{0}^T \int_{\R^d} \int_{0}^1 |v(s,x,\alpha)-v(t,x,\alpha)|^2  
\rho_{\delta_0}(t-s) d\alpha\, dx\,dt\,ds\Big] \Big)^\frac{1}{2} + \mathcal{O}(\delta_0)
\longrightarrow 0 \,\,\text{as}\,\, \delta_0 \goto 0, \notag
\end{align}
and therefore the first step follows.

\textbf{Step 2:} We will justify the $l\to 0$ limit. Let
\begin{align*}
\mathcal{B}_2 :=& \mathbb{E} \Big[\int_{\Pi_T}\int_{\R^d} 
\int_{0}^1  \int_{\R} \mathcal{F}^{\beta_\vartheta}(u_\eps(s,y),k)\vec{v}(s,y)\cdot \grad_y \psi(s,y) 
\varrho_\delta(x-y)  \varsigma_l(v(s,x,\alpha)-k)\,dk\,d\alpha\,dx\,dy\,ds\Big] \\
 & \quad-\mathbb{E}  \Big[\int_{\Pi_T}\int_{\R^d} \int_{0}^1  
 \mathcal{F}^{\beta_\vartheta}(u_\eps(s,y),v(s,x,\alpha))\vec{v}(s,y)\cdot \grad_y \psi(s,y)\,\varrho_\delta(x-y)\,d\alpha\,dx\,dy\,ds\Big] \\
&=\mathbb{E}  \Big[\int_{\Pi_T}\int_{\R^d} \int_{0}^1 
\int_{\R} \Big(\mathcal{F}^{\beta_\vartheta}(u_\eps(s,y),k) 
- \mathcal{F}^{\beta_\vartheta}(u_\eps(s,y),v(s,x,\alpha))\Big)\vec{v}(s,y)\cdot  \grad_y \psi(s,y) \\
&\hspace{5cm}\times \varrho_\delta(x-y) \varsigma_l(v(s,x,\alpha)-k)\,dk\,d\alpha\,dx\,dy\,ds\Big].
\end{align*}
By using the boundedness of $\vec{v}$ and Lipschitz property of $\mathcal{F}^{\beta_\vartheta}$, we arrive at 
\begin{align*}
|\mathcal{B}_2|& \le C l \int_{\Pi_T} |\grad_y \psi(s,y)|\,dy\,ds \goto 0\quad \text{as $l\goto 0$.}
\end{align*}

\textbf{Step 3:} We now justify the passage to the limit $\eps \goto 0$. 
Let 
$$ G_x(s,y,\omega,\xi)=\int_{\R^d}\int_{0}^1 \mathcal{F}^{\beta_\vartheta}(\xi,v(s,x,\alpha))\vec{v}(s,y)\cdot \grad_y \psi(s,y)\,
\varrho_\delta(x-y)\,d\alpha\,dx. $$
Then $G_x(s,y,\omega,\xi)$ is a  Carath\'{e}odory function for every $x\in \R^d$ and 
$\{ G_x(s,y,\omega, u_{\eps_n}(s,y))\}_n$ is bounded in $L^2((\Theta,  \Sigma, \mu);\R)$ and uniformly integrable. Thus, by Proposition \ref{prop:young-measure}  we conclude that
\begin{align}
& \lim_{\eps\rightarrow 0} \mathbb{E} \Big[\int_{\R^d}
\int_{\Pi_T}\int_{0}^1 \mathcal{F}^{\beta_\vartheta}(u_{\eps}(s,y),v(s,x,\alpha))\vec{v}(s,y)\cdot 
\grad_y \psi(s,y)\,\varrho_\delta(x-y)\,d\alpha\,dx\,ds\,dy\Big]\notag \\
& = \mathbb{E}\Big[\int_{\R^d}\int_{\Pi_T}\int_{0}^1 
\int_{0}^1 \mathcal{F}^{\beta_\vartheta}(u(s,y,\gamma),v(s,x,\alpha))\vec{v}(s,y)\cdot \grad_y \psi(s,y)\,
 \varrho_\delta(x-y)\,d\gamma d\alpha dx ds dy\Big].\notag 
\end{align} 

\textbf{Step 4:} Justification of the limit $\vartheta \goto 0$. Let 
\begin{align*}
 \mathcal{B}_3:& = \mathbb{E}\Big[\int_{\R^d}\int_{\Pi_T}\int_{0}^1 
\int_{0}^1 \Big(\mathcal{F}^{\beta_\vartheta}(u(s,y,\gamma),v(s,x,\alpha))-\mathcal{F}(u(s,y,\gamma),v(s,x,\alpha))\Big) \vec{v}(s,y)\cdot \grad_y \psi(s,y)\\
 & \hspace{3cm} \times \varrho_\delta(x-y)\,d\gamma d\alpha dx ds dy\Big].
\end{align*}
In view of \eqref{error:flux-} and the assumption \ref{A2}, we see that 
\begin{align*}
 |\mathcal{B}_3| \le C(V, c_f)\vartheta \int_{\Pi_T}|\grad_y \psi(s,y)|\,dy\,ds \goto 0\quad \text{as}\quad \vartheta \goto 0.
 \end{align*}

\textbf{Step 5:} Justification of the limit $\delta \goto 0$. 
\begin{align*}
 \mathcal{B}_4:=&\Big| \mathbb{E} \Big[\int_{\Pi_T}\int_{\R^d}
\int_{0}^1 \int_{0}^1  \mathcal{F}(u(s,y,\gamma),v(s,x,\alpha))\vec{v}(s,y)\cdot \grad_y \psi(s,y)\,
\varrho_\delta(x-y) \,d\gamma\,d\alpha\,dx\,dy\,ds \Big]\\
& \quad- \mathbb{E}\Big[\int_{\Pi_T}\int_{0}^1 \int_{0}^1  
\mathcal{F}(u(s,y,\gamma),v(s,y,\alpha))\vec{v}(s,y)\cdot \grad_y \psi(s,y)
\,d\gamma\,d\alpha\,dy\,ds\Big]\Big|\\
& \le C \mathbb{E} \Big[\int_{\Pi_T}\int_{\R^d}\int_{0}^1 
\int_{0}^1  \big|\mathcal{F}(u(s,y,\gamma),v(s,x,\alpha))-\mathcal{F}(u(s,y,\gamma),v(s,y,\alpha))\big|\\
& \hspace{5cm} \times
|\grad_y \psi(s,y)|\, \varrho_\delta(x-y) \,d\gamma\,d\alpha\,dx\,dy\,ds\Big].
\end{align*}
Since $\mathcal{F}$ is Lipschitz continuous in both of its variables, by Cauchy-Schwartz's inequality, we have 
\begin{align*}
 |\mathcal{B}_4| \le C\Big( E\Big[ \int_{\Pi_T}\int_{\R^d}\int_{0}^1
|v(s,y,\gamma)-v(s,y+\delta z,\gamma) |^2 \varrho(z) \,d\gamma\,dy\,dz\,ds\Big] \Big)^{\frac 12}  \goto 0 \quad \text{as} \,\, \delta \goto 0.
\end{align*}
This completes the proof.
\end{proof}

Following \cite{BisKarlMaj}, we arrive at 
\begin{lem}
 The following hold:
 \begin{align*}
  & \lim_{(\delta,\, \vartheta,\, \eps,\, l,\, \delta_0) \goto 0}(I_1 + J_1)
  = \mathbb{E}\Big[\int_{\R^d} | v_0(x)-u_0(x)|\psi(0,x)\,dx\Big]; \,\,\,\,\, \limsup_{(\eps,\,l,\,\delta_0)\rightarrow(0)}|J_7|=0. \\
   & \lim_{(\delta,\, \vartheta,\, \eps,\, l,\, \delta_0) \goto 0}(I_2 + J_2)= 
   \mathbb{E}\Big[\int_{\Pi_T}\int_{0}^1 \int_{0}^1  
|u(s,y,\gamma)-v(s,y,\alpha)| \partial_s \psi(s,y)\,d\gamma\,d\alpha\,dy\,ds\Big]. \\
& \lim_{(l,\,\delta_0) \goto 0} (I_4 + J_4) = \mathbb{E}\Big[ \int_{\Pi_T} \int_{\R^d}\int_{\mathbf{E}} \int_{0}^1 \int_0^1 \Big\{ \beta_{\vartheta}^{\prime\prime}\big( u_\eps(s,y)-v(s,x,\alpha) + 
\lambda \eta(u_\eps(s,y);z)\big) |\eta(u_\eps(s,y);z)|^2  \\
&\hspace{5.5cm} + \beta_{\vartheta}^{\prime\prime}\big(v(s,x,\alpha)-u_\eps(s,y) + \lambda \eta(v(s,x,\alpha);z)\big) |\eta(v(s,x,\alpha);z)|^2 \Big\} \\
& \hspace{6.5cm} \times (1-\lambda)  \psi(s,y)\varrho_\delta(x-y) \,d\alpha\,d\lambda\,m(dz)\,dy\,dy\,ds\Big].
 \end{align*}
\end{lem}
Let us consider the stochastic integrals. Note that $J_3=0$. In view of It\^{o}-L\'{e}vy formula, we see that
\begin{align}
I_3 &= \mathbb{E}\Big[\int_{\R}\int_{\Pi_{T}} 
J[\beta_{\vartheta}^{\prime},\phi_{\delta,\delta_0}](s;y,k)\Big(\int_{s-\delta_0}^s 
\varsigma_l(u_\eps(\sigma,y)-k)\,\text{div} (f(u_\eps(\sigma,y))\vec{v}(\sigma,y)) 
\,d\sigma\Big)\,ds\,dy\,dk\Big]\notag  \\
& \quad -\mathbb{E}\Big[\int_{\R}\int_{\Pi_{T}} J[\beta_{\vartheta}^{\prime},\phi_{\delta,\delta_0}](s;y,k)
\Big(\int_{s-\delta_0}^s \varsigma_l(u_\eps(\sigma,y)-k)
\, \eps  \Delta u_\eps(\sigma,y)\big) \,d\sigma\Big)\,ds\,dy\,dk\Big]\notag  \\
& 
\quad + \mathbb{E}\Big[\int_{\Pi_{T}}\int_{\R}\int_{r=s-\delta_0}^s
\int_{\R^d}\int_{\mathbf{E}} \int_{0}^1 \Big(\beta_{\vartheta}\big(v(r,x,\alpha)+ \eta(v(r,x,\alpha);z)-k\big)
-\beta_{\vartheta}(v(r,x,\alpha)-k)\Big) \notag \\
& \hspace{4cm}\times\Big( \varsigma_l(u_\eps(r,y)+\eta(u_\eps(r,y);z)-k)
-\varsigma_l(u_\eps(r,y)-k)\Big) \notag \\
& \hspace{5cm}\times\rho_{\delta_0}(r-s)\,\psi(s,y)\,\varrho_{\delta}(x-y) \,d\alpha \,m(dz)\,dx\,dr\,dk\,dy\,ds\Big] \notag \\
&\quad +\mathbb{E}\Big[\int_{\R}\int_{\Pi_{T}} 
J[\beta_{\vartheta},\phi_{\delta,\delta_0}](s;y,k)\Big\{ \int_{s-\delta_0}^s \int_{\mathbf{E}}\int_{0}^1(1-\lambda) |\eta(u_\eps(\sigma,y);z)|^2\notag\\
&\hspace{4cm}\times \varsigma_{l}^{\prime\prime}( u_\eps(\sigma,y)-k
+\lambda \eta( u_\eps(\sigma,y);z) )\,d\lambda\,m(dz)\,d\sigma \,\Big\}\,dy\,ds\,dk\Big] \notag \\
&\quad =: A_1^{l,\eps}(\delta,\delta_0) 
+ A_2^{l,\eps}(\delta,\delta_0)+  B^{\eps, l} + A_3^{l,\eps}(\delta,\delta_0), \notag
 \end{align}
 where 
 \begin{align*}
  J[\beta_{\vartheta}, \phi_{\delta,\delta_0}](s; y, k) :=&
\int_{\Pi_T}\int_{\mathbf{E}} \int_0^1 \Big(\beta_{\vartheta}\big(v(r,x,\alpha) +\eta(v(r,x,\alpha); z)-k\big)
 -\beta_{\vartheta}\big(v(r,x,\alpha)-k\big)\Big) \\
&\hspace{6cm}\times\phi_{\delta,\delta_0}(r,x,s,y) \,d\alpha\, \tilde{N}(dz,dr)\,dx. 
 \end{align*}
Thanks to the assumption \ref{A2}, by following the arguments as in the proof of \cite[Lemma 5.6]{BisKarlMaj}, we infer that 
 $A_1^{l,\eps}(\delta,\delta_0),\, A_2^{l,\eps}(\delta,\delta_0),\, A_3^{l,\eps}(\delta,\delta_0) \goto 0\, \text{as $\delta_0 \goto 0$}$, and 
 \begin{align*}
    \lim_{(l,\,\delta_0) \goto 0} B^{l,\eps}(\delta,\delta_0)
 =& \mathbb{E}\Big[\int_{\Pi_{T}}\int_{\R^d}
 \int_{\mathbf{E}} \int_{0}^1 \Big\{\beta_{\vartheta}\big(v(r,x,\alpha)+ \eta(v(r,x,\alpha);z)
  -u_\eps(r,y)-\eta(u_\eps;z)\big) \notag \\
  & \hspace{1.5cm} -\beta_{\vartheta}\big(v(r,x,\alpha)+ \eta(v(r,x,\alpha);z)-u_\eps(r,y)\big)  + \beta_{\vartheta}\big(v(r,x,\alpha)-u_\eps(r,y)\big) \notag \\
  &  \hspace{1cm} -\beta_{\vartheta}\big(v(r,x,\alpha)-u_\eps(r,y)-\eta(u_\eps(r,y);z)\big)
  \Big\}\psi(r,y)\,\varrho_{\delta}(x-y)\,d\alpha\,m(dz)\,dx\,dy\,dr\Big], \notag
 \end{align*}
 which yields
  \begin{align*}
& \lim_{l\goto 0}\lim_{\delta_0 \goto 0} \Big((I_3 +J_3)+ (I_4 + J_4)\Big) \notag \\
&= \mathbb{E} \Big[\int_{\Pi_T}\int_{\R^d}\Big( \int_{\mathbf{E}}
\int_{0}^1\int_{0}^1 b^2
(1-\theta)\beta_{\vartheta}^{\prime\prime}(a+\theta\,b)
\,d\theta\,d\alpha\,m(dz)\Big)\,\psi(t,y)\varrho_{\delta}(x-y)\,dx\,dy\,dt\Big],
\end{align*}
where $ a=v(t,x,\alpha)-u_\eps(t,y)$ and $b=\eta(v(t,x,\alpha);z)-\eta(u_\eps(t,y);z)$. In view of \ref{A3}, one has
(cf. proof of \cite[Lemma\,$5.11$]{BisKarlMaj}) $\displaystyle 
 b^2 \beta_\vartheta^{\prime\prime}(a+\theta\,b)\le 2 (1-\lambda^*)^{-2}\vartheta h_1^2(z)$,
and thus 
\begin{align*}
& \limsup_{\delta\goto 0,\,\vartheta \goto 0} \limsup_{\eps \rightarrow 0}\Big[ \lim_{l\goto 0}\lim_{\delta_0 \goto 0}\Big((I_3 +J_3)+ (I_4 + J_4)\Big)\Big]=0.
\end{align*} 
Finally, we add \eqref{stochas_entropy_1} and \eqref{stochas_entropy_2}, and pass to the limits $\delta_0 \goto 0$, $l\goto 0$, $\eps \goto 0$, $\vartheta \goto 0$ and $\delta \goto 0$ to
arrive at the following Kato inequality
\begin{align}
   & \mathbb{E} \Big[\int_{\R^d} | v_0(x)-u_0(x)|\psi(0,x)\,dx\Big] 
   + \mathbb{E} \Big[\int_{\Pi_T} \int_{0}^1 
   \int_{0}^1 |v(t,x,\alpha)-u(t,x,\gamma)|
   \partial_t \psi(t,x)\,d\alpha\,d\gamma \,dx\,dt\Big] \notag \\
   & \hspace{ 3cm}+ \mathbb{E} \Big[\int_{\Pi_T}\int_{0}^1 
   \int_{0}^1 \mathcal{F}\big(v(t,x,\alpha),u(t,x,\gamma)\big)\vec{v}(t,x)
  \cdot \grad_x \psi(t,x) \,d\alpha\,d\gamma\,dx\,dt \Big]\ge 0,\label{inq:Kato}
\end{align}
where $0\le \psi \in H^1( [0,\infty) \times \R^d)$ with compact support. One can choose special test function $\psi$ and $u_0=v_0$ in \eqref{inq:Kato} to conclude that 
$u(t,x,\gamma)=v(t,x,\alpha)$ for a.e. $(t,x)\in \Pi_T$ and a.e. $(\alpha,\gamma)\in (0,1)^2$\,(cf. proof of \cite[Theorem $2.2$]{BisKarlMaj}).
This finishes the proof.

\subsection{On Poisson random measure}
For the convenience of the reader, we recapitulate the basics of Poisson random measure. Let $\{\tau_n\}_{n\ge 1}$ be a sequence of independent exponential random variables with
parameter $\iota$ and $T_n =\sum_{i=1}^n \tau_i$.
 Then the process
 \begin{align*}
  N_t= \sum_{n\ge 1} {\bf 1}_{t\ge T_n}
 \end{align*}
 counts the number of random times $T_n$ which arise between $0$ and $t$. The jump times $T_1, T_2,\cdots$ form a random configuration 
 of points on $[0,\infty)$. This  counting procedure defines a measure on $[0,\infty)$ as follows: for any measurable set $A \subset 
  (0,\infty)$, set $N(\omega, A)= \# \{ i \ge 1: T_i(\omega) \in A\}$. Clearly, $N(\omega,\cdot)$ is a positive integer-valued 
  measure and for fixed $A$, $N(\cdot, A)$ is a Poisson random variable with parameter $\iota |A|$, where $|A|$ denotes the Lebesgue 
  measure of $A$. Also, if $A$ and $B$ are two disjoint sets then $N(\cdot,A)$ and $N(\cdot,B)$ are two independent random variables. This can be extended to a general 
setting. Let $ \bar{\mathbb{Z}_+}=\mathbb{Z}_+ \cup \{ + \infty\}$.
\begin{defi} Let $(\Theta,\mathcal{B},\rho)$ be a $\sigma$-finite measure space. A family of $\bar{\mathbb{Z}_+}$-valued random
variables 
$\{ N(B):B\in \mathcal{B}\}$ is called a Poisson random measure on $ \Theta$ with intensity measure $\rho$, if 
 \item[i)] For each $B$, $N(B)$ has a Poisson distribution with mean $\rho(B)$.
\item[ii)] If $ B_1, B_2,\cdots,B_m $ are disjoint, then $ N(B_1),N(B_2),....,N(B_m)$ are independent.
\item[iii)] For every $\omega \in \Omega$, $N(.,\omega)$ is a measure on $\Theta$.
\end{defi}
\noindent{\textbf{Construction of a Poisson random measure:}} Let $(\Theta,\mathcal{B},\rho)$ be a $\sigma$-finite measure space. We want to construct a Poisson
random measure $\{ N(B):B\in \mathcal{B}\}$ on $ \Theta$ with intensity measure $\rho$ on some probability space
$(\Omega,\mathcal{F},\mathbb{P})$. Assume that $\rho(\Theta)<\infty$. If $\rho=0$, then we choose $N(B)=0$. Assume that $ \rho(\Theta) >0 $. Then on
some probability space $(\Omega,\mathcal{F},\mathbb{P})$, one can construct a sequence $\{Z_n:n=1,2,3......\}$ of i.i.d random variables on $\Theta$ each having distribution $(\rho(\Theta))^{-1}\rho $ and a Poisson
random variable $Y$ with mean $\rho(\Theta)$ such that $Y$ and $ \{Z_n\}$ are independent. Define
\begin{equation}
 N(B)=
\begin{cases}
 0,&\text{if $ Y=0$} \notag \\
\sum_{j=1}^Y \chi_B(Z_j),&\text{if $ Y\geq 1$}. \notag
\end{cases}
\end{equation}
Then, $\{ N(B):B\in \mathcal{B}\}$ is called a Poisson random measure on $\Theta$ with intensity measure $\rho$, see \cite[Proposition $19.4$]{sato}.
Next we consider the case $\rho(\Theta)=\infty$. By $\sigma$-finiteness, there exist disjoint sets 
$\Theta_1,\Theta_2,....\in\mathcal{B}$ such that $\cup_{k=1}^{\infty} \Theta_k=\Theta$
and $\rho(\Theta_k)<\infty$, for each $k$. Define $\rho_k(B)=\rho(B\cap \Theta_k)$. Then, one can construct independent Poisson random measures $\{ N_k(B):B\in \mathcal{B}\}$, $ k=1,2,3,\cdots$, with intensity
measure $\rho_k$, defined on a probability space
$(\Omega,\mathcal{F},\mathbb{P})$. Then $$ N(B):= \sum_{k=1}^{\infty} N_k(B),\quad B\in \mathcal{B}$$ 
is a Poisson random measure on $\Theta$ with intensity measure $\rho$ (cf.~ \cite[Proposition $19.4$]{sato}).
\vspace{.1cm}

The construction of a Poisson random measure shows that it is a counting measure associated to a random sequence 
of points $X_n(\omega)$ in $\Theta$ such that
\begin{align*}
 N(\omega, B)= \sum_{n\ge 1} {\bf 1}_{B}(X_n(\omega)).
\end{align*} 
Let us give an example of a Poisson random measure on $[0,\infty)\times \R^d$.  Let $L= \{L_t\}_{t\ge 0}$ be a L\'{e}vy process taking values in $\R^d$ on a given filtered probability space 
$\big(\Omega, \mathbb{P}, \mathcal{F}, \{\mathcal{F}_t\}_{t\ge 0} \big)$. For $t\ge 0$ and $A\in \mathcal{B}(\R^d_0)$ , where 
$\R^d_0= \R^d \setminus \{0\}$, we define
\begin{align*}
 N \big( [0,t], A\big)&= \# \Big\{ 0\le s\le t: L_s -L_{s-}\in A \Big\} = \sum_{s\le t} {\bf 1}_{A} (\Delta L(s))
 \end{align*}
 where $\Delta L(s)= L_s -L_{s-}$. It counts the jumps $\Delta L(s)$ of the process $L$ of size in $A$ up to time $t$.
 Let $\nu(A)= \mathbb{E}\big[ N\big( [0,1], A\big)\big]$, the expected number of jumps of $L_t$ per unit time, whose size belongs to $A$. The L\'{e}vy measure $\nu(dz)$ may be infinite but
 satisfies $\int_{\R^d_0} ( |z|^2 \wedge 1) \nu(dz) < + \infty$. It is a Radon measure with a possible singularity at $z=0$ i.e., $\nu(dz)$
 restricted to each $\R^d \setminus B(0,r), r>0$ is a finite measure. One can show the following properties:
 \begin{itemize}
  \item[i).] $N\big( [0,t], A\big)$ is a random variable on $\big(\Omega, P, \mathcal{F}, \{\mathcal{F}_t\}_{t\ge 0} \big)$.
  \item[ii).] $t\mapsto N\big( [0,t], A\big) $ is a Poisson process with intensity $t\nu(A)$.
  \item[iii).] $N\big( [0,t], \o{}\big)=0$ and for any disjoint sets $A_1, A_2,\cdots, A_m$, the random variables $N\big( [0,t], A_1\big)$, $N\big( [0,t], A_2\big), \cdots, 
  N\big( [0,t], A_m\big)$ are independent. 
 \end{itemize} 
 The compensated Poisson random measure is defined by
 \begin{align*}
  \tilde{N}\big( [0,t], A\big)= N\big( [0,t], A\big)- t \nu(A).
 \end{align*}
 Next we define stochastic integral with respect to compensated Poisson random measure $\tilde{N}(dz,dt)= N(dz,dt)-\nu(dz)dt$, where $N(dz,dt)$ is a Poisson random measure and $\nu(dz)$
  is a L\'{e}vy measure. To do so, let us first define it so called for simple predictable functions. A simple predictable function  $f(s,z): \Omega \times [0,T]\times \R^d_0\goto \R$ is of the form 
  $$ f(s,z)= \sum_{i=1}^n \sum_{j=1}^m  \xi_{ij} {\bf 1}_{(\tau_i,\tau_{i+1}]}(t) {\bf 1}_{A_j}(z),$$
 	where $0=\tau_0 < \tau_1 < \cdots < \tau_{n+1}=T$ are stopping times, $\xi_{i1},\cdots, \xi_{im} \in \mathcal{F}_{\tau_i}$ for 
 	$i=1,\cdots,n$ with $\xi_{ij}$ are bounded for all $i,j$, and $A_1,\cdots A_m \in \mathcal{B}(\R^d_0)$ are disjoint sets with 
 	$\nu(A_1),\cdots \nu(A_m) < \infty$. For simple predictable function $f(s,z)$ of the above form, we define 
 		$$  I_t(f)=: \int_0^t \int_{\R^d_0} f(s,z)\tilde{N}(dz,ds)=  \sum_{i=1}^n \sum_{j=1}^m  \xi_{ij}  \tilde{N} \big((\tau_i \wedge t,\tau_{i+1} \wedge t], A_j \big).$$
 \begin{lem}
  Let $f: \Omega \times [0,T]\times \R^d_0\goto \R$ be a simple predictable function. Then $I_t(f)$ is a $L^2$- martingale and satisfies the isometry property 
  \begin{align}
   \mathbb{E}\Big[ \big| I_t(f)\big|^2\Big] = \mathbb{E}\Big[ \int_0^t \int_{\R^d_0} \big| f(s,z)\big|^2 \nu(dz)\,ds\Big].\label{eq:isometry}
  \end{align}
 \end{lem}
 In view of the isometry property, one can extend the integral to the closure of the space of simple predictable functions in $L^2\big(\Omega\times [0,T]\times \R^d_0\big)$ with respect to 
 $\sigma$-algebra $\mathcal{F}_T \otimes \mathcal{B}([0,T]) \otimes \mathcal{B}(\R^d_0)$ and product measure $\mathbb{P}\otimes dt\otimes \nu(dz)$. Note that above mentioned closure contains all 
 $\mathcal{P}_T \otimes \mathcal{B}(\R^d_0)$-measurable functions $f:\Omega \times [0,T]\times \R^d_0\goto \R$ such that 
 \begin{align}
  \mathbb{E}\Big[ \int_0^T \int_{\R^d_0} \big| f(s,z)\big|^2 \nu(dz)\,ds\Big]< + \infty.\label{property:l2}
 \end{align}
Thus, for any predictable functions satisfying \eqref{property:l2}, one can define the integral via limiting argument. One can also show that  $t\mapsto \int_0^t \int_{\R^d_0} f(s,z)\tilde{N}(dz,ds)$ is a 
martingale. Moreover \eqref{eq:isometry} holds.


\end{document}